\documentclass[12pt]{article}
\usepackage[utf8]{inputenc}
\usepackage[english]{babel}
\usepackage[margin=1in]{geometry}
\usepackage{color}
\usepackage{enumerate}
\usepackage{amsmath, amsthm, amssymb, amsfonts}
\usepackage{bbm}
\usepackage{mathtools}
\usepackage[all,cmtip]{xy}
\usepackage{todonotes}
\usepackage{hyperref}
\allowdisplaybreaks
\newtheorem{theorem}{Theorem}[section]
\newtheorem{corollary}[theorem]{Corollary}
\newtheorem{lemma}[theorem]{Lemma}
\newtheorem*{lma}{Lemma}
\newtheorem*{claim*}{Claim}

\newtheorem{proposition}[theorem]{Proposition}
\theoremstyle{definition}
\newtheorem{definition}[theorem]{Definition}
\newtheorem*{note}{Note}

\newtheorem{remark}{Remark}

\newcommand{\Z}{\mathbb{Z}}
\newcommand{\F}{\mathbb{F}}
\newcommand{\GL}{\operatorname{GL}}

\newcommand{\ed}{\operatorname{ed}}
\newcommand{\syl}{\operatorname{Syl}}

\newcommand{\trdeg}{\operatorname{trdeg}}

\newcommand{\mbf}{\mathbf}

\newcommand{\Id}{\text{Id}}

\newcommand{\divides}{\bigm|}

\newenvironment{theorem*}[2][Theorem]{\begin{trivlist}
\item[\hskip \labelsep {\bfseries #1}\hskip \labelsep {\bfseries #2}]}{\end{trivlist}}
\newenvironment{lemma*}[2][Lemma]{\begin{trivlist}
\item[\hskip \labelsep {\bfseries #1}\hskip \labelsep {\bfseries #2}]}{\end{trivlist}}
\newenvironment{corollary*}[2][Corollary]{\begin{trivlist}
\item[\hskip \labelsep {\bfseries #1}\hskip \labelsep {\bfseries #2}]}{\end{trivlist}}
\begin{document}
\title{\texorpdfstring{The essential $l$-dimension of finite groups of Lie type, $l \neq 2$}{The essential l-dimension of finite groups of Lie type, l neq 2}}
\author{Hannah Knight\thanks{This work was supported in part by NSF Grant Nos. DMS-1811846 and DMS-1944862 and NSF Award No. 2302822.}\\ Mathematics, UCLA, Los Angeles, CA, USA\\ {\color{blue}\href{mailto:hknight1@uci.edu}{hknight1@uci.edu}}} 
\date{}

\maketitle

\vspace{-1cm}

\begin{abstract}

In this paper, we compute the essential $l$-dimension of the finite groups of classical Lie type for odd primes $l$ not equal to the defining prime, specifically the linear groups, the symplectic groups, the orthogonal groups, and the unitary groups, and the simple factors in their Jordan-H\"{o}lder series. 

\end{abstract}



\section{Introduction}

The goal of this paper is to compute the essential $l$-dimension of the finite groups of classical Lie type, for odd primes $l$ not equal to the defining prime. Fix a field $k$. The essential dimension of a finite group $G$, denoted $\ed_k(G)$, is the smallest number of algebraically independent parameters needed to define a Galois $G$-algebra over any field extension $F/k$ (or equivalently $G\text{-torsors over }\text{Spec}F$). In other words, the essential dimension of a finite group $G$ is the supremum taken over all field extensions $F/k$ of the smallest number of algebraically independent parameters needed to define a Galois $G$-algebra over $F$.  The essential $l$-dimension of a finite group, denoted $\ed_k(G,l)$, is similar except that before taking the supremum, we allow finite extensions of $F$ of degree prime to $l$ and take the minimum of the number of parameters needed.  In other words, the essential $l$-dimension of a finite group is the supremum taken over all fields $F/k$ of the smallest number of algebraically independent parameters needed to define a Galois $G$-algebra over a field extension of $F$ of degree prime to $l$.  See Section \ref{edbackground} for more formal definitions. See also \cite{BR} and \cite{KM} for more detailed discussions. For a discussion of some interesting applications of essential dimension and essential $p$-dimension, see \cite{Reich}.

What is the essential dimension of the finite simple groups? This question is quite difficult to answer. A few results for small groups (not necessarily simple) have been proven. For example, it is known that $\ed_k(S_5) = 2$, $\ed_k(S_6) = 3$ for $k$ of characteristic not $2$ \cite{BF}, and $\ed_k(A_7) = \ed_k(S_7) = 4$ in characteristic $0$ \cite{Dun}.  It is also known that for $k$ a field of characteristic $0$ containing all roots of unity, $\ed_k(G) = 1$ if and only if $G$ is isomorphic to a cyclic group $\Z/n\Z$ or a dihedral group $D_m$ where $m$ is odd (\cite{BR}, Theorem 6.2). Various bounds have also been proven. See \cite{BR}, \cite{Mer}, \cite{Reich},\cite{Mor}, among others. For a nice summary of the results known in 2010, see \cite{Reich}.

We can find a lower bound to this question by considering the corresponding question for essential $p$-dimension. In my previous article (\cite{Kni}), I calculated the essential $p$-dimension of the split finite quasi-simple groups of classical Lie type at the defining prime $p$. In this article, we calculate the essential $l$-dimension of the groups at a prime $l$, where $l \neq 2$ and $l \neq p$ (where $p$ is the defining prime).

\begin{definition}
Let $v_l(n)$ denote the largest integer $i$ such that $l^i \divides n$. Let $\zeta_n$ denote a primitive $n$th root of unity. 
\end{definition}

\begin{theorem}\label{mainthm} 

\begin{enumerate}[(1)] Let $p$ be a prime, $q = p^r$, and $l$ a prime with $l \neq 2,p$.  Let $k$ be a field with $\text{char } k \neq l$. Let $d$ be the smallest positive integer such that $l \divides q^d - 1$, $s = v_l(q^d-1)$, and $n_0 = \lfloor \frac{n}{d} \rfloor$. Then 

\item (Theorem \ref{GLn})
$$\ed_k(GL_n(\F_q),l) = n_0[k(\zeta_{l^s}):k(\zeta_l)]$$

\item (Theorem \ref{SLn})  
$$\ed_k(SL_n(\F_q),l) = \begin{cases}
 \ed_k(GL_{n-1}(\F_q),l), &l \divides q - 1, \text{ } l \nmid n \\
 \ed_k(GL_n(\F_q),l), &l \nmid q - 1 \text{ or } l \divides n
 \end{cases}$$

\item (Theorem \ref{PGLn})
$$\ed_k(PGL_n(\F_q),l) = \begin{cases} \ed_k(GL_n(\F_q),l)&l \nmid q - 1\\
\ed_k(GL_{n-1}(\F_q),l), &l \divides q - 1 \text{ and } l \nmid n\\
l^{v_l(n)}(n-l^{v_l(n)})[k(\zeta_{l^s}):k(\zeta_l)], &l \divides q - 1 \text{ and } l \divides n \text{ and } n \neq l^t\\
l^{2t-1}[k(\zeta_{l^s}):k(\zeta_l)], &l \divides q - 1 \text{ and } n = l^t\\
 \end{cases}$$

\item (Theorem \ref{PSLn}) Let $n' \divides n$ and let $G = SL_n(\F_q)/\{x\text{Id} : x \in \F_q, \text{ } x^{n'} = 1\}$. Let $v = \min(v_l(n'),s)$. Then if $l \nmid q-1$ or $l \nmid n'$, then $\ed_k(G,l) = \ed_k(SL_n(\F_q),l)$. And if $l \divides q-1$, $l \divides n'$, then 
$$\ed_k(G,l) = \begin{cases} 
2, &l = n' = n = 3, \text{ }s = 1\\
\ed_k(PGL_n(\F_q),l), &\text{otherwise }
\end{cases}.$$
Note that for $n' = n,$ $G = PSL_n(\F_q).$

\item (Theorem \ref{PSp})
\begin{align*} 
\ed_k(PSp(2n,q),l) = \ed_k(Sp(2n,q),l) &= \begin{cases} \ed_k(GL_{2n}(\F_q),l), &d \text{ even} \\ 
\ed_k(GL_n(\F_q),l), &d \text{ odd} \end{cases}
\end{align*}

\item (Theorem \ref{On})
{\small \begin{align*}
&\ed_k(P\Omega^\epsilon(n,q),l)\\
&= \ed_k(O^\epsilon(n,q),l)\\
&= \begin{cases} \ed_k(GL_m(\F_q),l), &n = 2m+1, \text{ d odd}\\
&\text{or } n = 2m, d \text{ odd}, \epsilon = + \\
\ed_k(GL_{m-1}(\F_q),l), &n = 2m, d \text{ odd}, \epsilon = - \\
\ed_k(GL_{2m}(\F_q),l), &n = 2m+1, \text{ d even}\\
&\text{or } n = 2m, \text{ d even}, \epsilon = +, n_0 \text{ even}\\
&\text{or } n = 2m, d \text{ even}, \epsilon = -, n_0 \text{ odd}\\
ed_k(GL_{2m-2}(\F_q),l), &n = 2m, d \text{ even}, \epsilon = +, n_0 \text{ odd}\\
&\text{or } n = 2m, d \text{ even}, \epsilon = -, n_0 \text{ even}\\
\end{cases}
\end{align*}}

\item (Theorem \ref{Un})
$$\ed_k(U(n,q^2),l) = \begin{cases}\ed_k(GL_n(\F_{q^2}),l), &d = 2 \pmod{4}\\
\ed_k(GL_{\lfloor \frac{n}{2} \rfloor}(\F_{q^2}),l), &d \neq 2 \pmod{4}\end{cases}
$$

\item (Theorem \ref{SUn})
$$\ed_k(SU(n,q^2),l) = \begin{cases} \ed_k(U(n,q^2),l), &l \nmid q+1\\
\ed_k(SL_n(\F_{q^2}),l), &l \divides q + 1\end{cases}$$

\item (Theorem \ref{PSUn}) 
$$\ed_k(PSU(n,q^2),l) = \begin{cases} \ed_k(SU(n,q^2),l), &l \nmid n \text{ or } l \nmid q+1\\
\ed_k(PSL_n(\F_{q^2}),l), &l \divides n, \text{ } l \divides q + 1\end{cases}$$

\end{enumerate}
\end{theorem}

\begin{remark}\label{remark1} Duncan and Reichstein calculated the essential $p$-dimension of the pseudo-reflection groups. These groups overlap with the groups above in a few small cases. See the appendix (\ref{App1}) for the overlapping cases .
\end{remark}

\begin{note} When calculating essential $l$-dimension we can assume without loss of generality that $k$ contains a primitive $l$-th root of unity since adjoining an $l$-th root of unity gives an extension of degree prime to $l$.  However, this is not the case for $l^s$. For example, the cyclotomic polynomial for adjoining a $9$-th root of unity is $x^6 + x^3 + 1$, which has degree divisible by $3$. \end{note}

\noindent \textbf{Acknowledgements:} I would like to thank Zinovy Reichstein and Jesse Wolfson for their kind mentorship and invaluable help. I am also grateful to Hadi Salmasian, Federico Scavia, and Jean-Pierre Serre for very helpful comments on a draft.

\bigskip
 
\section{Essential Dimension and Representation Theory Background}\label{edbackground}

For completeness, we recall the relevant background. Fix a field $k$.  Let $G$ be a finite group, $p$ a prime. 

\begin{definition} Let $T: \text{Fields}/k \to \text{Sets}$ be a functor. Let $F/k$ be a field extension, and $t \in T(F)$. The\emph{ essential dimension of} $\mathit{t}$ is 
$$\ed_k(t) = \min_{F' \subset F \text{ s.t. } t \in Im(T(F') \to T(F))} \trdeg_k(F').$$\end{definition}

\begin{definition} Let $T: \text{Fields}/k \to \text{Sets}$ be a functor. The \emph{essential dimension of} $\mathit{T}$ is
$$\ed_{k}(T) = \sup_{t \in T(F), F/k \in \text{Fields}/k} \ed_k(t).$$\end{definition} 

\begin{definition} For $G$ be a finite group, let $$H^1(-;G):\text{Fields}/k \to \text{Sets}$$ be defined by $$H^1(-;G)(F/k) = \{\text{the isomorphism classes of } G\text{-torsors over }\text{Spec}F \}.$$ 
\end{definition}

\begin{definition} The \emph{essential dimension of} $\mathit{G}$ is
$$\ed_k(G) = \ed_k(H^1(-;G)).$$ \end{definition}

\begin{definition} Let $T: \text{Fields}/k \to \text{Sets}$ be a functor. Let $F/k$ be a field extension, and $t \in T(F)$. The \emph{essential} $\mathit{p}$\emph{-dimension of} $\mathit{t}$ is 
$$\ed_k(t,p) = \min \trdeg_k(F'')$$
where the minimum is taken over all
\begin{align*}
F'' \subset F' \text{ a finite extension}, \text{ with } F \subset F'\\
[F':F] \text{ finite } \text{ s.t. } p \nmid [F':F] \text{ and }\\
\text{the image of } t \text{ in } T(F') \text{ is in } \text{Im}(T(F'') \to T(F'))
\end{align*}
\end{definition}

\begin{note} $\ed_k(t,p) = \min_{F \subset F', p \nmid [F':F]} \ed_k(t|_{F'}).$
\end{note}

\begin{definition} Let $T: \text{Fields}/k \to \text{Sets}$ be a functor. The \emph{essential} $\mathit{p}$\emph{-dimension of} $\mathit{T}$ is
$$\ed_{k}(T,p) = \sup_{t \in T(F), F/k \in \text{Fields}/k} \ed_k(t,p).$$\end{definition}

\begin{definition} The \emph{essential} $\mathit{p}$\emph{-dimension of} $\mathit{G}$ is 
$$\ed_k(G,p) = \ed_k(H^1(-;G),p).$$ \end{definition}

Let $\syl_p(G)$ denote the set of Sylow $p$-subgroups of $G$. 

\begin{lemma}\label{Lempsyl} Let $S \in \syl_p(G)$. Then $\ed_k(G,p) = \ed_k(S,p).$  \end{lemma}

\begin{lemma}[\cite{Kni}, Corollary 2.11]\label{sylp} If $k_1/k$ a finite field extension of degree prime to $p$, $S \in \syl_p(G)$. Then $\ed_k(G,p) = \ed_k(S,p) = \ed_{k_1}(S,p).$ \end{lemma}

\begin{corollary}\label{rootofunity} Let $G$ be a finite group, $k$ a field of characteristic $\neq p$, $S \in \syl_p(G)$, $\zeta$ a primitive $p$-th root of unity, then
$$\ed_k(G,p) = \ed_{k(\zeta)}(S,p).$$ \end{corollary}

\begin{theorem}\label{KM4.1} [Karpenko-Merkurjev \cite{KM}, Theorem 4.1] Let $G$ be a $p$-group, $k$ a field with $\text{char } k \neq p$ containing a primitive $p$th root of unity. Then $\ed_k(G,p) = \ed_k(G)$ and $\ed_k(G,p)$ coincides with the least dimension of a faithful representation of $G$ over $k$.
\end{theorem}

\noindent The Karpenko-Merkurjev Theorem allows us to translate the question formulated in terms of extensions and transcendance degree into a question of representation theory of Sylow $p$-subgroups.  

\begin{definition}\label{Def3.1} Let $H$ be an abelian $p$-group. Define $H[p]$ to be the largest elementary abelian $p$-group contained in $H$, i.e. $H = \{z \in H : z^p = 1\}$. \end{definition}

\begin{definition} For $G$ an abelian group, $k$ a field, let $\widehat{G} = \text{Hom}(G, k_\text{sep}^\times)$, where $k_\text{sep}$ denotes a separable closure of $k$ in $\overline{k}$. \end{definition}

\begin{remark} Note that if $G$ is elementary abelian $p$-group and $k$ contains $p$-th roots of unity, then $\widehat{G}$ is simply the characters of $G$.
\end{remark}

\begin{remark} Note that for $G = (\Z/l^s\Z)^n$, $\widehat{G} = \text{Hom}(G,k(\zeta_{l^s})^\times)$.
\end{remark}

\begin{definition} For an abelian $p$-group $H$, let $\text{rank}(H)$ denote the rank of $H[p]$ as a vector space over $\F_p$.
\end{definition}

The next two lemmas are due to Meyer-Reichstein \cite{MR} and reproduced in \cite{BMS}.

\begin{lemma}[\cite{MR}, Lemma 2.3; \cite{BMS}, Lemma 3.5]\label{BMKS3.5} Let $k$ be a field with $\text{char } k \neq p$ containing $p$-th roots of unity. Let $H$ be a finite $p$-group and let $\rho$ be a faithful representation of $H$ of minimal dimension. Let $C = Z(H)$. Then $\rho$ decomposes as a direct sum of exactly $r = \text{rank}(C)$ irreducible representations
$$\rho = \rho_1 \oplus \ldots \oplus \rho_r.$$ and 
if $\chi_i$ are the central characters of $\rho_i$, then $\{\chi_i|_{C[p])}\}$ is a basis for $\widehat{C[p]}$ over $k$. \end{lemma}

\begin{lemma}[\cite{MR}, Lemma 2.3; \cite{BMS}, Lemma 3.4]\label{BMKS3.4} Let $k$ be a field with $\text{char } k \neq p$ containing $p$-th roots of unity. Let $H$ be a finite $p$-group and let $(\rho_i: H \to GL(V_i))_{1 \leq i \leq n}$ be a family of irreducible representations of $H$ with central characters $\chi_i$. Let $C = Z(H)$. Suppose that $\{\chi_i|_{C[p]} : 1 \leq i \leq n\}$ spans $\widehat{C[p]}$. Then $\bigoplus_i \rho_i$ is a faithful representation of $H$.
\end{lemma}

\noindent Lemmas \ref{BMKS3.5} and \ref{BMKS3.4} allow us to translate a question of analyzing faithful representations into a question of analyzing irreducible representations. We will need a few more lemmas for the proof.

\begin{definition} For $l$ a prime, $n \in \Z$, let $\mu_l(n)$ denote the the largest integer $d$ such that $l^d \leq n$. 
\end{definition}

\begin{lemma}\label{Pl(Sn)} Let $\sigma_{i}^{j}$ be the permutation which permutes the $i$th set of $l$ blocks of size $l^{j-1}$. Then 
$$\langle \{\sigma_i^j\}_{1 \leq j \leq \mu_l(n), 1 \leq i \leq \lfloor \frac{n}{l^{j}} \rfloor} \rangle \in \syl_l(S_n).$$ Let $P_l(S_n)$ denote this particular Sylow $l$-subgroup of $S_n$.
\end{lemma}

\begin{proof} 
For the proof, see the Appendix (\ref{App2}). \qedhere
\end{proof}

\begin{definition} Write $n$ in base $l$ as $n = \sum_{i=0}^{\mu_l(n)} a_i l^i$, and let $\xi_l(n)$ denote the sum of the nonzero digits of $n$ when written in base $l$, that is $\xi_l(n) = \sum_{i=0}^{\mu_l(n)} a_i$.
\end{definition}

\begin{definition} Let $I_j$ be the orbits of $\{1, \dots, n\}$ under the action of $P_l(S_n)$.  There are $\xi_l(n)$ such orbits. Let $i_j$ denote the smallest index in $I_j$. For each $j$, note that $|I_j| = l^{k}$ for some $k$. Let $k_j$ be such that $|I_j| = l^{k_j}$. 
\end{definition}

\begin{lemma}\label{edwreath} Let $H$ be a finite group. For any prime $l$, let $P = H^N \rtimes P_l(S_N)$. Then
$$\ed_k(P,l) = N\ed_k(H,l).$$
\end{lemma}

\begin{proof}

Since $H^N \subseteq P$, we immediately see that 
$$N\ed_k(H,l) = \ed_k(H^N,l) \leq \ed_k(P,l).$$

On the other hand, let $\rho': H \to V$ be a faithful representation of $H$ of minimal dimension such that $\dim(V) = ed_k(H, l)$. We can construct a faithful representation of $G$ as follows:  Let $\rho:G \to V^N$ be defined by $\otimes_i \rho$ on $H^N$ and the permutations in $S_N$ get sent to the corresponding permutation of the copies of $V^N$.  This a faithful representation of $G$ of the desired dimension. Thus 
\begin{align*} \ed_k(G,l) &\leq \dim(V^N) = N\dim(V) = N\ed_k(H,l).
\qedhere
\end{align*}
\end{proof}

\begin{lemma}\label{restrZ}  A representation of a finite $p$-group, $H$, is faithful if and only if its restriction to $Z(H)$ is faithful and if and
only if its restriction to $Z(H)[p]$ is faithful.
\end{lemma}
\begin{proof}
Let $H$ be a finite $p$-group.

\begin{claim*}\label{NcapZ} For $N$ a nontrivial normal subgroup of $H$, $N \cap Z(H)$ and $N \cap Z(H)[p]$ are not trivial.\end{claim*}
Granting this claim, we can complete the proof.\\
$\Rightarrow$: The forward direction is obvious. \\
$\Leftarrow$: Let $\rho$ be a representation of $H$ whose restriction to $Z(H)[p]$ is faithful.  Let $K = \text{ker} \rho$. Then $K$ is normal in $H$ and since the restriction of $\rho$ to $Z(H)[p]$ is faithful, we have $K \cap Z(H)[p] = \{1\}$.  Then by the Claim, $K$ must be trivial.  Thus $K = \{1\}$. An identical proof works for if $\rho$ restricted to $Z(H)$ is faithful.
\qedhere
\end{proof}

\begin{proof}[Proof of the Claim]
We will first show that $N \cap Z(H)$ is not trivial. Applying the class equation to $H \curvearrowright N$ by conjugation, we have 
$$|N| = |N^H| + \sum_{n_i} [H: \text{Stab}_{n_i}],$$
where the $n_i$ are chosen one from each conjugacy class, $N^H = \{n \in N : hnh^{-1} = n, \forall h \in H\} \subseteq Z(H)$ are the fixed points, and $\text{Stab}_{n_i} = \{h \in h : hn_ih^{-1} = n_i\}$ is the stabilizer of $n_i \in N$.  Since the stabilizer is a subgroup of $H$, a $p$-group, we have $$p \divides [H : \text{Stab}_{n_i}], \text{ for all } i.$$ Thus we must have $p \divides |N^H|$.  And since $1 \in N^H$, $N^H$ is not empty. Therefore, $|N^H| \geq p$. And $N^H \subseteq Z(H)$. Therefore, $N \cap Z(H)$ is not trivial.

Since $N \cap Z(H)$ is not trivial and is contained in a $p$-group, we must have an element of order $p$ by Cauchy's Lemma. Thus $N \cap Z(H)[p]$ is not trivial. 
\qedhere
\end{proof}

\begin{definition} Let $|G|_l = l^{v_l(|G|)}$; i.e. $|G|_l$ is the order of a Sylow $l$-subgroup of $G$. \end{definition}

\begin{lemma}\label{claim1} For an invertible matrix $A$, there is a rearrangement of the columns such that $a_{i,i} \neq 0$ for all $i.$. 
\end{lemma}

\begin{proof} An invertible matrix has non-zero determinant. The determinant of an $n \times n$ matrix is given by 
$$\sum_{\sigma \in S_n} (\prod_{i=1}^n a_{i,\sigma(i)} (-1)^{\text{sign}(\sigma)}.$$
Thus there exists some $\sigma \in S_n$, such that $\prod_{i=1}^n a_{i,\sigma(i)} (-1)^{\text{sign}(\sigma)} \neq 0$, and so each $a_{i,\sigma(i)} \neq 0$.  Then by applying $\sigma$ to the columns, we will get that $a_{i,i} \neq 0$ for all $i$. \qedhere

\end{proof}

\begin{definition} Let $\mu_{l^s}$ denote the group of $l^s$-th roots of unity. Note that $\mu_{l^s} =  \langle \zeta_{l^s} \rangle$.
\end{definition}

\begin{definition}\label{psidef} For $\mbf{a} \in (\Z/2^s\Z)^n$,  define $\psi_{\mbf{a}} \in \widehat{(\mu_{2^s})^n}$ to be $\psi_{\mbf{a}}: (\mu_{2^s})^n \to k(\zeta_{2^s})^\times$ given by 
 $$\psi_{\mbf{a}}(\mbf{x}) = \prod_{i=1}^n (x_i)^{a_i}.$$
Let $f = \frac{2^s}{\text{gcd}(a_i)}$. View $k(\zeta_f)$ as a vector space over $k$. Let $d = [k(\zeta_f):k]$. Then multiplication in $k(\zeta_f)$ corresponds to an element of $GL_d(k)$. Let the representation $\Psi_{\mbf{a}}: (\mu_{2^s})^n \to GL_d(k)$ be defined by 
$$\Psi_{\mbf{a}}(\mbf{x}) = \text{ multiplication by } \prod_{i=1}^n (x_i)^{a_i}.$$
\end{definition}

\begin{remark} Note that the map given by $\mbf{a} \mapsto \psi_{\mbf{a}}$ is an isomorphism between $(\Z/l^s\Z)^n$ and $\widehat{\mu_{l^s}}$. 
\end{remark}

\begin{definition}
Let $\Gamma = \text{Gal}(k(\zeta_{l^s})/k)$. For $\phi \in \Gamma$, note that $\phi(\zeta_{l^s}) = (\zeta_{l^s})^{\gamma_\phi}$ for a unique $\gamma_\phi \in (\Z/l^s\Z)^\times$. Define $\gamma_\phi$ to be the element of $(\Z/l^s\Z)^\times$ such that $\phi(\zeta_{l^s}) = (\zeta_{l^s})^{\gamma_\phi}$.
\end{definition}

\begin{remark} Note that the map $\phi \mapsto \gamma_\phi$ gives an injection $\Gamma \hookrightarrow (\Z/l^s\Z)^\times$. \end{remark}

\begin{lemma}\label{auxlemma} For any prime $l$, let $\Gamma = \text{Gal}(k(\zeta_{l^s})/k)$. Consider the action of $\Gamma$ on $\widehat{(\mu_{l^s})^n}$ given by $\phi(\psi_{\mbf{a}}) = \phi \circ \psi_{\mbf{a}}$. Then the corresponding action of $\gamma_\phi \in (\Z/l^s\Z)^\times$ on $(\Z/l^s\Z)^n \cong \widehat{(\mu_{l^s})^n}$ is given by scalar multiplication by $\gamma_\phi$. \end{lemma}

\begin{proof}

Let $\phi \in \Gamma,$  $\gamma = \gamma_\phi$, and $\psi_{\mbf{a}} \in \widehat{(\mu_{l^s})^n}$. Then
$$(\phi \circ \psi_{\mbf{a}})(e_i) = \phi((\zeta_{l^s})^{a_i}) = \zeta_{l^s}^{\gamma a_i} = \psi_{\gamma \mbf{a}}(e_i).$$  Thus the action of $\gamma$ in $(\Z/l^s\Z)^\times$ on $(\Z/l^s\Z)^n$ is given by scalar multiplication.
\qedhere
\end{proof}

\begin{lemma}\label{corrlemma} For any prime $l$, let $\Gamma = \text{Gal}(k(\zeta_{l^s})/k) \hookrightarrow (\Z/l^s\Z)^\times$. Then the irreducible representations of $(\mu_{l^s})^n$ over $k$ are in bijection with $\mbf{a} \in (\Z/l^s\Z)^n/\Gamma$, where the action of $\phi \in \Gamma$ is given by scalar multiplication by $\gamma_\phi$. The bijection is given by $\mbf{a} \in (\Z/l^s\Z)^n/\Gamma \mapsto \Psi_{\mbf{a}}: (\mu_{l^s})^n \to GL_d(k)$, where $d = [k(\zeta_f):k]$ for $f = \frac{l^s}{\text{gcd}(a_i)}$. Furthermore, if $\Psi_\mbf{a}$ is non-trivial on $(\mu_{l^s})^n[l]$, then $l \nmid a_i$ for some $i$ and $\Psi_{\mbf{a}}$ has dimension $[k(\zeta_{l^s}):k]$.
\end{lemma}

\begin{proof}
Let $S = (\mu_{l^s})^n$. Note that 
$$k[S] = M_{n_1}(D_1) \times \dots \times M_{n_t}(D_t),$$
for $D_i$ a division algebra over $k$. So $k[S]$ is an \'etale $k$-algebra.  Let $k_\text{sep}$ denote a separable closure of $k$ in $\overline{k}$. Let $\Gamma' = \text{Gal}(k_\text{sep}/k)$. For $A$ a finite \'etale $k$-algebra, consider the action of $\Gamma'$ on $\text{Hom}_k(A,k_\text{sep})$ given by $\phi(\lambda) = \phi \circ \lambda$ for $\phi \in \Gamma'$, $\lambda \in \text{Hom}_k(A, k_\text{sep})$. By \cite{Sz} (Theorem 1.5.4), the functor mapping a finite \'etale $k$-algebra $A$ to the finite set $\text{Hom}_k(A, k_\text{sep})$ gives an anti-equivalence between the category of finite \'etale $k$-algebras and the category of finite sets equipped with a continuous left $\Gamma'$-action; separable field extensions give rise to sets with transitive $\Gamma'$-action and Galois extensions give rise to $\Gamma'$-sets isomorphic to finite quotients of $\Gamma'$.  So if we write an \'etale $k$-algebra as a finite direct product of separable extensions of $k$, $A = L_1 \times L_2 \times \dots \times L_t$, then we can write $\text{Hom}_k(A, k_\text{sep})) = X_1 \coprod X_2 \coprod \dots \coprod X_t$, where $X_1, \dots, X_t$ are the orbits of $\Gamma'$ in $\text{Hom}_k(A, k_\text{sep})$. 

For $A = k[S] = \prod_i k_i$, then we have $\text{Hom}_{k}(k[S], k_\text{sep}) = \widehat{S}$. The irreducible representations of $S$ correspond to the factors $L_1, \dots, L_t$ in $k[S] = L_1 \times \dots \times L_t$ and these in turn correspond to $\widehat{S}/\Gamma'$, the orbits of $\Gamma'$ in $\widehat{S}$.  The correspondence between $\widehat{S}$ and irreducible representation of $S$ is given by $\psi_{\mbf{a}} \in \widehat{S}/\Gamma'$ corresponding to $\Psi_{\mbf{a}} \in \text{Irr}(S)$, where $\psi_{\mbf{a}}(e_i) = (\zeta_{l^s})^{a_i}$ and  $\Psi_{\mbf{a}}(e_i) = \text{ multiplication by } (\zeta_{l^s})^{a_i}$.

Note that for any $\lambda \in \widehat{S}$, $\lambda$ is a map on $k(\zeta_{l^s})$. Also, note that a homomorphism $k_\text{sep} \to k_\text{sep}$ must send $\zeta_{l^s}$ to a $l^s$-th root of unity; so it will map $k(\zeta_{l^s})$ to itself.  Thus any homomorphism in $\Gamma'$, when restricted to $k(\zeta_{l^s})$ will be an element of $\text{Gal}(k(\zeta_{l^s})/k)$. So in our case, we may replace $\Gamma'$ by $\Gamma = \text{Gal}(k(\zeta_{l^s})/k).$ By Lemma \ref{auxlemma}, the corresponding action of $\gamma_\phi \in (\Z/l^s\Z)^\times$ on $(\Z/l^s\Z)^n \cong \widehat{S}$ is given by scalar multiplication by $\gamma_\phi$.

Note that $S[l]$ is given by $\mbf{b} \in (\mu_{l^s})^n$ with $\mbf{b}^l = (1,\dots,1)$.  So $\mbf{b} = (\zeta_l^{x_1}, \dots \zeta_l^{x_n})$ for some $x_i \in \Z/l\Z$. Hence 
\begin{align*}
\psi_{\mbf{a}}(\mbf{b}) = (\zeta_l)^{\sum_{i=1}^n a_i x_i}
\end{align*}
Thus if $l \divides a_i$ for all $i$, then $\psi_{\mbf{a}}$ will be trivial on $S[l]$, and hence $\Psi_{\mbf{a}}$ will be trivial on $S[l]$ as well. So if $\Psi_{\mbf{a}}$ is non-trivial on $S[l]$, we must have $l \nmid a_i$ for some $i$, and so $\text{gcd}(a_i) = 1$. Thus the $\Psi_{\mbf{a}}$ that are non-trivial on $S[l]$ have dimension $[k(\zeta_{l^s}):k]$.
\qedhere
\end{proof}

\begin{lemma}\label{changepersp} For any prime $l$, let $\Gamma = \text{Gal}(k(\zeta_{l^s})/k) \hookrightarrow (\Z/l^s\Z)^\times$ and the action of $\phi \in \Gamma$ be given by scalar multiplication by $\gamma_\phi$. Then the orbit of $\Psi_{\mbf{a}}$ under the action of $P_l(S_n)$ on $\text{Irr}((\mu_{l^s})^n)$ will have the same size as the orbit of $\mbf{a}$ under the action of $P_l(S_n)$ on $(\Z/l^s\Z)^n/\Gamma$.  
\end{lemma}

\begin{proof}
Let $S = (\mu_{l^s})^n$. By Lemma \ref{corrlemma}, the irreducible representations of $S$ are in bijection with $\mbf{a} \in (\Z/l^s\Z)^n/\Gamma$, where the action of $\phi \in \Gamma$ is given by scalar multiplication by $\gamma_\phi$. The bijection is given by $\mbf{a} \in (\Z/l^s\Z)^n/\Gamma \mapsto \Psi_{\mbf{a}}.$ 

The action of $P_l(S_n)$ on $\text{Irr}(S)$ is given by 
$$\sigma(\lambda)(x) = \lambda(\sigma(x)),$$ 
for $\sigma \in P_l(S_n)$, $\lambda \in \text{Irr}(S)$, $x \in S$. And for $\Psi_{\mbf{a}} \in \text{Irr}(S),$ the orbit of $\Psi_{\mbf{a}}$ in $\text{Irr}(S)$ under the action of $P_l(S_n)$ corresponds to the orbit of $\psi_{\mbf{a}}$ in $\widehat{S}$ under the action of $P_l(S_n)$.

Under the isomorphism $\widehat{S} \cong (\Z/l^s\Z)^n$, we have that the action of $P_l(S_n)$ on $\widehat{S}$, which is given by 
$$\sigma_{\psi_\mbf{a}}(x) = \psi_{\mbf{a}}(\sigma(x)) = \psi_{\sigma^{-1}(\mbf{a})}(x),$$
corresponds to the action of $P_l(S_n)$ on $(\Z/l^s\Z)^n$ given by $\mbf{a} \mapsto \sigma^{-1}(\mbf{a}).$ 

Note that the action of $P_l(S_n)$ commutes with the action of $\Gamma$, so we get a corresponding action of $P_l(S_n)$ on $(\Z/l^s\Z)^n/\Gamma$ under the bijection $\text{Irr}(S) \leftrightarrow (\Z/l^s\Z)^n/\Gamma$, which is also given by $\mbf{a} \mapsto \sigma^{-1}(\mbf{a})$. The orbit of $\mbf{a}$ under this action will have the same size as the orbit of $\mbf{a}$ under the action $\mbf{a} \mapsto \sigma(\mbf{a})$. 

Therefore, the orbit of $\Psi_{\mbf{a}}$ under the action of $P_l(S_n)$ on $\text{Irr}(S)$ has the same size as the orbit of $\mbf{a}$ in $(\Z/l^s\Z)^n/\Gamma$ under the action of $P_l(S_n)$ given by $\mbf{a} \mapsto \sigma(\mbf{a})$.
\end{proof}

\subsection{Clifford's Theorem}

Let $N \triangleleft G$, $L = G/N$. 

\begin{definition} For a representation $\rho: G \to GL(V)$, $f:G' \to G$, define $f^*(\rho): G' \to GL(V)$ by $f^*(\rho) = \rho \circ f$. \end{definition}

Then note that for $f_1:G'' \to G'$, $f_2: G' \to G$, we have that 
$$f_1^*(f_2^*(\rho)) = (\rho \circ f_2) \circ f_1 = \rho \circ (f_2 \circ f_1) = (f_2 \circ f_1)^*(\rho).$$

Let $\text{Rep}(G)$ denote the set of isomorphism classes of representations of $G$.  Then $\text{Aut}(G)$ acts on $\text{Rep}(G)$ by $f_\rho = (f^{-1})^*(\rho)$ for $f \in \text{Aut}(G), \rho \in \text{Rep}(G)$. This is an action since 
$$(f \circ g)_\rho = ((f \circ g)^{-1})^*(\rho) = \rho \circ (g^{-1} \circ f^{-1}) = (\rho \circ g^{-1}) \circ f^{-1} = (f^{-1})^*((g^{-1})^*(\rho)).$$

Let $\text{Irr}(G) \subset \text{Rep}(G)$ denote the set of isomorphism classes of irreducible representations of $G$. Then $\text{Irr}(G)$ is invariant under the action of $\text{Aut}(G)$ since if $\rho$ is irreducible, then $f_\rho = (f^{-1})^*(\rho) = \rho \circ f^{-1}$ is also irreducible.

Let $\text{Inn}(G)$ denote the set of inner automorphisms of $G$.  for $g \in G$, let $\varphi_g$ denote the inner automorphism $\varphi_g: G \to G$ defined by $\varphi_g(h) = g^{-1}hg$.  

Note that since $N \triangleleft G$, we can restrict $\varphi_g$ to $N$, and so we have $G \to \text{Aut}(N)$ given by $g \mapsto (\varphi_g|_N \in \text{Aut(N)})$. Then $G$ acts on $\text{Rep}(N)$ by 
$$g(\lambda) := (\varphi_g|_N)(\lambda) = ((\varphi_g|_N)^{-1})^*(\lambda) = \lambda \circ (\varphi_g|_N)^{-1}$$
for $g \in G, \lambda \in \text{Rep}(N)$.

Note that if $g \in N$, we can define $\alpha:V \to V$ by $\varphi(v) = \lambda(g)v$, and then for $h \in G$, $v \in V$

$$(\lambda(h) \circ \alpha)(v) = \lambda(h)(\lambda(g)v) = \lambda(hg)(v)$$
and 
$$(\alpha \circ g(\lambda)(h))(v) = (\alpha \circ \lambda \circ (\varphi_g|_N)^{-1}(h))(v) = g\lambda(g^{-1}hg)(v) = \lambda(hg)(v).$$

Thus $g(\lambda) \cong \lambda$ for $g \in N$. Thus $N$ acts trivially on $\text{Irr}(N) \subset \text{Rep}(N)$. Hence $L = G/N$ acts on $\text{Irr}(N) \subset \text{Rep}(N)$.

\begin{theorem}[Clifford's Theorem]\label{cliff}  Let $N \triangleleft G$, $L = G/N$, $\rho \in \text{Irr}(G)$. Then there exist pairwise non-isomorphic $\lambda_1, \dots, \lambda_c \in \text{Irr}(N)$ such that 
$$\rho|_{N} \cong  \left( \oplus_{i=1}^c  \lambda_i \right)^{\oplus d}, \text{ for some } c, d,$$ 
Let $S \subset \text{Irr}(N)$ be the set $\{ \overline{\lambda_1}, \dots, \overline{\lambda_c}\}$.  Then 
\begin{enumerate}
\item $S$ is an $L$-invariant subset of $\text{Irr}(N)$.
\item $L$ acts on $S$ transitively.
\end{enumerate}
\end{theorem}

\begin{proof}

Write $\rho|_N$ as a direct sum of irreducible representations:
\begin{align*}
&\rho|_N \cong \mu_1 \oplus \mu_2 \oplus \dots \oplus \mu_t\\
&V = W_1 \oplus W_2 \oplus \dots \oplus W_t,
\end{align*}
where $W_i$ is an $N$-subspace and $\mu_i: N \to GL(W_i)$ is an irreducible representation. Then for $g \in G$,
$$\rho|_N \cong (g_\rho)|_N \cong g_{\mu_1} \oplus g_{\mu_2} \oplus \dots \oplus g_{\mu_t}.$$
So for all $g_{\mu_j}$ there exists $\mu_i$ such that $g_{\mu_j} \cong \mu_i$. And similarly, for all $g_i$ there exist $g_{\mu_j}$ such that $g_i \cong g_{\mu_j}$. Hence $G$ acts on $\{\overline{\mu_1}, \overline{\mu_2}, \dots, \overline{\mu_s}\}$. And since $N$ acts trivially, this means $L = G/N$ acts on $S = \{\overline{\mu_1}, \overline{\mu_2}, \dots, \overline{\mu_s}\}$. Gather together the isomorphic $\mu_i$ to write
$$\rho|_N \cong \lambda_1^{\oplus d_1} \oplus \lambda_2^{\oplus d_2} \oplus \dots \oplus \lambda_c^{\oplus d_c},$$
where the $\lambda_i$ are non-isomorphic. Then $L$ acts on $\{\overline{\lambda_1}, \dots, \overline{\lambda_c}\} = \{\overline{\mu_1}, \overline{\mu_2}, \dots, \overline{\mu_s}\}$. This proves the first statement in the theorem.

Let $W$ be an irreducible $N$-subspace of $V$ and $\lambda: N \to GL(W)$ the corresponding sub-representation of $\rho$. For $g \in G$,
$$ngW = gg^{-1}ngW = gW.$$
So $gW$ is also an $N$-subspace. and $\lambda': N \to GL(gW)$ is isomorphic to $g(\lambda)$, so $gW$ is also irreducible.

Note that $gW$ is also irreducible since if $gW$ were reducible, say $gW = W_1 \oplus W_2$, then we could write $W = g^{-1}(gW) = g^{-1}W_1 \oplus g^{-1}W_2$ and so $W$ would not be irreducible. So $gW$ is an irreducible subspace of $V_N$.

Then consider $\sum_{g \in G} gW$. This is a $G$-invariant subspace of $V$, thus since $V$ is irreducible we must have $V = \sum_{g \in G} gW$. This can be refined to be a direct sum. Then we have $V$ written as a direct sum in two different ways: $V = \oplus_{j=1}^{t'} g_jW$ and $V = \oplus_{i=1}^t W_i$. So we must have $t'  = t$ and $W_i \cong g_iW$ for some $g_i \in G$. For $\overline{\lambda_i}, \overline{\lambda_j}$, let $W_i = g_iW, W_j = g_jW$ be the corresponding $N$-subspaces. Then 
$$W_j = g_jW = (g_jg_i^{-1})g_iW = (g_jg_i^{-1})W_i.$$
Thus for $g = g_jg_i^{-1}$, we have $\overline{\lambda_j} = g_{\overline{\lambda_i}}$. Hence $G$ acts transitively on $S$ and so since $N$ acts trivially, we can conclude that $L = G/N$ acts transitively on $S$.
\qedhere
\end{proof}

\section{The General Linear Groups}

In this section, we will prove the following theorem:

\begin{theorem}\label{GLn} Let $p$ be a prime, $q = p^r$, and $l$ a prime with $l \neq 2,p$.  Let $d$ be the smallest positive integer such that $l \divides q^d - 1$, $s = v_l(q^d-1)$, and $n_0 = \lfloor \frac{n}{d} \rfloor$.  Then
$$\ed_k(GL_n(\F_q),l) = n_0[k(\zeta_{l^s}):k(\zeta_l)].$$
\end{theorem}

\noindent By (\cite{Sta}, Lemma 3.1), for $l \neq 2$,
$$|GL_n(\F_q)|_l = l^{sn_0 + \lfloor \frac{n_0}{l} \rfloor + \lfloor \frac{n_0}{l^2} \rfloor + \ldots } =  l^{sn_0} \cdot |S_{n_0}|_l.$$

\begin{proposition}\label{GLsyl} For $P \in \syl_l(GL_n(\F_q))$, 
$$P \cong (\mu_{l^s})^{n_0} \rtimes P_l(S_{n_0}).$$
\end{proposition}

Granting this proposition, we can prove Theorem \ref{GLn}:
\begin{proof}[Proof of Theorem \ref{GLn}]

By Lemma \ref{edwreath} and Proposition \ref{GLsyl}, 
$$\ed_k(GL_n(\F_q),l) = n_0\ed_k(\mu_{l^s}, l).$$
By Corollary \ref{rootofunity}, we may assume that $k$ contains a primitive $l$-th root of unity, that is, we may assume that $k = k(\zeta_l)$. Then by Theorem \ref{KM4.1}, $\ed_{k}(\mu_l^s,l) = \ed_{k}(\mu_{l^s})$. And by \cite{KM} (Corollary 5.2), $\ed_{k}(\mu_{l^s}) = [k(\zeta_{l^s}):k]$, where $\zeta_{l^s}$ denotes a primitive $l^s$-th root of unity. Thus 
\begin{align*}
\ed_k(GL_n(\F_q),l) &= n_0[k(\zeta_{l^s}):k(\zeta_l)].
\qedhere
\end{align*}
 \end{proof}
\begin{proof}[Proof of Proposition \ref{GLsyl}]
\noindent \footnote{This construction follows \cite{Sta}.}Let $\zeta_{l^s}$ be a primitive $l^s$-th root of unity in $\F_{q^d}$, and let $E$ be the image of $\zeta_{l^s}$ in $GL_d(\F_q)$. There are $n_0$ copies of $\langle E \rangle$ in $GL_n(\F_q)$, given by $\langle E_1 \rangle, \ldots, \langle E_{n_0} \rangle$ where
$$E_1 = \begin{pmatrix}
E &   &        & \\
  & 1 &        & \\
  &   & \ddots & \\
  &   &        & 1
\end{pmatrix}, \ldots, E_{n_0} = \begin{pmatrix}
1 &        &   &   &\\
  & \ddots &   &   &\\
  &        & 1 &   &\\
  &        &   & E &\\
  &        &   &   & \text{Id}_{n-n_0d}
\end{pmatrix}$$
The symmetric group on $n_0$ letters acts on $\langle E_1, \ldots, E_{n_0} \rangle$ by permuting the $E_i$, and it can be embedded into $GL_n(\F_q)$. Let
\begin{align*} 
P &= \langle E_1, \ldots, E_{n_0} \rangle \rtimes P_l(S_{n_0})\\
&\cong (\mu_{l^s})^{n_0} \rtimes P_l(S_{n_0})
\end{align*}
Then

\begin{align*}
|P|& = |(\mu_{l^s})^{n_0}| \cdot |P_l(S_{n_0})|\\
&= |GL_n(\F_q)|_l
\end{align*}
Therefore, $P \in \syl_l(GL_n(\F_q))$.
\qedhere
\end{proof}

\section{The Special Linear Groups}

\begin{theorem}\label{SLn} Let $p$ be a prime, $q = p^r$, and $l$ a prime with $l \neq 2,p$.  Let $k$ be a field with $\text{char } k \neq l$.  Then 
$$\ed_k(SL_n(\F_q),l) = \begin{cases}
 \ed_k(GL_{n-1}(\F_q),l), &l \divides q - 1, \text{ } l \nmid n \\
 \ed_k(GL_n(\F_q),l), &l \nmid q - 1 \text{ or } l \divides n
 \end{cases}$$
\end{theorem}

By (\cite{Gr}, Proposition 1.1), 
$$|SL_n(\F_q)| = \frac{|GL_n(\F_q)|}{q-1}.$$
So
\begin{align*}
|SL_n(\F_q)|_l &= \frac{|GL_n(\F_q)|_l}{l^{v_l(q-1)}} = l^{s(n-1)} \cdot |S_{n}|_l\\
\end{align*}

If $l \nmid q - 1$, then the Sylow $l$-subgroups of $SL_n(\F_q)$ are isomorphic to the Sylow $l$-subgroups of $\GL_n(\F_q)$. So we need only prove the case when $l \divides q - 1$. Thus in this section, we will assume $l \divides q - 1$, and so we have $d = 1$, $s = v_l(q-1)$, and $n_0 = n$ in the notation used in the section on $GL_n(\F_q)$.

The proof when $l \divides q - 1, l\nmid n$ is simple:
\begin{proof}[Proof of Theorem \ref{SLn} for the case $l \divides q-1$, $l \nmid n$]
Note that we can embed $GL_{n-1}(\F_q)$ in $SL_n(\F_q)$ by sending the matrix $A \in GL_{n-1}(\F_q)$ to 
$$\begin{pmatrix} A & 0\\
0 & \text{det}(A^{-1})\end{pmatrix}.$$
If $l \nmid n$, then $|S_n|_l = |S_{n-1}|_l$, thus 
$$|SL_n(\F_q)|_l = l^{s(n-1)} \cdot |S_n|_l = l^{s(n-1)} \cdot |S_{n-1}|_l = |GL_{n-1}(\F_q)|_l.$$
Therefore, the Sylow $l$-subgroups of $SL_n(\F_q)$ are isomorphic to Sylow $l$-subgroups of $GL_{n-1}(\F_q)$. Thus 
\begin{align*} \ed_k(SL_n(\F_q),l) &= \ed_k(GL_{n-1}(\F_q),l) = (n-1)[k(\zeta_{l^s}):k].
\qedhere
\end{align*}
\end{proof}

For the remainder of this section, we will assume that $l \divides n$ (in addition to $l \divides q-1$). Also, by Corollary \ref{rootofunity}, we may assume that $k$ contains a primitive $l$-th root of unity, that is, we may assume that $k = k(\zeta_l)$.

\subsection{\texorpdfstring{A Sylow $l$-subgroup and its center}{A Sylow l-subgroup and its center}}

%
%
%

\begin{lemma} For $P \in \syl_l(SL_n(\F_q))$, 
$$P \cong \{(\mbf{b}, \tau) \in (\mu_{l^s})^n \rtimes P_l(S_n) : \prod_{i=1}^n b_i = \text{sgn}(\tau)\},$$
where the action of $P_l(S_n)$ on $\mbf{b} \in (\mu_{l^s})^n$ is given by permuting the $b_i$.
\end{lemma}

\begin{proof}
By Proposition \ref{GLsyl}, the Sylow $l$-subgroups of $GL_n(\F_q)$ are isomorphic to $(\mu_{l^s})^n \rtimes P_l(S_n)$. Let 
$$P = \{(\mbf{b}, \tau) \in (\mu_{l^s})^n \rtimes P_l(S_n) : \prod_{i=1}^n b_i = \text{sgn}(\tau)\}.$$
Then $P \subset SL_n(\F_q)$ and 
$$|P| = \frac{l^{sn} \cdot |P_l(S_n)|_l}{l^s} = l^{s(n-1)} \cdot |S_n|_l = |SL_n(\F_q)|_l.$$
Thus $P$ is isomorphic to a Sylow $l$-subgroup of $SL_n(\F_l)$.
\end{proof}

\begin{remark}
Note that $P_l(S_n)$ is the direct product $\times_j P_l(S_{l^{k_j}})$ acting on $\{1,\dots,n\} = \cup_{j=1}^{\xi_l(n)} \{i_j, \dots, i_j + l^{k_j}-1\}$. So there  are $\xi_l(n)$ orbits of $\{1, \dots, n\}$ under the action of $P_l(S_n)$.  
\end{remark}

 \begin{lemma}\label{SLZ}  If $l  \divides n$, then for $P \in \syl_l(SL_n(\F_q))$,
 $$Z(P)[l] \cong (\mu_l)^{\xi_l(n)}.$$
\end{lemma}
\begin{proof} Fix $(\mbf{b},\tau) \in P = \{(\mbf{b}, \tau) \in (\mu_{l^s})^n \rtimes P_l(S_n) : \prod_{i=1}^n b_i = \text{sgn}(\tau)\}$. Then for $(\mbf{b}', \tau') \in P$, 
$$(\mbf{b},\tau)(\mbf{b}',\tau') = (\mbf{b} \tau(\mbf{b}'), \tau\tau') \text{ and } (\mbf{b}', \tau')(\mbf{b},\tau) = (\mbf{b}' \tau'(\mbf{b}), \tau'\tau).$$
Thus $(\mbf{b},\tau)$ is in the center if and only if $\tau \in Z(P_l(S_{n}))$ and 
$$\mbf{b}\tau(\mbf{b}') = \mbf{b}'\tau'(\mbf{b})$$ for all $\mbf{b}',\tau'$. Choosing $\tau' = \Id$, we see we must have $\mbf{b}\tau(\mbf{b}') = \mbf{b}'\mbf{b}$ for all $(\mbf{b'},\text{Id}) \in P$.  Thus we must have $\tau(\mbf{b}') = \mbf{b}'$ for all $(\mbf{b}', \Id)$ in $P$. For any $\tau \neq \Id$, we can choose $\mbf{b'}$ with $\prod_{i=1}^n b_i = 1$ (and hence $(\mbf{b}',\Id) \in P$) for which $\tau(\mbf{b}') \neq \mbf{b'}$. Thus we can conclude that we must have $\tau = \text{Id}$.

We also need $\tau'(\mbf{b}) = \mbf{b}$ for all $(\mbf{b}',\tau')$ in $P$. For any $\tau' \in P_l(S_{n})$, we can choose $\mbf{b}'$ such that $(\mbf{b}', \tau') \in P$, so we need $\tau'(\mbf{b}) = \mbf{b}$ for all $\tau' \in P_l(S_n)$.   So $b_i = b_{i'}$ for $i,i'$ in the same $I_j$.  So 
$$Z(P) = \{\mbf{b} \in (\mu_{l^s})^n : \prod_{i=1}^n b_i = 1 \text{ and the } b_i \text{ are equal in the blocks } I_j\}.$$
Then 
$$Z(P)[l] = \{\mbf{b} \in (\mu_{l})^n : \prod_{i=1}^n b_i = 1 \text{ and the } b_i \text{ are equal in the blocks } I_j\}.$$
 If we have the $b_i$ equal in the same $I_j$, we have
$$\prod_{i=1}^n b_i = \prod_{j=1}^{\xi_l(n)}  x_j^{l^{k_j}},$$
where $l^{k_j} = |I_j|$ denotes the size of the $j$-th block of $b_i$ equal to each other and $x_j$ is the value of the $b_i$ in that block. Note that since $l \divides n$, we know that $l \divides |I_j|$ for all $j$; thus $l \divides l^{k_j}$ for all $j$ and hence $x_j{l^{k_j}} = 1$ (since $x_j \in \mu_l$). Thus $\prod_{i=1}^n b_i = 1$.  So we can remove the condition that $\prod_{i=1}^n b_i = 1$ for $Z(P)[l]$:
\begin{align*} Z(P)[l] &= \{\mbf{b} \in (\mu_l)^n : \text{ the } b_i \text{ are equal in the blocks } I_j\} \cong (\mu_l)^{\xi_l(n)}. \qedhere
\end{align*}
\end{proof}

\subsection{\texorpdfstring{Case 1:  $n = l^t, $}{Case 1: n=lt}}

\begin{proof}[Proof of Theorem \ref{SLn} for the case $l \divides q-1$, $l \divides n$, $n = l^t$]

Let $P = \{(\mbf{b}, \tau) \in (\mu_{l^s})^n \rtimes P_l(S_n) : \prod_{i=1}^n b_i = \text{sgn}(\tau)\}$. Note that since $(\mu_{l^s})^{n-1} \subset SL_n(\F_q) \subset GL_n(\F_q)$, we have
 $$(n-1)[k(\zeta_{l^s}):k] \leq \ed_k(SL_n(\F_q),l) \leq n[k(\zeta_{l^s}):k].$$

Let $\rho$ be a faithful representation of $P$ of minimum dimension (and so it is also irreducible since the center has rank $1$). Then $\dim(\rho) \geq (l^t-1)[k(\zeta_{l^s}):k]$. Let $T = \{\mbf{b} \in (\mu_{l^s})^n : \prod_{i=1}^n b_i = 1\} \subset P$.  Then $T \triangleleft P$ and so by Clifford's Theorem (Theorem \ref{cliff}), $\rho|_T$ decomposes into a direct sum of irreducibles in the following manner:
$$\rho|_{T} \cong  \left( \oplus_{i=1}^c  \lambda_i \right)^{\oplus d}, \text{ for some } c, d,$$ 
with the $\lambda_i$ non-isomorphic, and $P/T$ acts transitively on the isomorphism classes of the $\lambda_i$. So the $\lambda_i$ have the same dimension and the number of $\lambda_i$, $c$, divides $|P/T|$, which is a power of $l$. Since $\rho$ is faithful, one of the $\lambda_i$ must be non-trivial on $T[l]$.

By Lemma \ref{corrlemma}, the irreducible representations of $T \cong (\mu_{l^s})^{n-1}$ are given by $\Psi_\mbf{a}$ with $\mbf{a} \in (\Z/l^s\Z)^{n-1}/\Gamma$, where $\Gamma = \text{Gal}(k(\zeta_{l^s})/k)$, and if $\Psi_\mbf{a}$ is non-trivial on $T[l]$, then $\Psi_\mbf{a}$ has dimension $[k(\zeta_{l^s}):k]$. So we must have $\dim(\lambda_i) = [k(\zeta_{l^s}):k]$ for all $i$.

If $c = 1$, then since $\rho|_T = \oplus_{d \text{ times}} \lambda$ is faithful, we must have $\lambda$ is faithful. Recall that $\ed_k(T) = \ed_k((\mu_{l^s})^{n-1}) = (n-1)[k(\zeta_{l^s}):k]$. Since $n = l^t$ and $l \neq 2$, we must have $n > 2$ and so $n - 1 > 1$. Thus there are no $[k(\zeta_{l^s}):k]$-dimensional faithful representations of $T$. But $\dim(\lambda) = [k(\zeta_{l^s}):k]$, so we can conclude that $\lambda$ is not faithful. So we cannot have $c = 1$, and thus since $c$ is a power of $l$ we can conclude that $c$ is a multiple of $l$.   Thus $\dim(\rho) $ is a multiple of $l[k(\zeta_{l^s}):k]$.  So since we know that 
$$(l^t-1)[k(\zeta_{l^s}):k] \leq \dim(\rho) \leq l^t [k(\zeta_{l^s}):k],$$
we can conclude that $\dim(\rho) = l^t [k(\zeta_{l^s}):k]$. Thus 
\begin{align*} \ed_k(SL_{l^t}(\F_q),l)) &= l^t[k(\zeta_{l^s}):k] = \ed_k(GL_{l^t}(\F_q),l).
\qedhere
\end{align*}
\end{proof}

\subsection{\texorpdfstring{Case 2: $l \divides n$, $n \neq l^t$}{Case 2: l divides n, n neq lt}}

\begin{proof}[Proof of Theorem \ref{SLn} for the case $l \divides q-1$, $l \divides n$, $n \neq l^t$] \text{ }

Let $P = \{(\mbf{b}, \tau) \in (\mu_{l^s})^n \rtimes P_l(S_n) : \prod_{i=1}^n b_i = \text{sgn}(\tau)\}$. Let $\rho$ be a faithful representation of $P$ of minimum dimension.  Let $\rho = \oplus_{j=1}^{\xi_l(n)} \rho_j$ be the decomposition into irreducibles. Let $C = Z(P)$. By Lemma \ref{BMKS3.5}, if $\chi_j$ are the central characters of $\rho_j$, then $\{\chi_j|_{C[l]}\}$ form a basis for $\widehat{C[l]}$. Let $\mbf{b}^j$  be the dual basis for $C[l]$ so that $\rho_j(\mbf{b}^i)$ is trivial for $i \neq j$.  

For $j \leq \xi_l(n)$, let
$$P_j = \{(\mbf{b,\tau}) \in P: b_i = 1 \text{ for } i \notin I_j, \text{ } \tau \text{ acts trivially on } i \text{ for } i \notin I_j\}.$$
For $j \leq \xi_l(n)$, define $\mbf{e}^j$ by 
$$(\mbf{e}^j)_i = \begin{cases} \zeta_l, &i \in I_j\\
1, &i \notin I_j \end{cases}.$$ Then $\{\mbf{e}^j\}$ is a basis for $C[l]$.  Write $\mbf{b}^j = \oplus_{i} a_{i,j}\mbf{e}^i$. Then $\rho_j$ will be non-trivial on $P_j \cap C[l]$ if and only if $a_{j,j} \neq 0$. Note that $(a_{i,j})$ is the change of basis matrix from $\{\mbf{e}^j\}$ to $\{\mbf{b}^j\}$. Since it is a change of basis matrix, it must be invertible. By Lemma \ref{claim1}, we can rearrange the $\mbf{b}^j$ such that for all $i$ $a_{i,i} \neq 0$ in the change of basis matrix from $\{\mbf{e}^j\}$ to $\{\mbf{b}^j\}$. And so we can rearrange the $\rho_j$ such that $\chi_j|_{C[l]}$ is non-trivial on $P_j \cap C[l]$ and thus $\rho_j$ is non-trivial on $P_j \cap C[l]$.

Note that 
$P_j$ is isomorphic to a Sylow $l$-subgroup of $SL_{l^{k_j}}(\F_q)$.  And  $P_j \cap C[l]$ is precisely $Z(P_j)[l]$, which has rank $1$.  Thus, since $\rho_j$ is non-trivial on $P_j \cap C[l]$, we can conclude that $\rho_j|_{P_j}$ is a faithful representation of $P_j$.  And we know by the case $n = l^t$ that $\ed_k(SL_{l^{k_j}}(\F_q),l) = l^{k_j}[k(\zeta_{l^s}):k].$ So we can conclude that 
$$\dim(\rho_j) \geq l^{k_j}[k(\zeta_{l^s}):k].$$
Thus
\begin{align*}
\dim(\rho) &= \sum_{j=1}^{\xi_l(n)} \dim(\rho_j)\\
&\geq \sum_{j=1}^{\xi_l(n)} l^{k_j}[k(\zeta_{l^s}):k]\\
&= \left(\sum_{j=1}^{\xi_l(n)} l^{k_j}\right)[k(\zeta_{l^s}):k]\\
&= n[k(\zeta_{l^s}):k]\\
\end{align*}
So
$$\ed_k(SL_n(\F_q),l) \geq n[k(\zeta_{l^s}):k].$$
Thus, since we also have 
$$\ed_k(SL_{n}(\F_q),l) \leq \ed_k(GL_{n}(\F_q),l) = n[k(\zeta_{l^s}):k] $$
we can conclude that 
\begin{align*} \ed_k(SL_n(\F_q),l)) &= n[k(\zeta_{l^s}):k] = \ed_k(GL_n(\F_q),l).
\qedhere
\end{align*}
\end{proof}

\section{The Projective General Linear Groups}

\begin{theorem}\label{PGLn} 
Let $p$ be a prime, $q = p^r$, and $l$ a prime with $l \neq 2,p$.    Let $k$ be a field with $\text{char } k \neq l$. Let $d$ be the smallest positive integer such that $l \divides q^d - 1$, $s = v_l(q^d-1)$, and $n_0 = \lfloor \frac{n}{d} \rfloor$.  Then 
$$\ed_k(PGL_n(\F_q),l) = \begin{cases} \ed_k(GL_n(\F_q),l), &l \nmid q - 1\\
\ed_k(GL_{n-1}(\F_q),l), &l \divides q - 1 \text{ and } l \nmid n\\
l^{v_l(n)}(n-l^{v_l(n)})[k(\zeta_{l^s}):k(\zeta_l)], &l \divides q - 1 \text{ and } l \divides n \text{ and } n \neq l^t\\
l^{2t-1}[k(\zeta_{l^s}):k(\zeta_l)], &l \divides q - 1 \text{ and } n = l^t\\
 \end{cases}$$
\end{theorem}

By (\cite{Gr}, Proposition 1.1), 
$$|PGL_n(\F_q)| = \frac{|GL_n(\F_q)|}{q-1}.$$
So
\begin{align*}
|PGL_n(\F_q)|_l &= \frac{|GL_n(\F_q)|_l}{l^{v_l(q-1)}} = l^{s(n-1)} \cdot |S_{n}|_l\\
\end{align*}

If $l \nmid q - 1$, then the Sylow $l$-subgroups of $PGL_n(\F_q)$ are isomorphic to the Sylow $l$-subgroups of $\GL_n(\F_q)$. So we need only prove the case when $l \divides q - 1$. Thus in this section, we will assume $l \divides q - 1$, and so we have $d = 1$, $s = v_l(q-1)$, and $n_0 = n$ in the notation used in the section on $GL_n(\F_q)$.

\begin{lemma} For $P \in \syl_l(PGL_n(\F_q))$
$$P \cong (\mu_{l^s})^n/\{(x,x,\dots,x)\} \rtimes P_l(S_n).$$ 
where the action of $P_l(S_n)$ on $\mbf{a}$ is given by permuting the $a_i$.
\end{lemma}

\begin{proof}
$PGL_n(\F_q)$ is defined to be 
$$PGL_n(\F_q) = GL_n(\F_q)/Z(GL_n(\F_q)).$$
By Proposition \ref{GLsyl}, the Sylow $l$-subgroups of $GL_n(\F_q)$ are isomorphic to $(\mu_{l^s})^n \rtimes P_l(S_n)$. The center of $GL_n(\F_q)$ is given by 
$$Z(GL_n(\F_q)) = \{x\text{Id}_n : x \in \F_q, x \neq 0\}.$$ 
By looking at the Sylow $l$-subgroup calculated in the section on $GL_n(\F_q)$ and modding by $Z(GL_n(\F_q))$, we see that a Sylow $l$-subgroup of $PGL_n(\F_q)$ will be isomorphic to
$$P = (\mu_{l^s})^n/\{(x,x,\dots,x)\} \rtimes P_l(S_n).$$
\end{proof}

The proof when $l \divides q - 1, l\nmid n$ is simple:

\begin{proof}[Proof of Theorem \ref{PGLn} for the case $l \divides q-1$, $l \nmid n$]

Note that for $l \nmid n$, $P_l(S_n) = P_l(S_{n-1})$.  
Let $P' = (\mu_{l^s})^{n-1} \rtimes P_l(S_{n-1})$. We can construct an isomorphism from $P'$ to $P = (\mu_{l^s})^n/\{(x, \dots, x)\} \rtimes P_l(S_n)$ by sending $(b_1, \dots b_{n-1})$ to $(b_1, \dots, b_{n-1},1) \in P$.
Therefore, the Sylow $l$-subgroups of $PGL_n(\F_q)$ are isomorphic to Sylow $l$-subgroups of $GL_{n-1}(\F_q)$. Thus 
\begin{align*} \ed_k(PGL_n(\F_q),l) &= \ed_k(GL_{n-1}(\F_q),l). \qedhere
\end{align*}
\end{proof}

For the remainder of this section, we will assume that $l \divides n$ (in addition to $l \divides q-1$). By Corollary \ref{rootofunity}, we may assume that $k$ contains a primitive $l$-th root of unity, that is, we may assume that $k = k(\zeta_l)$.

\subsection{\texorpdfstring{Case 1: $l \divides q-1$, $l \divides n$, $n \neq l^t$}{Case 1: l divides q-1, l divides n, n neq lt}}

For the proof of Theorem \ref{PGLn} in the case $l \divides q - 1$, $l \divides n$, $n \neq l^t$, we will need the following lemmas..

\begin{lemma}\label{ZPGLn} For $P \in \syl_l(PGL_n(\F_q))$ in the case $l \divides q - 1$, $l \divides n$, $n \neq l^t$, let $\mbf{b}^j$ be given by 
$$(\mbf{b}^j)_i = \begin{cases} \zeta_l, &i \in I_j\\
1, &i \notin I_j \end{cases}.$$ 
Then 
$$Z(P)[l] = \langle \mbf{b}^j \rangle_{j=1}^{\xi_l(n)}/\{(x,\dots,x)\} \cong \langle \mbf{b}^j \rangle_{j=1}^{\xi_l(n)-1} \cong (\mu_{l})^{\xi_l(n)-1}.$$
\end{lemma}

\begin{proof}
Let $P = (\mu_{l^s})^n/\{(x,\dots,x)\} \rtimes P_l(S_n)$. Fix $(\mbf{b}, \tau) \in P$. Then for $(\mbf{b}', \tau') \in P$,
$$(\mbf{b},\tau)(\mbf{b}',\tau') = (\mbf{b} \tau(\mbf{b}'), \tau\tau') \text{ and } (\mbf{b}', \tau')(\mbf{b},\tau) = (\mbf{b}' \tau'(\mbf{b}), \tau'\tau).$$
Thus $(\mbf{b},\tau)$ is in the center if and only if $\tau \in Z(P_l(S_{n}))$ and 
$$\mbf{b}\tau(\mbf{b}') = \mbf{b}'\tau'(\mbf{b}) \mod \{(x,\dots,x)\}$$ for all $\mbf{b}',\tau'$. Choosing $\tau' = \Id$, we see we must have $\mbf{b}\tau(\mbf{b}') = \mbf{b}'\mbf{b} \mod \{(x,\dots,x)\}$.  Thus we must have $\tau(\mbf{b}') = \mbf{b}' \mod \{(x,\dots,x)\}$ for all $\mbf{b}'$ in $(\mu_{l^s})^n/\{(x,\dots,x)\}$. For any $\tau \neq \text{Id}$, we can choose a $\mbf{b}'$ for which this is not satisfied, so we can conclude that we must have $\tau = \Id$. We also need $\tau(\mbf{b}) = \mbf{b} \mod \{(x,\dots,x)\}$ for all $\tau \in P_l(S_n)$. 

Since $n \neq l^t$, for each $i, i'$ in the same $I_j$, there exists $\tau \in P_l(S_n)$ that sends $i$ to $i'$ and fixes some other index. Since there is an index that is fixed by $\tau$, in order for $\tau(\mbf{b})$ to equal $\mbf{b}\mbf{x}$ for $\mbf{x} = (x,\dots,x)$, we must have $x = 1$ and so $\tau(\mbf{b}) = \mbf{b}$. So $b_i = b_{i'}$ for $i,i'$ in the same $I_j$. Let $\mbf{b}^j$ be given by 
$$(\mbf{b}^j)_i = \begin{cases} \zeta_l, &i \in I_j\\
1, &i \notin I_j \end{cases}.$$ 
Then 
\begin{align*} Z(P)[l] = \langle \mbf{b}^j \rangle_{j=1}^{\xi_l(n)}/\{(x,\dots,x)\} \cong \langle \mbf{b}^j \rangle_{j=1}^{\xi_l(n)-1} &\cong (\mu_{l})^{\xi_l(n)-1}. \qedhere
\end{align*}
\end{proof}

\begin{lemma}\label{gammaaction} For $\Gamma = \text{Gal}(k(\zeta_{l^s})/k)$, $\mbf{a} \in (\Z/l^s\Z)^n$ with $\sum_{i\in I_j} a_i$ invertible for some $j$, 
the orbit of $\mbf{a}$ under the action of $P_l(S_n)$ on $(\Z/l^s\Z)^n/\Gamma$ has the same size as the orbit of $\mbf{a}$ under the action of $P_l(S_n)$ on $(\Z/l^s\Z)^n$.
\end{lemma}
\begin{proof}

We will show that the orbits have the same size by showing that the stabilizers have the same size.  Let $\tau \in P_l(S_n)$ be in the stabilizer of $\mbf{a}$ in $(\Z/l^s\Z)^n/\Gamma$. Then there exists $\phi \in \Gamma$ such that $\tau(\mbf{a}) = \gamma_\phi \mbf{a}$. Let $\gamma = \gamma_\phi$. We want to show that we must then have $\gamma = 1$ (because this would mean that $\tau$ is in the stabilizer under the action of $P_l(S_n)$ on $(\Z/l^s\Z)^n$).

Note that since $\tau$ permutes the $i$ in $I_j$, we have $\sum_{i\in I_j} \tau(\mbf{a})_i = \sum_{i \in I_j} a_i$.  So if $\tau(\mbf{a}) = \gamma\mbf{a}$, we must have $\sum_{i \in I_j} a_i = \gamma\sum_{i \in I_j} a_i$. Thus since $\sum_{i \in I_j} a_i$ is invertible, we can conclude that $\gamma = 1$.
\qedhere
\end{proof}

\begin{lemma}\label{irrH1} For any $l$, let $\mbf{a} = (a_1, \dots, a_{l^{k}})$ with $\sum_{i=1}^{l^{k}} a_i$ invertible. Then 
$$|\text{orbit}(\mbf{a})| \geq l^{k}$$ under the action of $P_l(S_{l^{k}})$ on $(\Z/l^s\Z)^{l^k}$.  
\end{lemma}

\begin{proof}

We will prove the lemma by induction. 

\noindent \textbf{Base Case:} For $k = 1$, we have $\mbf{a} = (a_1, \dots, a_l)$ with $\sum_{i=1}^l a_i$ invertible. If $a_i$ were equal for all $i$, say equal to $x$, then we would have $\sum_{i=1}^l a_i = lx$, which is not invertible. So we can conclude that there exist $j \neq j'$ such that $a_j \neq a_{j'}$. Hence the size of the orbit is at least $2$, so it must be $l$ since it must divide $|P_l(S_l)| = l$.

\noindent \textbf{Induction Step:} Assume that the lemma is true for $k-1$.

For $1 \leq j \leq l$, let $J_j$ denote the $j$th sub-block of $l^{k-1}$ entries in $\{1, \dots, l^k\}$. If $l \divides \sum_{i \in J_j} a_i$ for all $j$, then we would have 
$$l \divides \sum_{j=1}^l \left(\sum_{i \in J_j}a_i\right) = \sum_{i=1}^{l^k} a_i $$
and so $\sum_{i = 1}^{l^k} a_i$ would not be invertible. So we can conclude that there exists a $j$ such that $l \nmid \sum_{i \in J_j} a_i$ and hence $\sum_{i \in J_j} a_i$ is invertible. Without loss of generality, we may assume that $j = 1$ and so $\sum_{i \in J_1} a_i$ is invertible.

Consider the copy of $P_l(S_{l^{k-1}}) \subset P_l(S_{l^k})$ that acts on $J_1$.  Then the orbit of $\mbf{a}$ under the action of $P_l(S_{l^{k-1}}) \subset P_l(S_{l^k})$ will have the same size as the orbit of $(a_1, \dots a_{l^{k-1}})$ under the action of $P_l(S_{l^{k-1}})$. And by the induction hypothesis, the orbit of $(a_1, \dots a_{l^{k-1}})$ under the action of $P_l(S_{l^{k-1}})$ has size at least $l^{k-1}$.  Thus the orbit of $\mbf{a}$ under the action of $P_l(S_{l^k})$ has size at least $l^{k-1}$.

Let $x = \sum_{i \in J_1} a_i$. If $\sum_{i\in J_j} a_i$ were equal to $\sum_{i \in J_1} a_i$ for all $j$, then we would have 
$$\sum_{i=1}^{l^k} a_i = \sum_{j=1}^l \left(\sum_{i \in J_j}a_i\right) = \sum_{j=1}^l x = lx,$$
which is not invertible. So we can conclude that there exist $j \neq 1$ such that $\sum_{i \in J_1} a_i \neq \sum_{i \in J_{j}} a_i$. Without loss of generality, we may assume that $j = l$, and so
$$\sum_{i \in J_1} a_i \neq \sum_{i \in J_l} a_i.$$
Then for $\tau$ the permutation $i \mapsto i+l^{k-1} \mod l^{k}$, we get 
$$\tau(\mbf{a}) = (a_{l^{k-1}+1}, \dots, a_{l^{k}}, a_1, \dots, a_{l^{k-1}}).$$
And since $\sum_{i \in J_1} a_i \neq \sum_{i \in J_l} a_i$, this is not equal to any of the $\tau(\mbf{a})$ for $\tau \in P_l(S_{l^{k-1}})$, where $P_l(S_{l^{k-1}})$ acts on $J_1$.  Thus the size of the orbit is at least $l^{k-1} + 1$, and so it must be at least $l^{k}$ since it must divide $|P_l(S_{l^{k}})|$ which is a power of $l$.
\qedhere
\end{proof}

\begin{corollary}\label{irrH1cor1} Let $\mbf{a} = (a_1, \dots, a_n)$ with $\sum_{i \in I_j} a_i$ invertible and $\sum_{i \in I_{j'}} a_i$ invertible for some $j' \neq j$. Then
 $$|\text{orbit}(\mbf{a})| \geq l^{k_j+v_l(n)}$$
under the action of $P_l(S_n)$ on $(\Z/l^s\Z)^n$. 
\end{corollary}

\begin{proof}
Note that $P_l(S_n)$ is the direct product $\times_j P_l(S_{l^{k_j}})$ acting on $(\Z/l^s\Z)^n = \prod_{j=1}^{\xi_l(n)} (\Z/l^s\Z)^{l^{k_j}}$. So the orbit of $\mbf{a}$ under the action of $P_l(S_n)$ is equal to the product of the orbits of $\mbf{a}$ under the action of $P_l(S_{l^{k_j}})$. Thus it suffices to find the sizes of the orbit of $(a_i)_{i \in I_j}$ under the action of $P_l(S_{l^{k_j}})$ and the orbit of $(a_i)_{i \in I_{j'}}$ under the action of $P_l(S_{l^{k_{j'}}})$ and multiply these to find the minimum size of the orbit of $\lambda_{\mbf{a}}$ under the action of $P_l(S_n)$.  By Lemma \ref{irrH1}, these orbits will have size at least $l^{k_j}$ and $l^{k_{j'}}$ respectively. Thus the orbit of $\mbf{a}$ will have size at least $l^{k_j + k_{j'}}$. And for all $j'$, $k_{j'} \geq v_l(n)$; so $l^{k_j + k_{j'}} \geq l^{k_j + k_{v_l(n)}}$. Hence $|\text{orbit}(\mbf{a})| \geq l^{k_j+v_l(n)}$. 
\qedhere
\end{proof}

\begin{corollary}\label{irrH1cor} For $\mbf{a} = (a_1, \dots, a_n)$ with $\sum_{i \in I_j} a_i$ invertible and $\sum_{i \in I_{j'}} a_i$ invertible for some $j' \neq j$, we can conclude that $$|\text{orbit}(\Psi_{\mbf{a}})| \geq l^{k_j+v_l(n)}$$
under the action of $P_l(S_n)$ on $\text{Irr}((\mu_{l^s})^n)$.
\end{corollary}
\begin{proof}
By Lemma \ref{changepersp}, the orbit of $\Psi_{\mbf{a}}$ under the action of $P_l(S_n)$ on $\text{Irr}((\mu_{l^s})^n)$ has the same size as the orbit of $\mbf{a}$ under the action of $P_l(S_n)$ on $(\Z/l^s\Z)^n/\Gamma$. By Lemma \ref{gammaaction}, the orbit in $(\Z/l^s\Z)^n/\Gamma$ is the same as the orbit in $(\Z/l^s\Z)^n$. And by Corollary \ref{irrH1cor1}, the orbit of $\mbf{a}$ has size greater than or equal to $l^{k_j+v_l(n)}$.
\end{proof}

We can now complete the proof in the case $l \divides q - 1$, $l \divides n$, $n \neq l^t$.

\begin{proof}[Proof of Theorem \ref{PGLn} for the case $l \divides q - 1$, $l \divides n$, $n \neq l^t$.]

Recall that $P = (\mu_{l^s})^n/\{(x,x,\dots,x)\} \rtimes P_l(S_n)$. Let $\rho$ be a faithful representation of $P$ of minimum dimension.  Let $\rho = \oplus_{j=1}^{\xi_l(n)-1} \varphi_j$ be the decomposition into irreducibles. Let $C = Z(P)$. By Lemma \ref{BMKS3.5}, if $\chi_j$ are the central characters of $\varphi_j$, then $\{\chi_j|_{C[l]}\}$ form a basis for $\widehat{C[l]}$. Let $\mbf{b}^j$  be the dual basis for $C[l]$ so that $\varphi_j(\mbf{b}^i)$ is trivial for $i \neq j$.  

Let $S' = (\mu_{l^s})^n/\{(x,\dots,x)\}$. For $j \leq \xi_l(n)$, let
$$S_j = \{\mbf{b} \in S': b_i = 1 \text{ for } i \notin I_j\}.$$
For $j \leq \xi_l(n)$, define $\mbf{e}^j \in S'$ by 
$$(\mbf{e}^j)_i = \begin{cases} \zeta_l, &i \in I_j\\
1, &i \notin I_j \end{cases}.$$ Then $\{\mbf{e}^j\}$ is a basis for $C[l]$.  Write $\mbf{b}^j = \oplus_{i} a_{i,j}\mbf{e}^i$. Then $\varphi_j$ will be non-trivial on $S_j \cap C[l]$ if and only if $a_{j,j} \neq 0$. Note that $(a_{i,j})$ is the change of basis matrix from $\{\mbf{e}^j\}$ to $\{\mbf{b}^j\}$. Since it is a change of basis matrix, it must be invertible. By Lemma \ref{claim1}, we can rearrange the $\mbf{b}^j$ such that for all $i$ $a_{i,i} \neq 0$ in the change of basis matrix from $\{\mbf{e}^j\}$ to $\{\mbf{b}^j\}$. And so we can rearrange the $\rho_j$ such that $\chi_j|_{C[l]}$ is non-trivial on $S_j \cap C[l]$ and thus $\varphi_j$ is non-trivial on $S_j \cap C[l]$.

Fix $j \leq \xi_l(n)-1$ and let $\varphi = \varphi_j$. By Clifford's Theorem (Theorem \ref{cliff}), $\varphi|_{S'}$ decomposes into a direct sum of irreducibles in the following manner:
$$\varphi|_{S'} \cong  \left( \oplus_{i=1}^c  \lambda_i \right)^{\oplus d}, \text{ for some } c, d,$$ 
with the $\lambda_i$ non-isomorphic, and $P_l(S_{n})$ acts transitively on the isomorphism classes of the $\lambda_i$, so the $\lambda_i$ have the same dimension and the number of $\lambda_i$, $c$, divides $|P_l(S_{n})|$. Since $\varphi$ is non-trivial on $S_j \cap C[l] \subset S'[l]$, one of the $\lambda_i$ must be non-trivial on $S_j \cap C[l]$. Without loss of generality assume the $\lambda_1$ is non-trivial on $S_j \cap C[l]$.

Note that the irreducible representations of $S'$ are in bijection with irreducible representations of $(\mu_{l^s})^n$ which are trivial on $\{(x,\dots,x)\}$.  By Lemma \ref{corrlemma}, the irreducible reprsentations of $(\mu_{l^s})^n$ are given by $\Psi_{\mbf{a}}$ with $\mbf{a} \in (\Z/l^s\Z)^n/\Gamma$, for $\Gamma = \text{Gal}(k(\zeta_{l^s})/k)$, and if $\Psi_{\mbf{a}}$ is non-trivial on $(\mu_{l^s})^n[l]$, then $\Psi_\mbf{a}$ has dimension $[k(\zeta_{l^s}):k]$.  So since $\lambda_1$ is non-trivial on $S_j \cap C[l] \subset (\mu_{l^s})^n[l]$, we must have $\dim(\lambda_1) = [k(\zeta_{l^s}):k]$, and so $\dim(\lambda_i) = [k(\zeta_{l^s}):k]$ for all $i$. 

Note that for $\mbf{x} = (x,\dots,x)$, $\psi_{\mbf{a}}(\mbf{x}) = x^{\sum_{i=1}^n a_i}$.  So $\Psi_\mbf{a}$ will be trivial on $\{(x,\dots,x)\}$ if and only if $\sum_{i=1}^n a_i = 0$. So $\lambda_1 \cong \Psi_{\mbf{a}}$ for some $\mbf{a} \in (\Z/l^s\Z)^n/\Gamma$ with $\sum_{i=1}^n a_i = 0$.

Also, since $\lambda_1 \cong \Psi_{\mbf{a}}$ is non-trivial on $S_j \cap C[l]$, we must have
$$\text{ }\sum_{i \in I_j} a_il^{s-1} \neq 0.$$ 
Since $\sum_{i=1}^n a_i = 0$, we must have $\sum_{i=1}^n a_il^{s-1} = 0$; hence since $\sum_{i \in I_j} a_il^{s-1} \neq 0$, we must also have $\sum_{i \in I_{j'}} a_il^{s-1} \neq 0$ for some $j' \neq j$. Therefore, we must have $\sum_{i \in I_j} a_i$ invertible and $\sum_{i \in I_{j'}} a_i$ invertible for some $j' \neq j$. So by Corollary \ref{irrH1cor}, the orbit of $\lambda_i = \Psi_{\mbf{a}}$ under the action of $P_l(S_n)$ has size at least $l^{k_j + v_l(n)}$. So $c \geq l^{k_j + v_l(n)}$. Thus for $\lambda = \rho_j$,
$$\dim(\lambda) \geq l^{k_j + v_l(n)}[k(\zeta_{l^s}):k].$$
Hence
\begin{align*}
\dim(\rho) &= \sum_{j=1}^{\xi_l(n)-1} \dim(\rho_j)\\
&\geq \sum_{j=1}^{\xi_l(n)-1} l^{k_j + v_l(n)}[k(\zeta_{l^s}):k]\\
&= l^{v_l(n)}\left(\sum_{j=1}^{\xi_l(n)-1} l^{k_j}\right)[k(\zeta_{l^s}):k]\\
&= l^{v_l(n)}(n-l^{v_l(n)})[k(\zeta_{l^s}):k]
\end{align*}

We can construct a faithful representation of this dimension in the following manner.  For each $j \leq \xi_l(n)-1$, let
$$\lambda_j:  S' \to GL(k(\zeta_{l^s}))$$ be the irreducible representation of dimension $[k(\zeta_{l^s}):k]$ given by $\lambda_j = \Psi_{\mbf{a}}$, where $a_{i_j} = 1$, $a_n = -1$ and all other $a_i$ are $0$.

The stabilizer of $\mbf{a}$ under the action of $P_l(S_n)$ is given by those permutations which fix $I_j$ and $I_{\xi_l(m)}$. There are $\frac{|S_n|_l}{l^{k_j+v_l(n)}}$ such permutations; thus the orbit of $\mbf{a}$ has size $l^{k_j+v_l(n)}$. So the orbit of $\Psi_{\mbf{a}}$ under the action of $P_l(S_{n})$ on the irreducible representations (not isomorphism classes) of $S'$ has size $l^{k_j+v_l(n)}$. Let $\text{Stab}_j$ be the stabilizer of $\lambda_j$ in $P_l(S_{n})$ (which has order $\frac{|P_l(S_{n})|}{l^{k_j + v_l(n)}}$). We can extend $\lambda_j$ to $S' \rtimes \text{Stab}_j$ by defining $\lambda_j(\mbf{b},\tau) = \tau_{\lambda_j}(\mbf{b}) = \lambda_j(\mbf{b})$ (since $\tau \in \text{Stab}_j$). Let $\rho_j = \text{Ind}_{S' \rtimes \text{Stab}_j}^P \lambda_j$. Then $\rho_j$ has dimension 
$$[P_l(S_{n}): \text{Stab}_j]\dim(\lambda_j) = l^{k_j + v_l(n)}[k(\zeta_{l^s}):k],$$ and $\rho_j$ is non-trivial on $S_j \cap C[l]$. Then let $\rho = \oplus_{j=1}^{\xi_l(n)-1} \rho_j$.  Then $\rho$ has dimension $l^{v_l(n)}(n-l^{v_l(n)})[k(\zeta_{l^s}):k]$, and by Lemma \ref{BMKS3.4}, $\rho$ is faithful.
Thus we have shown that for $n \neq l^t$,
\begin{align*}
\ed_k(PGL_{n}(\F_q),l)) &= l^{v_l(n)}(n-l^{v_l(n)})[k(\zeta_{l^s}):k].
\qedhere
\end{align*}
\end{proof}

  \subsection{\texorpdfstring{Case 2: $l \divides q-1$, $n = l^t$}{Case 2: l divides q-1, n = lt}}
 
 \begin{definition} For $1 \leq j \leq l$, let $J_j$ denote the $j$th sub-block of $l^{k-1}$ entries in $\{1, \dots, l^k\}$. Let $A_j = \sum_{i \in J_j} a_i$. \end{definition}
 
For the proof of Theorem \ref{PGLn} in the case $l \divides q-1$, $n = l^t$, we will need the following lemmas.
 
 \begin{lemma} For $P \in \syl_n(PGL_l(\F_q))$ in the case $l \divides q - 1$, $n = l^t$ 
$$Z(P) \cong \langle (1, \dots, 1, \zeta_l, \dots, \zeta_l, \zeta_l^2, \dots, \zeta_l^2, \dots, \zeta_l^{(l-1)}, \dots, \zeta_l^{(l-1)}) \rangle \cong \mu_l.$$
\end{lemma}

\begin{proof}
Just as in the case $n \neq l^t$, in order for $(\mbf{b}, \tau)$ to be in the center we must have $\tau = \text{Id}$ and $\tau(\mbf{b}) = \mbf{b} \mod \{(x,x,\dots,x)\}$ for all $\tau \in P_l(S_n)$.  

Note that for each $i, i'$ in the same $J_j$, there exists $\tau \in P_l(S_n)$ that sends $i$ to $i'$ and fixes some other index. Since there is an index that is fixed by $\tau$, in order for $\tau(\mbf{b})$ to equal $\mbf{b}\mbf{x}$ for $\mbf{x} = (x,x,\dots,x)$, we must have $x = 1$ and so $\tau(\mbf{b}) = \mbf{b}$. So $b_1 = \dots = b_{l^{t-1}}$, $b_{l^{t-1}+1} = \dots = b_{2l^{t-1}}$, $\dots$, $b_{l^t-l^{t-1}+1} = \dots = b_{l^t}.$  If we consider the last generator, $\sigma_1^{t}$, we see that we must have $b_{i+l^{t-1}} = b_{i}x$ for some fixed $x = \zeta_l^{a}$.  Thus $\mbf{b}$ must be of the form
$$\mbf{b} = (b, \dots, b, b\zeta_l^a, \dots b\zeta_l^a, \dots, b\zeta_l^{(l-1)}, \dots, b\zeta_l^{(l-1)}).$$
In $PGL_{l^t}(\F_q)$, the set of all elements of this form is a cyclic group of order $l$ generated by
$$\mbf{b} = (1, \dots, 1, \zeta_l, \dots, \zeta_l, \zeta_l^2, \dots, \zeta_l^2, \dots, \zeta_l^{(l-1)}, \dots, \zeta_l^{(l-1)}).$$
So we have 
\begin{align*} Z(P) &\cong \mu_l.
\qedhere
\end{align*}
\end{proof}

\begin{lemma}\label{gammaaction2} Let $l \divides q - 1$, $l \neq 2$, $\Gamma = \text{Gal}(k(\zeta_{l^s})/k),$ $n = l^t$, and $\mbf{a} \in (\Z/l^s\Z)^n$ with $$\sum_{i=1}^{n} a_i = 0 \text{ and } A_j \text{ invertible for some } j.$$ 
Then the orbit of $\mbf{a}$ under the action of $P_l(S_n)$ on $(\Z/l^s\Z)^n/\Gamma$ has the same size as the orbit of $\mbf{a}$ under the action of $P_l(S_n)$ on $(\Z/l^s\Z)^n$.
\end{lemma}

\begin{proof}[Proof of Lemma \ref{gammaaction2}]

If $s = 1$, then $\Gamma$ is trivial; thus $(\Z/l^s\Z)^n/\Gamma = (\Z/l^s\Z)^n$ and the lemma is obvious. So it suffices to prove the lemma for $s \neq 1$. We will show that the orbits have the same size by showing that the stabilizers have the same size.     Let $\tau \in P_l(S_n)$ be in the stabilizer of $\mbf{a}$ in $(\Z/l^s\Z)^n/\Gamma$. Then there exists $\phi \in \Gamma$ such that $\tau(\mbf{a}) = \gamma_\phi \mbf{a}$. Let $\gamma = \gamma_\phi$. We want to show that we must then have $\gamma = 1$ (because this would mean that $\tau$ is in the stabilizer under the action of $P_l(S_n)$ on $(\Z/l^s\Z)^n$).

Without loss of generality, we may assume that $A_1$ is invertible. Note that $\tau$ permutes the $J_j$, and the permutation is given by $j \mapsto j - y \mod l$ for some $y \in \Z/l\Z$. So $\sum_{i \in J_{1+y}} a_i = \gamma\sum_{i \in J_1} a_i = \gamma A_1$. 

Suppose by way of contradiction that $y \neq 0$. Then in general $\sum_{i \in J_j} a_i = \gamma^{k}A_1$, where $j = 1+ky \mod l$.  So the collection of $\sum_{i \in J_j} a_i$ for $j \in \Z/l\Z$ is the same as the collection of $\gamma^{k}A_1$ for $k \in \Z/l\Z$. 
Then we have 
$$0 = \sum_{i=1}^n a_i = \sum_{j=1}^l \left( \sum_{i \in J_j} a_i \right) = \sum_{k=0}^{l-1} \gamma^{k}A_1 = A_1 \sum_{k=0}^{l-1}\gamma^{k}.$$
So since $A_1$ is invertible, we can conclude that 
$$0 = \sum_{k=0}^{l-1}\gamma^{k}.$$
Furthermore, note that since $1 = 1 + ly \mod l$, we will have $\sum_{i \in J_1} a_i = \gamma^l \sum_{i \in J_1} a_i$, and hence since $\sum_{i \in J_1} a_i$ is invertible, we can conclude that $\gamma^l = 1$. So $|\gamma| \divides l$. 

Note that since $l \neq 2$, $(\Z/l^s\Z)^\times$ is cyclic and hence $\Gamma \hookrightarrow (\Z/l^s\Z)^\times$ is cyclic. Thus $\Gamma$ has a unique subgroup of order $l$; this subgroup consists of multiplication by $1 + b l^{s-1}$ for $0 \leq b < l$. If $\gamma = 1$, then $\sum_{k=1}^{l-1} \gamma^k = l$. If $\gamma \neq 1$, then since $|\gamma| \divides l$, $\gamma$ generates exactly this subgroup of order $l$. Thus we have
\begin{align*}
0 &= \sum_{k=0}^{l-1} \gamma^k\\
&= \sum_{b = 0}^{l-1} 1+bl^{s-1}\\
&= l + l^{s-1}\sum_{b=0}^{l-1} b\\
&= l + l^{s-1}\frac{(l-1)l}{2}\\
&= l \qquad \qquad \qquad \qquad \text{ in } \Z/l^s\Z \text{ since } l \neq 2
\end{align*}
So for any $\gamma$ with $|\gamma| \divides l$, we will have $0 = \sum_{k=0}^{l-1} \gamma^k = l$. This is a contradiction since $s \neq 1$. Therefore we can conclude that $y = 0$. So $\tau$ stabilizes the sub-blocks $J_j$. In particular, it stabilizes $J_1$; hence 
$$\sum_{i \in J_1} a_i = \gamma\sum_{i \in J_1} a_i.$$
Thus since $\sum_{i \in J_1} a_i = A_1$ is invertible, we must have $\gamma = 1$.
\qedhere
\end{proof}

\begin{lemma}\label{irrH2} Let $l \divides q - 1$, $l \neq 2$, $n = l^t$, and $\mbf{a} \in (\Z/l^s\Z)^n$ with 
$$\sum_{i=1}^{n} a_i = 0, \sum_{j=1}^{l} (j-1)A_j \text{ invertible, and }A_j \text{ invertible for some }j.$$ 
Then
$$|\text{orbit}(\mbf{a})| \geq l^{2t-1}$$ under the action of $P_l(S_n)$ on $(\Z/l^s\Z)^n$. 
\end{lemma}

\begin{lma}[Lemma \ref{irrH2}] Let $l \divides q - 1$, $l \neq 2$, $n = l^t$, and $\mbf{a} \in (\Z/l^s\Z)^n$ with 
$$\sum_{i=1}^{n} a_i = 0, \sum_{j=1}^{l} (j-1)A_j \text{ invertible, and }A_j \text{ invertible for some }j.$$ 
Then
$$|\text{orbit}(\mbf{a})| \geq l^{2t-1}$$ under the action of $P_l(S_n)$ on $(\Z/l^s\Z)^n$. 
\end{lma}

\begin{proof}[Proof of Lemma \ref{irrH2}]

Since $A_j$ is invertible and $\sum_{i=1}^n a_i = 0$ is not invertible, we must also have $A_{j'}$ invertible for some $j' \neq j$. So by Lemma \ref{irrH1}  we can conclude that that the orbit of $\mbf{a}$ under the action of $P_l(S_{l^{t-1}}) \times P_l(S_{l^{t-1}}) \subset P_l(S_n)$ has size at least $l^{t-1} \cdot l^{t-1} = l^{2t-2}$. 

Without loss of generality, we may assume that $j = 1$ and so we have $A_1$ is invertible. If $A_j$ were equal to $A_1$ for all $j$, then we would have 
$$\sum_{j=1}^{l} (j-1)A_j = \sum_{j=1}^{l} (j-1)A_1 = \frac{l(l-1)}{2}A_1,$$
which is divisible by $l$ (since $l$ is odd) and so is not invertible. So we can conclude that there exist $j' \neq 1$ such that $A_1 \neq A_{j'}$. Without loss of generality, we may assume that $j' = l$, and so $A_1 \neq A_l.$

Then for $\tau$ the permutation $i \mapsto i+l^{t-1} \mod l^{t}$, we get 
$$\tau(\mbf{a}) = (a_{l^{t-1}+1}, \dots, a_{l^{t}}, a_1, \dots, a_{l^{t-1}}).$$
And since $\sum_{i \in J_1} a_i \neq \sum_{i \in J_l} a_i$, this is not equal to any of the $\sigma(\mbf{a})$ for $\sigma \in P_l(S_{l^{t-1}}) \times P_l(S_{l^{t-1}})$.  Thus the size of the orbit is at least $l^{2t-2} + 1$, and so it must be at least $l^{2t-1}$ since it must divide $|P_l(S_{l^{t}})|$ which is a power of $l$.

For $\mbf{a} = (1,0,\dots,-1)$, the stabilizer is given by those permutations which fix $1$ and $n$, there are $\frac{|S_n|_l}{l^{2t-1}}$ such permutations; thus the orbit has size $l^{2t-1}$. 
\qedhere
\end{proof}

Putting the above two lemmas together, we obtain the following corrolary:

\begin{corollary}\label{irrH2cor} For $l \divides q - 1$, $l \neq 2$, $n = l^t$, $\mbf{a} \in (\Z/l^s\Z)^n$ with 
$$\sum_{i=1}^{n} a_i = 0, \sum_{j=1}^{l} (j-1)A_j \text{ invertible, and }A_j \text{ invertible for some }j,$$ we can conclude that
$$|\text{orbit}(\Psi_{\mbf{a}})| \geq l^{2t-1}$$ 
under the action of $P_l(S_n)$ on $\text{Irr}((\mu_{l^s})^n)$.
\end{corollary}

\begin{proof}
By Lemma \ref{changepersp}, the orbit of $\Psi_{\mbf{a}}$ under the action of $P_l(S_n)$ on $\text{Irr}((\mu_{l^s})^n)$ has the same size as the orbit of $\mbf{a}$ under the action of $P_l(S_n)$ on $(\Z/l^s\Z)^n/\Gamma$. By Lemma \ref{gammaaction2}, the size of the orbit in $(\Z/l^s\Z)^n/\Gamma$ is the same as size of the orbit in $(\Z/l^s\Z)^n$. And by Lemma \ref{irrH2}, the orbit will have size at least $l^{2t-1}$. 
\end{proof}

We can now complete the proof in the case $l \divides q-1$, $n = l^t$.

\begin{proof}[Proof of Theorem \ref{PGLn} for the case $l \divides q-1$, $n = l^t$.]

Recall that $P \cong (\mu_{l^s})^n/\{(x,x,\dots,x)\} \rtimes P_l(S_n).$  Let $\rho$ be a faithful representation of $P$ of minimum dimension (and so it is also irreducible since the center has rank $1$.)   Let $S' = (\mu_{l^s})^n/\{(x,\dots,x)\}$. By Clifford's Theorem (Theorem \ref{cliff}), $\rho|_{S'}$ decomposes into a direct sum of irreducibles in the following manner:
$$\rho|_{S'} \cong  \left( \oplus_{i=1}^c  \lambda_i \right)^{\oplus d}, \text{ for some } c, d,$$ 
with the $\lambda_i$ non-isomorphic, and $P_l(S_n)$ acts transitively on the $\lambda_i$, so the $\lambda_i$ have the same dimension and the number of $\lambda_i$, $c$, divides $|P_l(S_n)|$ (which is a power of $l$), so $c$ is a power of $l$. Also, since $\rho$ is faithful, it is non-trivial on $Z(P)$, thus one of the $\lambda_i$ must be non-trivial on $Z(P) \subset (\mu_{l^s})^n[l]$.   Without loss of generality assume the $\lambda_1$ is non-trivial on $Z(P)$.

Note that the irreducible representations of $S'$ are in bijection with irreducible representations of $(\mu_{l^s})^n$ which are trivial on $\{(x,\dots,x) : x^n = 1\}$.  By Lemma \ref{corrlemma}, the irreducible reprsentations of $(\mu_{l^s})^n$ are given by $\Psi_{\mbf{a}}$ with $\mbf{a} \in (\Z/l^s\Z)^n/\Gamma$, for $\Gamma = \text{Gal}(k(\zeta_{l^s})/k)$, and if $\Psi_{\mbf{a}}$ is non-trivial on $(\mu_{l^s})^n[l]$, then $\Psi_\mbf{a}$ has dimension $[k(\zeta_{l^s}):k]$.  Since $\lambda_1$ is non-trivial on $Z(P) \subset (\mu_{l^s})^n[l]$, we must have $\dim(\lambda_1) = [k(\zeta_{l^s}):k]$, and so $\dim(\lambda_i) = [k(\zeta_{l^s}):k]$ for all $i$. 

And $\Psi_\mbf{a}$ will be trivial on $\{(x,\dots,x) : x^n = 1\}$ if and only if $\sum_{i=1}^n a_i = 0$. So $\lambda_1 \cong \Psi_{\mbf{a}}$ for some $\mbf{a} \in (\Z/l^s\Z)^n/\Gamma$ with $\sum_{i=1}^n a_i$. 

Recall that $J_j$ denotes the $j$th sub-block of $l^{t-1}$ entries in $\{1, \dots, l^t\}$. Since $\lambda_i$ is non-trivial on 
$$Z(P) = \langle (1, \dots, 1, \zeta_l, \dots, \zeta_l, \zeta_l^2, \dots, \zeta_l^2, \dots, \zeta_l^{(l-1)}, \dots, \zeta_l^{(l-1)})\rangle,$$ we must have that 
$$0 \neq \sum_{j=1}^{l} ((j-1)\sum_{i \in J_j} a_il^{s-1}) = \sum_{j=1}^l (j-1) A_jl^{s-1}.$$
Thus $\sum_{j=1}^{l} (j-1)A_j$ is invertible, and so we must have $A_j$ invertible for some $j$.  So by Corollary \ref{irrH2cor}, the orbit of $\lambda_1$ under the action of $P_l(S_{n})$ will have size at least $l^{2t-1}$. So $c \geq l^{2t-1}$. Thus $\dim(\rho) \geq l^{2t-1}[k(\zeta_{l^s}):k]$.

We can construct a faithful representation of this dimension in the following manner.  Let $\mbf{a} = (1, 0, \dots, -1)$ and consider
$$\Psi_{\mbf{a}}:  S' \to GL(k(\zeta_{l^s})) = GL_d(k),$$  
where $d = [k(\zeta_{l^s}):k]$. The stabilizer of $\mbf{a}$ under the action of $P_l(S_n)$ is given by those permutations which fix $1$ and $n$. There  are $\frac{|S_n|_l}{l^{2t-1}}$ such permutations; thus the orbit of $\mbf{a}$ has size $l^{2t-1}$. So the orbit of $\Psi_{\mbf{a}}$ under the action of $P_l(S_{n})$ on the irreducible representations (not isomorphism classes) of $S'$ has size $l^{2t-1}$. Let $\text{Stab}_{\mbf{a}}$ be the stabilizer of $\Psi_{\mbf{a}}$ in $P_l(S_{n})$ (which has order $\frac{|P_l(S_{n})|}{l^{2t-1}}$). We can extend $\Psi_{\mbf{a}}$ to $S' \rtimes \text{Stab}_{\mbf{a}}$ by defining $\Psi_{\mbf{a}}(\mbf{b},\tau) = \tau_{\Psi_{\mbf{a}}}(\mbf{b}) = \Psi_{\mbf{a}}(\mbf{b})$ (since $\tau \in \text{Stab}_{\mbf{a}}$). Let $\rho = \text{Ind}_{S' \rtimes \text{Stab}_{\mbf{a}}}^P \Psi_{\mbf{a}}$. Then $\rho$ has dimension 
$$[P_l(S_{n}): \text{Stab}_{\mbf{a}}]\dim(\Psi_{\mbf{a}}) = l^{2t-1}[k(\zeta_{l^s}):k],$$ and $\rho$ is non-trivial (and hence faithful) on $Z(P)$. So this is a faithful representation of $P$ of dimension  $l^{2t-1}[k(\zeta_{l^s}):k]$.
Thus we have shown that for $n = l^t$,
\begin{align*}
\ed_k(PGL_{n}(\F_q),l)) &= l^{2t-1}[k(\zeta_{l^s}):k].
\qedhere
\end{align*}
\end{proof}

\section{\texorpdfstring{The Projective Special Linear Groups and Quotients of $SL_n(\F_q)$ by cyclic subgroups of the center}{The Projective Special Linear Groups and Quotients of SLn(Fq) by cyclic subgroups of the center}}

$PSL_n(\F_q)$ is defined to be 
$$PSL_n(\F_q) = SL_n(\F_q)/Z(SL_n(\F_q)).$$
The center of $SL_n(\F_q)$ is given by 
$$Z(SL_n(\F_q)) = \{x\text{Id}_n : x \in \F_q, x^n = 1\}.$$ 
By looking at the Sylow $l$-subgroup calculated in the section on $SL_n(\F_q)$ and modding by $Z(SL_n(\F_q))$, we see that a Sylow $l$-subgroup of $PSL_n(\F_q)$ is isomorphic to
$$P \cong \{(\mbf{b}, \tau) \in (\mu_{l^s})^n/ \{(x,\dots,x) : x^n = 1\} \rtimes P_l(S_n) : \prod_{i=1}^n b_i = \text{sgn}(\tau)\}.$$

Note that for $n' | n$, we obtain a subgroup of $SL_n(\F_q)$ containing  $PSL_n(\F_q)$ of order $\frac{|SL_n(\F_q)|}{(n', \text{ }q-1)}$ by taking the quotient of $SL_n(\F_q)$ by the cyclic subgroup of order $n'$ given by $\{x\text{Id} : x \in \F_q, \text{ } x^{n'} = 1\}$. The Sylow $l$-subgroups will be given by

$$P \cong \{(\mbf{b}, \tau) \in (\mu_{l^s})^n/ \{(x,\dots,x) : x^{n'} = 1\} \rtimes P_l(S_n) : \prod_{i=1}^n b_i = \text{sgn}(\tau)\}.$$

\begin{theorem}\label{PSLn} Let $p$ be a prime, $q = p^r$, and $l$ a prime with $l \neq 2,p$.  Let $k$ be a field with $\text{char } k \neq l$. Let $s = v_l(q-1)$, $n' \divides n$, and let $G = SL_n(\F_q)/\{x\text{Id} : x \in \F_q, \text{ } x^{n'} = 1\}$. Let $v = \min(v_l(n'),s)$. Then if $l \nmid q-1$ or $l \nmid n'$, then $\ed_k(G,l) = \ed_k(SL_n(\F_q),l)$. And if $l \divides q-1$, $l \divides n'$, then 
$$\ed_k(G,l) = \begin{cases} 
2, &l = n' = n = 3, \text{ } s = 1\\
\ed_k(PGL_n(\F_q),l), &\text{otherwise }
\end{cases}.$$
Note that for $n' = n,$ $G = PSL_n(\F_q).$
\end{theorem}

\smallskip

If $l \nmid q - 1$ or $l \nmid n'$, then the Sylow $l$-subgroups of $G$ are isomorphic to the Sylow $l$-subgroups of $SL_n(\F_q)$. So we need only prove the theorem when $l \divides q-1,$ $l \divides n'$. Thus in this section, we will assume $l \divides q - 1$ and $l \divides n'$ (and hence $l \divides n$ since $n' \divides n$). And so we have $d = 1$, $s = v_l(q-1)$, and $n_0 = n$ in the notation used in the section on $GL_n(\F_q)$. By Corollary \ref{rootofunity}, we may assume that $k$ contains a primitive $l$-th root of unity, that is, we may assume that $k = k(\zeta_l)$. 

We will need the following lemma.

\begin{lemma}\label{Helper1}
Let $l \neq 2$, $\Gamma = \text{Gal}(k(\zeta_{l^s})/k)$, $\gamma \in \Gamma$, $\mbf{a} \in (\Z/l^s\Z)^l$ with $a_l$ invertible and
$$a_i = \gamma^i a_l + x \sum_{k=1}^i \gamma^{k-1} \text{ } \forall i.$$
 Then $\sum_{i=1}^l ia_i$ is not invertible.
\end{lemma}

\begin{proof}
Write $\gamma = c + bl$ for $l \nmid c$. Then
\begin{align*}
a_l &= \gamma^l a_l + x\sum_{k=1}^l \gamma^{k-1}\\
&= c^l a_l + x\sum_{k=1}^l c \mod l\\
&= c^l a_l \mod l
\end{align*}
So we must have $c^l = 1 \mod l$. Thus $\gamma^l = 1 \mod l$. Hence $0 = \gamma^l - 1 = (\gamma-1)^l \mod l.$ Therefore $\gamma = 1 \mod l$ and so $c = 1$. Then 
\begin{align*}
a_{i+1 \text{ mod } l} - a_i &= (\gamma^{i+1}a_l + x\sum_{k=1}^{i+1} \gamma^{k-1}) - (\gamma^i a_l + x\sum_{k=1}^i \gamma^{k-1})\\
&= a_l(\gamma^{i+1}-\gamma^i) + x\gamma^i\\
&= a_l((1+bl)^{i+1} - (1+bl)^i) + (1+bl)^{i} x\\
&= a_l(1+bl)^i(1+bl-1) + (1+bl)^ix\\
&=a_lbl(1+bl)^i + (1+bl)^{i} x\\
&=(1+bl)^i(a_lbl+x)
\end{align*}
 Note that $\sum_{i=1}^l a_{i+1 \text{ mod } l} = \sum_{i=1}^l a_i$. So
\begin{align*}
0 &= \sum_{i=1}^l a_{i+1 \text{ mod } l} - a_i\\
&= \sum_{i=1}^l (1+bl)^i(a_lbl+x)\\
&= (a_lbl+x)\sum_{i=1}^l (1+bl)^i\\
&= (a_lbl+x)\sum_{i=1}^l (1 + ibl)  \mod l^2\\
&= (a_lbl+x)l + bl\frac{l(l+1)}{2} \mod l^2\\
&= lx \mod l^2 \text{ since } l \neq 2
\end{align*}
So we can conclude that $l \divides x$. Then
\begin{align*}
\sum_{i=1}^l ia_i &= \sum_{i=1}^l i\left(\gamma^i a_l + x\sum_{k=1}^i \gamma^{k-1}\right)\\
&= a_l\sum_{i=1}^l i \mod l, \text{ since } l \divides x \text{ and } \gamma = 1 \mod l\\
&= a_l\frac{l(l+1)}{2} \mod l\\
&= 0 \mod l \text{ since } l \neq 2
\end{align*}
Therefore, $\sum_{i=1}^l ia_i$ is not invertible.
\end{proof} 

\subsection{\texorpdfstring{The irreducible representations of $T = \{\mbf{b} \in (\mu_{l^s})^n : \prod_{i=1}^n b_i = 1\}$}{The irreducible representations of T = b in mulsn : prodi=1n bi = 1}}

\begin{remark} Let $T = \{\mbf{b} \in (\mu_{l^s})^n : \prod_{i=1}^n b_i = 1\}$. Note that the map given by $\mbf{a} \mapsto \Psi_{\mbf{a}}|_T$ gives an isomorphism between $(\Z/l^s\Z)^n/\{(x,\dots,x)\}$ and $\widehat{T}$. 
\end{remark}

\begin{lemma}\label{auxlemmaPSLn} For any prime $l$, let  $\Gamma = \text{Gal}(k(\zeta_{l^s})/k)$. Let $T = \{\mbf{b} \in (\mu_{l^s})^n : \prod_{i=1}^n b_i = 1\}$ and let  $H = (\Z/l^s\Z)^{n}/\{(x,\dots,x)\}$. Consider the action of $\Gamma$ on $\widehat{T}$ given by $\phi(\lambda) = \phi \circ \lambda$ for $\lambda \in \widehat{T}$. Then the corresponding action of $\gamma_\phi \in (\Z/l^s\Z)^\times$ on $H = (\Z/l^s\Z)^n/\{(x,\dots,x)\} \cong \widehat{T}$ is given by scalar multiplication by $\gamma_\phi$. \end{lemma}

\begin{proof}

Let $\phi \in \Gamma,$  $\gamma = \gamma_\phi$, and $\psi_{\mbf{a}}|_T \in \widehat{T}$. By Lemma \ref{auxlemma}, $\phi \circ \psi_{\mbf{a}} = \psi_{\gamma \mbf{a}}$.  So
$$\phi \circ \psi_{\mbf{a}}|_T = (\phi \circ \psi_{\mbf{a}})_T = \psi_{\gamma \mbf{a}}|_T.$$  Thus the action of $\gamma$ in $(\Z/l^s\Z)^\times$ on $(\Z/l^s\Z)^n/\{(x,\dots,x)\}$ is given by scalar multiplication.
\qedhere
\end{proof}

\begin{lemma}\label{corrlemmaPSLn} For any prime $l$, let  $v = \min(v_l(n'),s)$. Let $\Gamma = \text{Gal}(k(\zeta_{l^s})/k) \hookrightarrow (\Z/l^s\Z)^\times$. Let $T = \{\mbf{b} \in (\mu_{l^s})^n : \prod_{i=1}^n b_i = 1\}$ and let $H = (\Z/l^s\Z)^{n}/\{(x,\dots,x)\}$. Then the irreducible representations of $T$ are in bijection with  $\mbf{a} \in H/\Gamma$, where the action of $\phi \in \Gamma$ is given by scalar multiplication by $\gamma_\phi$. The bijection is given by $\mbf{a} \in H/\Gamma \mapsto \Psi_{\mbf{a}}|_T: T \to GL_d(k)$, where $d = [k(\zeta_f):k]$ for $f = \frac{l^s}{\text{gcd}(a_i)}$. Furthermore, if $\Psi_{\mbf{a}}|_T$ is non-trivial on $T[l]$, then $l \nmid a_i$ for some $i$ and $\Psi_{\mbf{a}}$ has dimension $[k(\zeta_{l^s}):k]$.
\end{lemma}

\begin{proof} 
Note that 
$$k[T] = M_{n_1}(D_1) \times \dots \times M_{n_t}(D_t),$$
for $D_i$ a division algebra over $k$. So $k[T]$ is an \'etale $k$-algebra. Let $k_\text{sep}$ denote a separable extension of $k$. Let $\Gamma' = \text{Gal}(k_\text{sep}/k)$. For $A$ a finite \'etale $k$-algebra, consider the action of $\Gamma'$ on $\text{Hom}_k(A,k_\text{sep})$ given by $\phi(\lambda) = \phi \circ \lambda$ for $\phi \in \Gamma'$, $\lambda \in \text{Hom}_k(A, k_\text{sep})$. By \cite{Sz} (Theorem 1.5.4), the functor mapping a finite \'etale $k$-algebra $A$ to the finite set $\text{Hom}_k(A, k_\text{sep})$ gives an anti-equivalence between the category of finite \'etale $k$-algebras and the category of finite sets equipped with a continuous left $\Gamma'$-action; separable field extensions give rise to sets with transitive $\Gamma'$-action and Galois extensions give rise to $\Gamma'$-sets isomorphic to finite quotients of $\Gamma'$.  So if we write an \'etale $k$-algebra as a finite direct product of separable extensions of $k$, $A = L_1 \times L_2 \times \dots L_k$, then we can write $\text{Hom}_k(A, k_\text{sep}) = X_1 \coprod X_2 \coprod \dots \coprod X_t$, where $X_1, \dots, X_t$ are the orbits of $\Gamma'$ in $\text{Hom}_k(A, k_\text{sep})$. 

For $A = k[T] = \prod_i k_i$, then we have $\text{Hom}_{k}(k[T], k_\text{sep}) = \widehat{T}$. The irreducible representations of $T$ correspond to the factors $L_1, \dots, L_t$ in $k[T] = L_1 \times \dots \times L_t$ and these in turn correspond to $\widehat{T}/\Gamma'$, the orbits of $\Gamma'$ in $\widehat{T}$. The correspondence between $\widehat{T}$ and irreducible representations of $T$ is given by $\psi_{\mbf{a}}|_T \in \widehat{T}/\Gamma'$ corresponding to $\Psi_{\mbf{a}}|_T \in \text{Irr}(T)$, where $\psi_{\mbf{a}}(e_i) = (\zeta_{l^s})^{a_i}$ and  $\Psi_{\mbf{a}}(e_i) = \text{ multiplication by } (\zeta_{l^s})^{a_i}$.

Note that for any $\lambda \in \widehat{T}$, $\lambda$ is a map on $k(\zeta_{l^s})$. Also, note that a homomorphism $k_\text{sep} \to k_\text{sep}$ must send $\zeta_{l^s}$ to a $l^s$-th root of unity; so it will map $k(\zeta_{l^s})$ to itself.  Thus any homomorphism in $\Gamma'$, when restricted to $k(\zeta_{l^s})$ will be an element of $\text{Gal}(k(\zeta_{l^s})/k)$. So in our case, we may replace $\Gamma'$ by $\Gamma = \text{Gal}(k(\zeta_{l^s})/k).$ By Lemma \ref{auxlemmaPSLn}, the corresponding action of $\gamma_\phi \in (\Z/l^s\Z)^\times$ on $H \cong \widehat{T}$ is given by scalar multiplication by $\gamma_\phi$.

Note that $T[l]$ is given by $\mbf{b} \in T$ with $\mbf{b}^l = (1,\dots,1)$.  So $\mbf{b} = (\zeta_l^{x_1}, \dots \zeta_l^{x_n})$ for some $x_i \in \Z/l\Z$ with $l \divides \sum_{i=1}^n x_i$ . Hence 
\begin{align*}
\psi_{\mbf{a}}(\mbf{b}) = (\zeta_l)^{\sum_{i=1}^n a_i x_i}
\end{align*}
Thus if $l \divides a_i$ for all $i$, then $\psi_{\mbf{a}}$ will be trivial on $T[l]$, and hence $\Psi_{\mbf{a}}$ will be trivial on $T[l]$ as well. So if $\Psi_{\mbf{a}}$ is non-trivial on $T[l]$, we must have $l \nmid a_i$ for some $i$, and so $\text{gcd}(a_i) = 1$. Thus the $\Psi_{\mbf{a}}$ that are non-trivial on $T[l]$ have dimension $[k(\zeta_{l^s}):k]$.
  \qedhere
\end{proof}

\begin{lemma}\label{changeperspPSLn} For any prime $l$, let  $\Gamma = \text{Gal}(k(\zeta_{l^s})/k) \hookrightarrow (\Z/l^s\Z)^\times$ and the action of $\phi \in \Gamma$ be given by scalar multiplication by $\gamma_\phi$.  Let $T = \{\mbf{b} \in (\mu_{l^s})^n : \prod_{i=1}^n b_i = 1\}$. Let $H = (\Z/l^s\Z)^{n}/\{(x,\dots,x)\}$. Then the orbit of $\Psi_{\mbf{a}}$ under the action of $P_l(S_n)$ on $\text{Irr}(T)$ will have the same size as the orbit of $\mbf{a}$ under the action of $P_l(S_n)$ on $H/\Gamma$.  
\end{lemma}

\begin{proof}
By Lemma \ref{corrlemmaPSLn}, the irreducible representations of $T$ are in bijection with $\mbf{a} \in H/\Gamma$, where the action of $\phi \in \Gamma$ is given by scalar multiplication by $\gamma_\phi$. The bijection is given by $\mbf{a} \in H/\Gamma \mapsto \Psi_{\mbf{a}}.$

The action of $P_l(S_n)$ on $\text{Irr}(T)$ is given by 
$$\sigma(\lambda)(x) = \lambda(\sigma(x)),$$ 
for $\sigma \in P_l(S_n)$, $\lambda \in \text{Irr}(T)$, $x \in T$. And for $\Psi_{\mbf{a}} \in \text{Irr}(T),$ the orbit of $\Psi_{\mbf{a}}$ in $\text{Irr}(T)$ under the action of $P_l(S_n)$ corresponds to the orbit of $\psi_{\mbf{a}}$ in $\widehat{T}$ under the action of $P_l(S_n)$.

Under the isomorphism $\widehat{T} \cong H$, we have that the action of $P_l(S_n)$ on $\widehat{T}$, which is given by 
$$\sigma(\psi_\mbf{a})(x) = \psi_{\mbf{a}}(\sigma(x)) = \psi_{\sigma^{-1}(\mbf{a})}(x),$$
corresponds to the action of $P_l(S_n)$ on $H$ given by $\mbf{a} \mapsto \sigma^{-1}(\mbf{a}).$ 

Note that the action of $P_l(S_n)$ commutes with the action of $\Gamma$, so we get a corresponding action of $P_l(S_n)$ on $H/\Gamma$ under the bijection $\text{Irr}(T) \leftrightarrow H/\Gamma$, which is also given by $\mbf{a} \mapsto \sigma^{-1}(\mbf{a})$. The orbit of $\mbf{a}$ under this action will have the same size as the orbit of $\mbf{a}$ under the action $\mbf{a} \mapsto \sigma(\mbf{a})$. 

Therefore, the orbit of $\Psi_{\mbf{a}}$ under the action of $P_l(S_n)$ on $\text{Irr}(T)$ has the same size as the orbit of $\mbf{a}$ in $H/\Gamma$ under the action of $P_l(S_n)$ given by $\mbf{a} \mapsto \sigma(\mbf{a})$.
\end{proof}

\subsection{\texorpdfstring{Case 1: $l \divides q-1$, $l \divides n'$, $n = l^t$}{Case 1: l divides q-1, l divides n' divides n, n = lt}}

\begin{definition} For $n = l^t$, $1 \leq j \leq l$, let $J_j$ denote the $j$th sub-block of $l^{t-1}$ entries in $\{1, \dots, l^t\}$, and let  $A_j = \sum_{i \in J_j} a_i$. \end{definition}

\subsubsection{\texorpdfstring{Case 1a: $l = n' = n = 3, \text{ }s = 1$}{Case 1a: l = n' = n = 3, s = 1}}

\begin{proof}[Proof of Theorem \ref{PSLn} in the case $n = 3, s=1$]

For $l = n' = n = 3$, $s = 1$, we have 
$$P = \{\mbf{b} \in (\mu_{3})^3 : \prod_{i=1}^3 b_i = 1\}/\{(x,\dots,x)\} \rtimes \Z/3\Z.$$ Let $\tau$ be given by $i \mapsto i-1 \mod 3$, a generator of $\Z/3\Z$.  Note that any element of $\{\mbf{a} \in (\mu_{3})^3 : \prod_{i=1}^3 b_i = 1\}/\{(x,\dots,x)\}$ can be written as $\mbf{a} = (\zeta_3^a, \zeta_3^b, \zeta_{3}^{-a-b})$. And $\tau(\zeta_3^a, \zeta_3^b, \zeta_3^{-a-b}) = (\zeta_3^{-a-b}, \zeta_3^a, \zeta_3^{b})$, so $\tau(\mbf{b}) - \mbf{b} = (\zeta_3^{-2a-b}, \zeta_3^{a-b}, \zeta_3^{2b+a})$.  Note that for any $a,b \in \Z/3\Z$, the three quantities $-2a-b, a-b, 2b+a$ are equal in $\Z/3\Z$.  Thus $\tau$ is in the stabilizer of $\mbf{b}$ for all $\mbf{b} \in \{\mbf{b} \in (\mu_3)^3 : \prod_{i=1}^3 b_i = 1\}/\{(x,\dots,x)\}$.  Hence 
$$P = \{\mbf{a} \in (\mu_{3})^3 : \prod_{i=1}^n a_i = 1\}/\{(x,\dots,x)\} \times \Z/3\Z \cong \mu_3 \times \Z/3\Z.$$
Thus $\ed_k(PSL_3(\F_q),3) = 2$ in the case $l = n' = n = 3, s = 1$.  \qedhere
\end{proof}

\subsubsection{\texorpdfstring{The center in the case $n = l^t$ with $n > 3$ or $s \neq 1$}{The center in the case n = lt with n > 3 or s neq 1}}

\begin{lemma}\label{ZPSLn2} For 
$$P  = \{(\mbf{b}, \tau) \in (\mu_{l^s})^n/ \{(x,\dots,x) : x^{n'} = 1\} \rtimes P_l(S_n) : \prod_{i=1}^n b_i = \text{sgn}(\tau)\}$$ 
in the case $l \divides q - 1$, $n = l^t$ with $n > 3$ or $s \neq 1$,
$$Z(P)[l] = \langle (\zeta_l, \dots, \zeta_l, \zeta_l^2, \dots, \zeta_l^2, \dots, \zeta_l^{(l-1)}, \dots, \zeta_l^{(l-1)}, 1, \dots, 1) \rangle \cong \mu_l.$$
\end{lemma}
\begin{proof}
Fix $(\mbf{b},\tau) \in P$. Then for $(\mbf{b}', \tau') \in P$,
$$(\mbf{b},\tau)(\mbf{b}',\tau') = (\mbf{b} \tau(\mbf{b}'), \tau\tau') \text{ and } (\mbf{b}', \tau')(\mbf{b},\tau) = (\mbf{b}' \tau'(\mbf{b}), \tau'\tau).$$
Thus $(\mbf{b},\tau)$ is in the center if and only if $\tau \in Z(P_l(S_{n}))$ and
$$\mbf{b}\tau(\mbf{b}') = \mbf{b}'\tau'(\mbf{b}) \mod \{(x,\dots,x) : x^{n'} = 1\}$$ for all $\mbf{b}',\tau'$. Choosing $\tau' = \Id$, we see we must have $\mbf{b}\tau(\mbf{b}') = \mbf{b}'\mbf{b} \mod \{(x,\dots,x) : x^{n'} = 1\}$.  Thus we must have $\tau(\mbf{b}') = \mbf{b}' \mod \{(x,\dots,x) : x^{n'} = 1\}$ for all $\mbf{b}'$ with $(\mbf{b}',\text{Id}) \in P$.

If $\tau \neq \text{Id}$, then without loss of generality assume $\tau(1) = 2$ and $\tau(2) = 3$.

\textbf{Case 1: } If $\mbf{n > 3}$, then choosing 
$$b'_1 = \zeta_l, b'_2 = \zeta_l, b'_3 = 1,b'_4 = \zeta_l^{-2}, \text{ and all other entries } 1,$$ we have $\mbf{b}' \in T$. But 
$$\tau(\mbf{b}')_2 = \zeta_l = b'_2,$$ 
whereas 
$$\tau(\mbf{b'})_3 = \zeta_l \neq 1 = b'_3.$$ 
So $\tau(\mbf{b'}) \neq \mbf{b'} \mod \{(x,\dots,x) : x^{n'} = 1\}$.

\textbf{Case 2: } If $\mbf{n = 3, s > 1}$, then choosing 
$$b'_1 = (\zeta_{l^s})^{l^s-1}, b'_2 = \zeta_{l^s}, \text{ and } b'_3 = 1,$$ 
we have that 
$$\tau(\mbf{b}')_2 = (\zeta_{l^s})^{l^s-1}  = (\zeta_{l^s})^{l^s-2}b'_2$$
whereas 
$$\tau(\mbf{b'})_3 = \zeta_{l^s} = \zeta_{l^s}b'_3 \neq (\zeta_{l^s})^{l^s - 2}b'_3 \text{ since } s > 1.$$ 
So $\tau(\mbf{b'}) \neq \mbf{b'} \mod \{(x,\dots,x) : x^{n'} = 1\}$.


In either case, for any $\tau \neq \text{Id}$, we can choose a $\mbf{b}'$ for which $\tau(\mbf{b}') \neq \mbf{b}' \mod \{(x,\dots,x) : x^{n'} = 1\}$, so we can conclude that we must have $\tau = \Id$. 

We also need $\tau'(\mbf{b}) = \mbf{b} \mod \{(x,\dots,x)\}$ for all $\tau' \in P_l(S_n)$. Note that for each $i, i'$ in the same $J_j$, there exists $\tau' \in P_l(S_n)$ that sends $i$ to $i'$ and fixes some other index. Since there is an index that is fixed by $\tau$, in order for $\tau(\mbf{b})$ to equal $\mbf{b}\mbf{x}$ for $\mbf{x} = (x,x,\dots,x)$, we must have $x = 1$ and so $\tau(\mbf{b}) = \mbf{b}$. So $b_1 = \dots = b_{l^{t-1}}$, $b_{l^{t-1}+1} = \dots = b_{2l^{t-1}}$, $\dots$, $b_{l^t-l^{t-1}+1} = \dots = b_{l^t}.$  If we consider the last generator, $\sigma_1^{t}$, we see that we must have $b_{i+l^{t-1}} = b_{i}x$ for some fixed $x = \zeta_l^{a}$.  Thus $\mbf{b}$ must be of the form
$$\mbf{b} = (b\zeta_l^a, \dots b\zeta_l^a, \dots, b\zeta_l^{(l-1)}, \dots, b\zeta_l^{(l-1)}, b, \dots, b).$$
In $PSL_{l^t}(\F_z)$, the set of all elements of this form is a cyclic group of order $l$ generated by
$$\mbf{b} = (\zeta_l, \dots, \zeta_l, \zeta_l^2, \dots, \zeta_l^2, \dots, \zeta_l^{(l-1)}, \dots, \zeta_l^{(l-1)}, 1, \dots, 1).$$
So we have 
\begin{align*} Z(P)[l] &= \langle (\zeta_l, \dots, \zeta_l, \zeta_l^2, \dots, \zeta_l^2, \dots, \zeta_l^{(l-1)}, \dots, \zeta_l^{(l-1)}, 1, \dots, 1) \rangle \cong \mu_l. \qedhere
\end{align*}

\end{proof}

\subsubsection{\texorpdfstring{Case 1b: $n = l^t$ with $n > 3$ or $s \neq 1$}{Case 1b:  n = lt with n > 3 or s neq 1}}

For the proof of Theorem \ref{PSLn} in the case $l \neq 2$, $n = l^t$ with $n > 3$ or $s \neq 1$, we will need the following lemmas.



\begin{lemma}\label{irrHPSLn2'} Let $\Gamma = \text{Gal}(k(\zeta_{l^s})/k)$, $H = (\Z/l^s\Z)^{n}/\{(x,\dots,x)\}$, $n = l^t$ with $n > 3$ or $s \neq 1$, $v = \min(v_l(n'),s)$, and $\mbf{a} \in H$ with 
$$\sum_{i=1}^{n} a_i = 0 \mod l^v \text{ and } \Psi_{\mbf{a}} \text{ non-trivial on } Z(P)[l].$$ 
Then
$$|\text{orbit}(\mbf{a})| \geq l^{2t-1}$$ under the action of $P_l(S_n)$ on $H/\Gamma$. 
\end{lemma}

\begin{proof}



We will prove the lemma by induction. 

\noindent \textbf{Base Case: $t = 1$ ($n = l$)}


$Z(P)[l]$ is generated by $g = (\zeta_l, \zeta_l^2, \dots, \zeta_l^{(l-1)}, 1)$.  So since $\Psi_{\mbf{a}}$ is non-trivial on $Z(P)[l]$, we must have 
$$1 \neq \Psi_{\mbf{a}}(g) = \prod_{i=1}^l \zeta_{l}^{ia_i} = \zeta_{l}^{\sum_{i=1}^{l} ia_i}.$$
Thus $l \nmid \sum_{i=1}^l ia_i$ and so $\sum_{i=1}^l ia_i$ is invertible. 

$P_l(S_l) \cong \Z/l\Z$, so the stabilizer is either trivial or all of $P_l(S_l)$. $P_l(S_l)$ is generated by $\tau: i \mapsto i+1 \mod l$.  Suppose by way of contradiction that $\tau \neq \Id$ is in the stabilizer of $\mbf{a}$, that is 
$\tau(\mbf{a}) = \gamma \mbf{a} + (x,\dots,x)$ for some $\gamma, x$. 
Then for any $i$, 
$$a_i = \gamma^ia_l + \sum_{k=1}^{i} \gamma^{k-1}x = \gamma^i a_l + x\sum_{k=1}^i \gamma^{k-1}.$$ 
So by Lemma \ref{Helper1}, $\sum_{i=1}^l i a_i$ is not invertible, a contradiction. Therefore any non-trivial $\tau$ is not in the stabilizer of $\mbf{a}$. So the stabilizer is trivial. Thus the orbit has size $|P_l(S_l)| = l$.

\noindent \textbf{Induction Step:} Assume that the lemma is true for $t-1$. 

$Z(P)[l]$ is generated by 
$$g = (\zeta_l, \dots, \zeta_l, \zeta_l^2, \dots, \zeta_l^2, \dots, \zeta_l^{(l-1)}, \dots, \zeta_l^{(l-1)}, 1, \dots, 1).$$  So since $\lambda_i = \Psi_{\mbf{a}}$ is non-trivial on $Z(P)[l]$. we must have 
$$1 \neq \Psi_{\mbf{a}}(g) = \zeta_{l}^{\sum_{j=1}^{l} jA_j}.$$ 
Thus $l \nmid \sum_{j=1}^l j A_j$ and so $\sum_{j=1}^l jA_j$ is invertible. Hence $A_{j_1}$ is invertible for some $j_1$. 

Consider the copy of $P_l(S_{l^{t-1}}) \subset P_l(S_n)$ that acts on $J_{j_1}$. Suppose that $\sigma \in P_l(S_{l^{t-1}})$ is in the stabilizer of $\mbf{a}$ under the action on $H/\Gamma$.  Then $\sigma$ is in the stabilizer of $\mbf{a}|_{J_1}$ under the action of $P_l(S_{l^{t-1}})$ on $((\Z/l^s\Z)^{l^{t-1}}/\{(x,\dots,x)\})/\Gamma$, which by the induction hypothesis has size at most $\frac{|P_l(S_{l^{t-1}})|}{l^{2(t-1)-1}}$. Thus the orbit of $\mbf{a}$ under the action of $P_l(S_{l^{t-1}})$ acting on $J_{j_1}$ has size at least $l^{2(t-1)-1} = l^{2t-3}$.

 
If $l \divides A_j - A_{j_1}$ for all $j$, then we would have 
\begin{align*}
\sum_{j=1}^l jA_j &= \sum_{j=1}^l j(A_{j_1} + A_j-A_{j_1})\\
&= \sum_{j=1}^l jA_{j_1} \mod l\\
&= A_{j_1}\frac{l(l+1)}{2}\\
&= 0 \mod l \text{ since } l \neq 2.
\end{align*}
So $\sum_{j=1}^l jA_j$ would not be invertible, a contradiction. Therefore, there exists $A_{j_2}$ such that $l \nmid A_{j_2} - A_{j_1}$. 

Let $\tau$ be in the copy of $P_l(S_{l^{t-1}}) \subset P_l(S_{l^t})$ that acts on $J_{j_2}$. Suppose by way of contradiction that $\tau(\mbf{a}) = \gamma\sigma(\mbf{a}) + (x,\dots,x)$ for some $\gamma \in \Gamma,$ $x \in \Z/l^s\Z$, and $\sigma$ in the copy of $P_l(S_{l^{t-1}})$ acting on $J_{j_1}$. Then we must have $A_{j_2} = \gamma A_{j_2} + l^{t-1}x$ and $A_{j_1} = \gamma A_{j_1} + l^{t-1}x$. Hence $A_{j_2}-A_{j_1} = \gamma(A_{j_2}-A_{j_1})$. Then since $l \nmid A_{j_2}-A_{j_1}$, we can conclude that $\gamma = 1$. Hence $\tau(\mbf{a}) = \sigma(\mbf{a}) + (x,\dots,x).$ 

Since $l\neq 2$, both $\sigma$ and $\tau$ act trivially on $a_{i_0}$ for some $i_0 \notin J_{j_1} \cup J_{j_2}$. So 
\begin{align*}
&\tau(a_{i_0}) = \sigma(a_{i_0}) + x\\
&\Rightarrow a_{i_0} = a_{i_0} + x\\
&\Rightarrow x = 0.
\end{align*} 
Hence $\tau(\mbf{a}) = \sigma(\mbf{a}).$ This is only possible for $\tau = \text{Id} = \sigma$.  Thus for $\tau \neq \text{Id}$ in the copy of $P_l(S_{l^{t-1}})$ acting on $J_{j_2}$, we can conclude that $\tau(\mbf{a})$ is not equal to $\gamma\sigma(\mbf{a}) + (x,\dots,x)$ for any of the $\sigma \in P_l(S_{l^{t-1}})$ acting on $J_{j_1}$.  Thus the size of the orbit under the action of $P_l(S_{l^{t-1}}) \times P_l(S_{l^{t-1}}) \subset P_l(S_{l^t})$ on $H/\Gamma$ is at least $l^{2t-3}+1$, and so it must be at least $l^{2t-2}$ since it must divides $|P_l(S_{l^t})|$ which is a power of $l$.

Without loss of generality, we may assume that $j_1 = 1$ and $j_2 = l$, that is $A_1$ is invertible and $l \nmid A_l - A_1$. Let $\tau$ be the permutation $i \mapsto i + l^{t-1} \mod l^t$, that is 
$$\tau(\mbf{a}) = (a_{l^{t-1}+1}, \dots, a_{l^{t}}, a_1, \dots, a_{l^{t-1}}).$$
Suppose by way of contradiction that $\tau(\mbf{a}) = \gamma\sigma(\mbf{a}) + (x,\dots,x)$ for some $\gamma \in \Gamma,$ $x \in \Z/l^s\Z$, and $\sigma$ in $P_l(S_{l^{t-1}}) \times P_l(S_{l^{t-1}}) \subset P_l(S_n)$ acting on $J_{j_1} \times J_{j_2}$. Then since $\sigma$ stabilizes the $J_j$ and $\tau(J_j) = J_{j+1 \mod l}$, we must have 
$$A_{j+1 \mod l} = \gamma A_j + l^{t-1} x.$$
Thus
\begin{align*}
\sum_{j=1}^l jA_j &= \sum_{j=1}^l j(\gamma A_j + l^{t-1} x)\\
&= \gamma \sum_{j=1}^l j  A_j \mod l^{t-1}\\
\end{align*}
So since $\sum_{j=1}^l jA_j$ is invertible, we must have $\gamma = 1 \mod l^{t-1}$. So
$$A_1 = \gamma A_l + l^{t-1}x = A_l \mod l$$
and hence $l \divides A_l - A_1,$ a contradiction. So we can conclude that $\tau(\mbf{a})$ is not equal to $\gamma\sigma(\mbf{a}) + (x,\dots,x)$ for any of the $\sigma \in P_l(S_{l^{t-1}}) \times P_l(S_{l^{t-1}})$.  Thus the size of the orbit under the action of $P_l(S_{l^t})$ is at least $l^{2t-2} + 1$, and so it must be at least $l^{2t-1}$ since it must divide $|P_l(S_{l^{t}})|$ which is a power of $l$.
\end{proof}

\begin{corollary}\label{irrHPSLn2cor'} Let $H = (\Z/l^s\Z)^{n}/\{(x,\dots,x)\}$, $n = l^t$ with $n > 3$ or $s \neq 1$, $v = \min(v_l(n'),s)$, and $\mbf{a} \in H$ with 
$$\sum_{i=1}^{n} a_i = 0 \mod l^v \text{ and } \Psi_{\mbf{a}} \text{ non-trivial on } Z(P)[l].$$ 
Then
$$|\text{orbit}(\Psi_{\mbf{a}}|_T)| \geq l^{2t-1}$$
under the action of $P_l(S_n)$ on $\widehat{T'}$. 
\end{corollary}
\begin{proof}
By Lemma \ref{changeperspPSLn}, the orbit of $\Psi_{\mbf{a}}|_T$ under the action of $P_l(S_n)$ has the same size as the orbit of $\mbf{a}$ under the action of $P_l(S_n)$ on $H/\Gamma$. And by Lemma \ref{irrHPSLn2'}, the orbit  of $\mbf{a}$ under the action of $P_l(S_n)$ on $H/\Gamma$ has size at least $l^{2t-1}$. Therefore the orbit of $\Psi_{\mbf{a}}|_T$ has size at least $l^{2t-1}.$ \qedhere
\end{proof}

We can now complete the proof in the case $n = l^t$ with $n > 3$ or $s \neq 1$.
  
\begin{proof}[Proof of Theorem \ref{PSLn} for the case $l \divides q-1$, $l \divides n'$, $n = l^t$ with $n > 3$ or $s \neq 1$.]

Let $\rho$ be a faithful representation of $P$ of minimum dimension (and so it is also irreducible since the center has rank $1$.)   Let $T' = \{\mbf{b} \in (\mu_{l^s})^n : \prod_{i=1}^n b_i = 1\}/\{(x,\dots,x) : x^{n'} = 1\} \subset P$.  Then $T' \triangleleft P$ and so by Clifford's Theorem (Theorem \ref{cliff}), $\rho|_{T'}$ decomposes into a direct sum of irreducibles in the following manner:
$$\rho|_{T'} \cong  \left( \oplus_{i=1}^c  \lambda_i \right)^{\oplus d}, \text{ for some } c, d,$$ 
with the $\lambda_i$ non-isomorphic, and $P_l(S_n)$ acts transitively on the $\lambda_i$, so the $\lambda_i$ have the same dimension and the number of $\lambda_i$, $c$, divides $|P_l(S_n)|$ (which is a power of $l$), so $c$ is a power of $l$. Also, since $\rho$ is faithful, it is non-trivial on $Z(P)[l]$, thus one of the $\lambda_i$ must be non-trivial on $Z(P)[l]$. Without loss of generality, assume that $\lambda_1$ is non-trivial on $Z(P)[l]$.
 
 Note that the irreducible representations of $T'$ are in bijection with irreducible representations of $T = \{\mbf{b} \in (\mu_{l^s})^n : \prod_{i=1}^n b_i = 1\}$ which are trivial on $\{(x,\dots,x) : x^{n'} = 1\}$.   By Lemma \ref{corrlemmaPSLn}, the irreducible representations of $T$ are given by $\Psi_{\mbf{a}}|_T$, with $\mbf{a} \in H/\Gamma$, for $\Gamma = \text{Gal}(k(\zeta_{l^s})/k)$, and if $\Psi_{\mbf{a}}$ is non-trivial on $T[l]$, then $\Psi_\mbf{a}$ has dimension $[k(\zeta_{l^s}):k]$.  Since $\lambda_1$ is non-trivial on $Z(P)[l]$, it must be non-trivial on $T[l]$, so we must have $\dim(\lambda_1) = [k(\zeta_{l^s}):k]$, and so $\dim(\lambda_i) = [k(\zeta_{l^s}):k]$ for all $i$. 
 
Note that for $\mbf{x} = (x,\dots,x)$, $\psi_{\mbf{a}}(\mbf{x}) = x^{\sum_{i=1}^n a_i}$.  So $\psi_{\mbf{a}}|_T \in \text{Irr}(T)$ will be trivial on $\{(x,\dots,x) : x^{n'} = 1\}$ if and only if $\sum_{i=1}^n a_i = 0 \mod l^v$, where $v = \min(v_l(n'),s)$. So $\lambda_1 \cong \Psi_{\mbf{a}}|_T$ for some $\mbf{a} \in H/\Gamma$ with $\sum_{i=1}^n a_i = \mod l^v$.

Since $\lambda_1$ is non-trivial on $Z(P)[l]$,  by Corollary \ref{irrHPSLn2cor'} the orbit the orbit of $\lambda_1$ under the action of $P_l(S_n)$ has size at least $l^{2t-1}$. So $c \geq l^{2t-1}$. Thus
$$\dim(\rho) \geq l^{2t-1}[k(\zeta_{l^s}):k] = \ed_k(PGL_{l^t}(\F_q),l).$$
Also, since $PSL_n(\F_q) \subset PGL_n(\F_q)$,
$$\ed_k(PSL_n(\F_q),l) \leq \ed_k(PGL_n(\F_q),l).$$
Therefore for $n = l^t$ with $n > 3$ or $s \neq 1$,
\begin{align*}
\ed_k(PSL_n(\F_q),l) &= \ed_k(PGL_n(\F_q),l) = l^{2t-1}[k(\zeta_{l^s}):k].
\qedhere
\end{align*}

\end{proof}

\subsection{\texorpdfstring{Case 2: $l \divides q-1$, $l \divides n'$, $n \neq l^t$}{Case 2: l divides q-1, l divides n' divides n, n neq lt}}

\begin{definition} For $n \neq l^t$, $1 \leq j \leq \xi_l(n)$, let $A_j = \sum_{i \in I_j} a_i$. \end{definition}

For the proof of Theorem \ref{PSLn} in the case $l \divides q-1$, $l \divides n'$, $n \neq l^t$, we will need the following lemmas.

\begin{lemma}\label{ZPSLn1} For 
$$P  = \{(\mbf{b}, \tau) \in (\mu_{l^s})^n/ \{(x,\dots,x) : x^{n'} = 1\} \rtimes P_l(S_n) : \prod_{i=1}^n b_i = \text{sgn}(\tau)\}$$ 
in the case $l \divides q - 1$, $l \divides n$, $n \neq l^t$,
$$Z(P)[l] \cong (\mu_l)^{\xi_l(n)-1}.$$
\end{lemma}

\begin{proof}[Proof of Lemma \ref{ZPSLn1}]
Fix $(\mbf{b},\tau) \in P$. Then for $(\mbf{b}', \tau') \in P$,
$$(\mbf{b},\tau)(\mbf{b}',\tau') = (\mbf{b} \tau(\mbf{b}'), \tau\tau') \text{ and } (\mbf{b}', \tau')(\mbf{b},\tau) = (\mbf{b}' \tau'(\mbf{b}), \tau'\tau).$$
Thus $(\mbf{b},\tau)$ is in the center if and only if $\tau \in Z(P_l(S_{n}))$ and
$$\mbf{b}\tau(\mbf{b}') = \mbf{b}'\tau'(\mbf{b}) \mod \{(x,\dots,x) : x^{n'} = 1\}$$ for all $\mbf{b}',\tau'$. Choosing $\tau' = \Id$, we see we must have $\mbf{b}\tau(\mbf{b}') = \mbf{b}'\mbf{b} \mod \{(x,\dots,x) : x^{n'} = 1\}$.  Thus we must have $\tau(\mbf{b}') = \mbf{b}' \mod \{(x,\dots,x) : x^{n'} = 1\}$ for all $\mbf{b}'$ with $(\mbf{b}',\text{Id}) \in P$.

If $\tau \neq \text{Id}$, then without loss of generality assume $\tau(1) = 2$ and $\tau(2) = 3$. Since $n \neq l^t$, we must have $n > 3$, so choosing 
$$b'_1 = \zeta_l, b'_2 = \zeta_l, b'_3 = 1,b'_4 = \zeta_l^{-2}, \text{ and all other entries } 1,$$ we have $(\mbf{b}',\text{Id}) \in P$. But 
$$\tau(\mbf{b}')_2 = \zeta_l = b'_2,$$ 
whereas 
$$\tau(\mbf{b'})_3 = \zeta_l \neq 1 = b'_3.$$ 
So $\tau(\mbf{b'}) \neq \mbf{b'} \mod \{(x,\dots,x) : x^{n'} = 1\}$. Thus for any $\tau \neq \text{Id}$, we can choose a $(\mbf{b}', \text{Id}) \in P$ for which $\tau(\mbf{b}') \neq \mbf{b}' \mod \{(x,\dots,x) : x^{n'} = 1\}$, so we can conclude that we must have $\tau = \Id$. 

We also need $\tau'(\mbf{b}) = \mbf{b} \mod \{(x,\dots,x) : x^{n'} = 1\}$ for any $(\mbf{b}, \tau') \in P$. And for any $\tau' \in P_l(S_n)$, we can find $\mbf{b}'$ such that $(\mbf{b}',\tau') \in P$  So we need  $\tau'(\mbf{b}) = \mbf{b} \mod \{(x,\dots,x) : x^{n'} = 1\}$ for any $\tau' \in P_l(S_n)$.  Since $n \neq l^t$, for each $i, i'$ in the same $I_j$, there exists $\tau' \in P_l(S_n)$ that sends $i$ to $i'$ and fixes some other index. Since there is an index that is fixed by $\tau'$, in order for $\tau'(\mbf{b})$ to equal $\mbf{b}\mbf{x}$ for $\mbf{x} = (x,\dots,x)$, we must have $x = 1$ and so $\tau'(\mbf{b}) = \mbf{b}$. So $b_i = b_{i'}$ for $i,i'$ in the same $I_j$. Let $\mbf{b}^j$ be given by 
$$(\mbf{b}^j)_i = \begin{cases} \zeta_l, &i \in I_j\\
1, &i \notin I_j \end{cases}.$$ 
Note that since $l \divides n$, $\prod_{i=1}^n (\mbf{b}^j)_i = 1 = \text{sgn}(\Id)$; so $(\mbf{b}^j, \text{Id}) \in P$. Then
\begin{align*} Z(P)[l] &= \langle \mbf{b}^j \rangle_{j=1}^{\xi_l(n)}/\{(x,\dots,x) : x^{n'} = 1\}\\
&\cong \langle \mbf{b}^j \rangle_{j=1}^{\xi_l(n)-1} \text{ since } l \divides n'\\
&\cong (\mu_{l})^{\xi_l(n)-1}. \qedhere
\end{align*}

\end{proof}

\begin{lemma}\label{irrHPSLn1'} Let $\Gamma = \text{Gal}(k(\zeta_{2^s})/k)$, $H = (\Z/l^s\Z)^{n}/\{(x,\dots,x)\}$, $n \neq l^t$, $v = \min(v_l(n'),s)$, $1 \leq j_1 \leq \xi_l(n)$, and $\mbf{a} \in H$ with $$\sum_{i=1}^{n} a_i = 0 \mod l^v, \text{ } A_{j_1}\text{ invertible and } A_{j_2} \text{ invertible}.$$ 
Then 
$$|\text{orbit}(\mbf{a})| \geq l^{k_{j_1}+k_{j_2}}$$
under the action of $P_l(S_n)$ on $H/\Gamma$. 
\end{lemma}

\begin{proof}

Without loss of generality, assume that $k_{j_1} \geq k_{j_2}$. We will prove the lemma by induction on $k_{j_1}$ and $k_{j_2}$.

\noindent \textbf{Base Case: } $k_{j_1} = k_{j_2} = 1$.

If $k_{j_1} = k_{j_2} = 1$, then $n = el$ for $e < l$ and $a_{j_1}$ is invertible. Without loss of generality, assume that $j_1 = 1$ and $j_2 = 2$, that is $A_1$ and $A_2$ are invertible. 

Since $A_1$ is invertible, $a_i$ is invertible for some $i \in [1,l]$. Without loss of generality, assume that $a_l$ is invertible. And since $A_2$ is invertible, $a_i$ is invertible for some $i \in [l+1,2l]$. Without loss of generality, assume that $a_{2l}$ is invertible.

Consider the action of $P_l(S_l) \cong \Z/l\Z$ on $I_{1}$. The stabilizer under the action of $P_l(S_l)$ will either be trivial or all of $P_l(S_l)$.  Suppose by way of contradiction that $\tau \neq \text{Id}$ is in the stabilizer of $\mbf{a}$, that is $\tau(\mbf{a}) = \gamma \mbf{a} = (x,\dots,x)$ for some $\gamma,x$. Then for any $i$,
$$a_i = \gamma^i a_l + x\sum_{k=1}^i \gamma^{k-1}.$$
So by Lemma \ref{Helper1}, $A_1 = \sum_{i=1}^l i a_i$ is not invertible, a contradiction. Therefore any non-trivial $\tau$ is not in the stabilizer of $\mbf{a}$. So the stabilizer is trivial. Thus the orbit under the action of $P_l(S_l)$ acting on $I_1$ has size $|P_l(S_l)| = l$.

Now consider the action of $P_l(S_l)$ on $I_{j_2}$.  Suppose by way of contradiction that for $\tau \neq \text{Id}$, $\tau(\mbf{a}) = \gamma \sigma(\mbf{a}) + (x,\dots,x)$ for some $\gamma \in \Gamma,$ $x \in \Z/l^s\Z$, and $\sigma$ in the copy of $P_l(S_l)$ acting on $J_1$. Then since $\sigma$ stabilizes $J_2$, we must have $\tau(\mbf{a}|_{J_2}) = \gamma \mbf{a}|_{J_2} + (x,\dots,x)$. Then the same reasoning as above, we can reach a contradiction. Therefore, for any non-trivial $\tau$ in $P_l(S_l)$ acting on $I_{j_2}$, $\tau(\mbf{a})$ cannot be equal to $\gamma\sigma(\mbf{a}) + (x,\dots,x)$ for $\sigma$ in $P_l(S_l)$ acting on $I_1$. Thus the size of the orbit under the action of $P_l(S_l) \times P_l(S_l) \subset P_l(S_n)$ acting on $I_1 \times I_2$ is at least $l + 1$, and so it must be $l^2$ since it must divides $|P_l(S_n)|$ which is a power of $l$.

\noindent \textbf{Induction Step 1 (induction on $k_{j_1}$): } Assume that $k_{j_1} > 1$ and the lemma is true for $k_{j_1}-1$, $k_{j_2} = 1$. 

 Without loss of generality, assume that $k_{j_1} = 1$.

Let $K_j$ denote the $j$th sub-block of $l^{k_{1}-1}$ entries in $I_{1}$. And let $B_j = \sum_{i \in K_j} a_i$.  Then since $A_{1} = \sum_{i=1}^l B_j$ is invertible, we must have $B_j$ invertible for some $j$.  Without loss of generality, assume that $B_1$ is invertible. Consider the copy of $P_l(S_{l^{k_1-1}}) \times P_l(S_l)$ that acts on $K_1 \times I_{j_2}$. Then by the induction assumption, the orbit of $\mbf{a}$ under the action of $P_l(S_{l^{k_1-1}}) \times P_l(S_l)$ has size at least $l^{k_1}$.

If $l \divides B_j - B_1$ for all $j$, then we would have 
\begin{align*}
A_{1} &= \sum_{j=1}^l B_j\\
&= \sum_{j=1}^l (B_1 + B_j - B_1)\\
&= \sum_{j=1}^l B_1 \mod l\\
&= B_1 \frac{l(l+1)}{2}\\
&= 0 \mod l \text{ since } l \neq 2
\end{align*}
This is a contradiction with the fact that $A_{1}$ is invertible. Therefore, there exists $B_{j'}$ such that $l \nmid B_{j'} - B_1$. Without loss of generality, assume that $j' = l$, that is $l \nmid B_l - B_1$.

Let $\tau$ be the permutation in $P_l(S_{l^{k_1}})$ acting on $I_{1}$ given by $i \mapsto i + l^{k_1-1} \mod l^{k_1}$, that is 
$$\tau(\mbf{a}) = (a_{l^{k_1-1}+1},\dots a_{l^{k_1}}, a_1, a_{l^{k_1-1}}, a_{l^{k_1}+1}, \dots a_{n}).$$
Suppose by way of contradiction that $\tau(\mbf{a}) = \gamma \sigma(\mbf{a}) + (x,\dots,x)$ for some $\gamma \in \Gamma$, $x \in \Z/l^s\Z$, and $\sigma \in P_l(S_{l^{k_1-1}}) \times P_l(S_l)$ acting on $K_1 \times I_{j_2}$. Then since $\sigma$ stabilizes the $K_j$ and $\tau(K_j) = K_{j+1 \mod l}$, we must have 
$$B_{j+1\mod l} = \gamma B_j + l^{k_1-1}x.$$
Thus
\begin{align*}
A_1 &= \sum_{j=1}^l B_j\\
&= \sum_{j=1}^l (\gamma B_j + l^{k_1-1}x)\\
&= \gamma \sum_{j=1}^l B_j \mod l^{k_1-1}\\
&= \gamma A_1 \mod l^{k_1-1}
\end{align*}
So since $A_1$ is invertible, we must have $\gamma = 1 \mod l^{k_1-1}$. So
$$B_1 = \gamma B_l + l^{k_1-1}x = B_l \mod l \text{ since } k_1 > 1$$
and hence $l \divides B_l - B_1,$ a contradiction. So we can conclude that $\tau(\mbf{a})$ is not equal to $\gamma \sigma(\mbf{a}) + (x,\dots,x)$ for any of the $\sigma$ in $P_l(S_{l^{k_1-1}}) \times P_l(S_l)$. Thus the size of the orbit under the action of $P_l(S_{l^{k_1}}) \times P_l(S_l)$ acting on $I_{j_1} \times I_{j_2}$ is at least $l^{k_1} + 1$, and so it must be at least $l^{k_1+1}$ since it must divides $|P_l(S_{l^{k_1}}) \times P_l(S_l)|$ which is a power of $l$.

So we have now shown that the lemma is true for any $k_{j_1}$ with $k_{j_2} = 1$.

\noindent \textbf{Induction Step 2 (induction on $k_{j_2}$): } Assume that $k_{j_2} > 1$ and the lemma is true for $k_{j_1}$, $k_{j_2}-1$.

Let $K_j$ denote the $j$th sub-block of $l^{k_{1}-1}$ entries in $I_{j_2}$. And let $B_j = \sum_{i \in K_j} a_i$.  Then since $A_{j_2} = \sum_{i=1}^l B_j$ is invertible, we must have $B_j$ invertible for some $j$.  Without loss of generality, assume that $B_1$ is invertible. Consider the copy of $P_l(S_{l^{k_{j_1}}}) \times P_l(S_{l^{k_{j_2}-1}})$ that acts on $I_{j_1} \times K_1$. Then by the induction assumption, the orbit of $\mbf{a}$ under the action of $P_l(S_{l^{k_{j_1}}}) \times P_l(S_{l^{k_{j_2}-1}})$ has size at least $l^{k_{j_1}+k_{j_2}-1}$.

Just as in induction step 1, there must exist $B_{j'}$ such that $l \nmid B_{j'} - B_1$. And without loss of generality, we may assume that $j' = l$, that is $l \nmid B_l - B_1$.

Let $\tau$ be the permutation in $P_l(S_{l^{k_{j_2}}})$ acting on $I_{j_2}$ given by $i \mapsto i + l^{k_{j_2}-1} \mod l^{k_{j_2}}$.
Suppose by way of contradiction that $\tau(\mbf{a}) = \gamma \sigma(\mbf{a}) + (x,\dots,x)$ for some $\gamma \in \Gamma$, $x \in \Z/l^s\Z$, and $\sigma \in P_l(S_{l^{k_{j_1}}}) \times P_l(S_{l^{k_{j_2}-1}})$ acting on $I_{j_1} \times K_1$. Then since $\sigma$ stabilizes the $K_j$ and $\tau(K_j) = K_{j+1 \mod l}$, we must have 
$$B_{j+1\mod l} = \gamma B_j + l^{k_1-1}x.$$
Thus
\begin{align*}
A_{j_2} &= \sum_{j=1}^l B_j\\
&= \sum_{j=1}^l (\gamma B_j + l^{k_{j_2}-1}x)\\
&= \gamma \sum_{j=1}^l B_j \mod l^{k_{j_2}-1}\\
&= \gamma A_{j_1} \mod l^{k_{j_2}-1}
\end{align*}
So since $A_{j_2}$ is invertible, we must have $\gamma = 1 \mod l^{k_{j_2}-1}$. So
$$B_1 = \gamma B_l + l^{k_{j_2}-1}x = B_l \mod l \text{ since } k_{j_2} > 1$$
and hence $l \divides B_l - B_1,$ a contradiction. So we can conclude that $\tau(\mbf{a})$ is not equal to $\gamma \sigma(\mbf{a}) + (x,\dots,x)$ for any of the $\sigma$ in $\sigma \in P_l(S_{l^{k_{j_1}}}) \times P_l(S_{l^{k_{j_2}-1}})$. Thus the size of the orbit under the action of $P_l(S_{l^{k_{j_1}}}) \times P_l(S_{l^{k_{j_2}-1}}) \subset P_l(S_n)$ is at least $l^{k_{j_1}+k_{j_2}-1} + 1$, and so it must be at least $l^{k_{j_1}+k_{j_2}}$ since it must divides $|P_l(S_n)|$ which is a power of $l$. \qedhere   
\end{proof}

\begin{corollary}\label{irrHPSLn1cor'} Let $H = (\Z/l^s\Z)^{n}/\{(x,\dots,x)\}$, $n \neq l^t$, $v = \min(v_l(n'),s)$, $j_1 \in \{1, \dots, \xi_l(n)\}$, and $\mbf{a} \in H$ with $$\sum_{i=1}^{n} a_i = 0 \mod l^v, \text{ } A_{j_1}\text{ invertible}.$$ 
Then $$|\text{orbit}(\Psi_{\mbf{a}}|_T)| \geq l^{k_{j_1}+v_l(n)}$$
under the action of $P_l(S_n)$ on $\widehat{T'}$. 
\end{corollary}
\begin{proof}
By Lemma \ref{changeperspPSLn}, the orbit of $\Psi_{\mbf{a}}|_T$ under the action of $P_l(S_n)$ has the same size as the orbit of $\mbf{a}$ under the action of $P_l(S_n)$ on $H/\Gamma$. 

Since $l^v \divides \sum_{i=1}^{n} a_i$, we know that $\sum_{i=1}^n a_i = \sum_{j=1}^l A_j$ is not invertible. So since $A_{j_1}$ is invertible, we must have $A_{j_2}$ invertible as well for some $j_2 \neq j_1$. 

Then by \ref{irrHPSLn1'}, the orbit  of $\mbf{a}$ under the action of $P_l(S_n)$ on $H/\Gamma$ has size at least $l^{k_{j_1}+k_{j_2}}$. And for all $j$, $k_{j} \geq v_l(n)$. Thus the orbit of $\mbf{a}$ under the action of $P_l(S_n)$ on $H/\Gamma$ has size at  least $l^{k_{j_1}+v_l(n)}$.

Therefore the orbit of $\Psi_{\mbf{a}}|_T$ has size at least $l^{k_{j_1}+v_l(n)}.$ \qedhere
\end{proof}

We can now complete the proof in the case $n \neq l^t$.

\begin{proof}[Proof of Theorem \ref{PSLn} for the case $l \divides q-1$, $l \divides n'$, $n \neq l^t$.]

Recall that
$$P = \{(\mbf{b}, \tau) \in (\mu_{l^s})^n/\\
 \{(x,\dots,x) : x^{n'} = 1\} \rtimes P_l(S_n) : \prod_{i=1}^n b_i = \text{sgn}(\tau)\}.$$
  Let $\rho$ be a faithful representation of $P$ of minimum dimension.  Let $\rho = \oplus_{j=1}^{\xi_l(n)-1} \varphi_j$ be the decomposition into irreducibles. Let $C = Z(P)$. For $j \leq \xi_l(n)-1$, let 
 $$T_j = \{\mbf{b} \in (\mu_{l^s})^n : \prod_{i=1}^n b_i = 1, \text{ }b_i = 1 \text{ for } i \notin I_j\}/\{(x,\dots,x) : x^{n'} = 1\}.$$
  By the same reasoning as for $PGL_n(\F_q)$, we can rearrange the $\rho_j$ such that $\chi_j|_{C[l]}$ is non-trivial on $T_j \cap C[l]$ and thus $\varphi_j$ is non-trivial on $T_j \cap C[l]$.

Fix $j \leq \xi_l(n)-1$ and let $\varphi = \varphi_j$. Let $T' = \{\mbf{b} \in (\mu_{l^s})^n : \prod_{i=1}^n b_i = 1\}/ \{(x,\dots,x) : x^{n'} = 1\}$. Then $T' \triangleleft P$. So by Clifford's Theorem (Theorem \ref{cliff}), $\varphi|_{T'}$ decomposes into a direct sum of irreducibles in the following manner:
$$\varphi|_{T'} \cong  \left( \oplus_{i=1}^c  \lambda_i \right)^{\oplus d}, \text{ for some } c, d,$$ 
with the $\lambda_i$ non-isomorphic, and $P_l(S_{n})$ acts transitively on the isomorphism classes of the $\lambda_i$, so the $\lambda_i$ have the same dimension and the number of $\lambda_i$, $c$, divides $|P_l(S_{n})|$. Since $\varphi$ is non-trivial on $T_j \cap C[l]$, one of the $\lambda_i$ must be non-trivial on $T_j \cap C[l]$. Without loss of generality, assume that $\lambda_1$ is non-trivial on $T_j \cap C[l]$.

 Note that the irreducible representations of $T'$ are in bijection with irreducible representations of $T = \{\mbf{b} \in (\mu_{l^s})^n : \prod_{i=1}^n b_i = 1\}$ which are trivial on $\{(x,\dots,x) : x^{n'} = 1\}$.   By Lemma \ref{corrlemmaPSLn}, the irreducible representations of $T$ are given by $\Psi_{\mbf{a}}|_T$, with $\mbf{a} \in H/\Gamma$, for $\Gamma = \text{Gal}(k(\zeta_{l^s})/k)$, and if $\Psi_{\mbf{a}}$ is non-trivial on $T[l]$, then $\Psi_\mbf{a}$ has dimension $[k(\zeta_{l^s}):k]$.  Since $\lambda_1$ is non-trivial on $T_j \cap C[l] \subset T[l]$, it must be non-trivial on $T[l]$, so we must have $\dim(\lambda_1) = [k(\zeta_{l^s}):k]$, and so $\dim(\lambda_i) = [k(\zeta_{l^s}):k]$ for all $i$. 
 
Note that for $\mbf{x} = (x,\dots,x)$, $\psi_{\mbf{a}}(\mbf{x}) = x^{\sum_{i=1}^n a_i}$.  So $\Psi_{\mbf{a}}|_T \in \text{Irr}(T)$ will be trivial on $\{(x,\dots,x) : x^{n'} = 1\}$ if and only if $\sum_{i=1}^n a_i = 0 \mod l^v$, where $v = \min(v_l(n'),s)$. So $\lambda_1 \cong \Psi_{\mbf{a}}|_T$ for some $\mbf{a} \in H/\Gamma$ with $\sum_{i=1}^n a_i = \mod l^v$.

Also, since $\lambda_1 \cong \Psi_\mbf{a}|_T$ is non-trivial on $T_j \cap C[l] = \langle \mbf{b}^j \rangle$, where $(\mbf{b}^j)_i = \begin{cases} \zeta_l, &i \in I_j\\
1, &i \notin I_j \end{cases},$
we must have 
$$1 \neq \prod_{i\in I_{j}} \zeta_{l}^{a_i} = \zeta_{l}^{\sum_{i\in I_{j}} a_i} = \zeta_l^{A_{j}}.$$
Thus $l \nmid A_{j}$ and so $A_{j}$ is invertible.  So by Corollary \ref{irrHPSLn1cor'}, the orbit of $\lambda_i = \Psi_{\mbf{a}}|_{T}$ under the action of $P_l(S_n)$ has size at least $l^{k_j + v_l(n)}$. So $c \geq l^{k_j + v_l(n)}$. Thus for $\varphi = \varphi_j$,
$$\dim(\varphi) \geq l^{k_j + v_l(n)}[k(\zeta_{l^s}):k].$$
Hence
\begin{align*}
\dim(\rho) &= \sum_{j=1}^{\xi_l(n)-1} \dim(\varphi_j)\\
&\geq \sum_{j=1}^{\xi_l(n)-1} l^{k_j + v_l(n)}[k(\zeta_{l^s}):k]\\
&= l^{v_l(n)}\left(\sum_{j=1}^{\xi_l(n)-1} l^{k_j}\right)[k(\zeta_{l^s}):k]\\
&= l^{v_l(n)}(n-l^{v_l(n)})[k(\zeta_{l^s}):k]\\
&= \ed_k(PGL_n(\F_q),l)
\end{align*}
Also, since $PSL_n(\F_q) \subset PGL_n(\F_q)$,
$$\ed_k(PSL_n(\F_q),l) \leq \ed_k(PGL_n(\F_q),l).$$
Therefore for $n \neq l^t$,
\begin{align*}
\ed_k(PSL_n(\F_q),l) &= \ed_k(PGL_n(\F_q),l) = l^{v_l(n)}(n-l^{v_l(n)})[k(\zeta_{l^s}):k]. \qedhere
\end{align*}
\end{proof}

\section{The Symplectic Groups}

\begin{theorem}\label{PSp} Let $p$ be a prime, $q = p^r$, and $l$ a prime with $l \neq 2,p$.  Let $k$ be a field with $\text{char } k \neq l$. Let $d$ be the smallest positive integer such that $l \divides q^d - 1$. Then
\begin{align*} 
\ed_k(PSp(2n,q),l) = \ed_k(Sp(2n,q),l) &= \begin{cases} \ed_k(GL_{2n}(\F_q),l), &d \text{ even} \\ 
\ed_k(GL_n(\F_q),l), &d \text{ odd} \end{cases}
\end{align*}
\end{theorem}

\begin{proof}
\noindent By Grove (\cite{Gr}, Theorem 3.12),

$$|PSp(2n,q)| = \frac{|Sp(2n,q)|}{(2,q-1)}.$$
So since $l \neq 2$, $|l,PSp(2n,q)|_l = |Sp(2n,q)|_l.$ Hence since $PSp(2n,q)$ is a quotient of $Sp(2n,q)$, we can conclude that their Sylow $l$-subgroups are isomorphic.
Let $d$ be the smallest positive integer such that $l \divides q^d - 1$ and let $s = v_l(q^d-1)$.

\noindent If $\mbf{d}$ \textbf{ is even}: Then by Stather (\cite{Sta}), $|Sp(2n,q)|_l = |GL_{2n}(\F_{q})|_l.$ Hence since $Sp(2n,q)$ is a subgroup of $GL_{2n}(\F_{q})$, we can conclude that their Sylow $l$-subgroups are isomorphic.

\noindent If $\mbf{d}$ \textbf{ is odd}:
Then by Stather (\cite{Sta}), letting $n_0 = \lfloor \frac{n}{d} \rfloor$, we have
\begin{align*}
|Sp(2n,q)|_l &= |GL_{n}(\F_q)|_l = l^{sn_0} \cdot |S_{n_0}|_l
\end{align*}
Let $\epsilon$ be primitive $l^s$-th root in $\F_{q^d}$, and let $E$ be the image of $\epsilon$ in $GL_d(\F_q)$.  Let
{\footnotesize $$E_1 = \begin{pmatrix}
E &   &        &   &            &   &        &\\
  & 1 &        &   &            &   &        &\\
  &   & \ddots &   &            &   &        &\\
  &   &        & 1 &            &   &        &\\
  &   &        &   & (E^{-1})^T &   &        & \\
  &   &        &   &            & 1 &        & \\
  &   &        &   &            &   & \ddots & \\
  &   &        &   &            &   &        & 1
\end{pmatrix}, \ldots,$$
$$E_{n_0} = \begin{pmatrix}
1 &        &   &   &                     &   &        &   &            & \\
  & \ddots &   &   &                     &   &        &   &            & \\
  &        & 1 &   &                     &   &        &   &            & \\
  &        &   & E &                     &   &        &   &            & \\
  &        &   &   & \text{Id}_{n-n_0d}  &   &        &   &            & \\
  &        &   &   &                     & 1 &        &   &            & \\
  &        &   &   &                     &   & \ddots &   &            & \\
  &        &   &   &                     &   &        & 1 &            & \\
  &        &   &   &                     &   &        &   & (E^{-1})^T & \\
  &        &   &   &                     &   &        &   &            & \text{Id}_{n-n_0d}
\end{pmatrix}$$}
Then for all $i$, $E_i \in Sp(2n,p^r)$. Note we can embed $P_l(S_{n_0})$ into $Sp(2n,q)$. Let 
\begin{align*} P = \langle E_1, \ldots, E_{n_0} \rangle \rtimes P_l(S_{n_0}) \cong (\Z/l^s\Z)^{n_0} \rtimes P_l(S_{n_0})
\end{align*}
Then $P \in \syl_l(Sp(2n,q))$, and $P$ is isomorphic to a Sylow $l$-subgroup of $GL_n(\F_q)$.
\qedhere
\end{proof}

\section{\texorpdfstring{The Orthogonal Groups, $p \neq 2$}{The Orthogonal Groups, p neq 2}}

Assume $p \neq 2$.

\subsection{Definitions}
We recall briefly the definitions of $O^\epsilon(n,p^r)$ and $P\Omega^\epsilon(n,p^r)$:
\subsubsection{\texorpdfstring{The case $n = 2m$}{The case n = 2m}}

Let $$A^+ = \begin{pmatrix} 0_m & \text{Id}_m\\
\text{Id}_m & 0_m\end{pmatrix}.$$
Let $\eta \in \F_{p^r}^\times$ be a non-square and let $\text{Id}_m^\eta$ be the $m \times m$ identity matrix with the first entry replaced by $\eta$. Let 
$$A^- = \begin{pmatrix} 0_m & \text{Id}_m\\
\text{Id}_{m}^\eta & 0_m\end{pmatrix}.$$

\begin{definition} Define $O^+(2m,p^r)$ as $\{M \in GL(2m,\F_{p^r}) : M^T A^+ M = A^+\}$. \end{definition}

\begin{definition} Define $O^-(2m,p^r)$ as $\{M \in GL(2m, \F_{p^r}) : M^T A^- M = A^-\}$.\end{definition}

\begin{definition} For $\epsilon = +,-$, define $SO^\epsilon(2m,p^r) = \{M \in O^\epsilon(2m,p^r) : \det(M) = 1\}$. \end{definition}

\begin{definition} For $\epsilon = +,-$, define $\Omega^\epsilon(2m,p^r) = (SO^\epsilon)'(2m,p^r)$ (the commutator subgroup). \end{definition}

\begin{definition} For $\epsilon = +, -$, define 
$$P\Omega^\epsilon(2m,p^r) = \Omega^\epsilon(2m,p^r)/(\Omega^\epsilon(2m,p^r) \cap \{\epsilon \text{Id}\}).$$ \end{definition}

\subsubsection{\texorpdfstring{The case $n = 2m+1$}{The case n = 2m+1}}

Let $$L = \begin{pmatrix} -1 & \mbf{0} & \mbf{0}\\
\mbf{0} & 0_m & \text{Id}_m\\
\mbf{0} & \text{Id}_m & 0_m
\end{pmatrix}.$$

\begin{definition} Define $O(2m + 1,p^r) = \{M \in GL(2m + 1,\F_{p^r}) : M^T L M = L\}$. \end{definition}

\begin{definition} Define $SO(2m + 1,p^r) = \{M \in O(2m + 1,p^r) : \det(M) = 1\}$. \end{definition}

\begin{definition} Define $\Omega(2m + 1,p^r) = SO'(2m + 1,p^r)$ (the commutator subgroup). \end{definition}

\subsection{Theorem and Proof}
 
\begin{theorem}\label{On} Let $p$ be a prime, $q = p^r$, and $l$ a prime with $l \neq 2,p$.  Let $k$ be a field with $\text{char } k \neq l$.  Let $d$ be the smallest positive integer such that $l \divides q^d - 1$, and let $n_0 = \lfloor \frac{n}{d} \rfloor$. 
{\small \begin{align*}
&\ed_k(P\Omega^\epsilon(n,q),l)\\
&= \ed_k(O^\epsilon(n,q),l)\\
&= \begin{cases} \ed_k(GL_m(\F_q),l), &n = 2m+1, \text{ d odd}\\
&\text{or } n = 2m, d \text{ odd}, \epsilon = + \\
\ed_k(GL_{m-1}(\F_q),l), &n = 2m, d \text{ odd}, \epsilon = - \\
\ed_k(GL_{2m}(\F_q),l), &n = 2m+1, \text{ d even}\\
&\text{or } n = 2m, \text{ d even}, n_0 \text{ even}, \epsilon = +\\
&\text{or } n = 2m, d \text{ even}, n_0 \text{ odd}, \epsilon = -\\
ed_k(GL_{2m-2}(\F_q),l), &n = 2m, d \text{ even}, n_0 \text{ odd}, \epsilon = +\\
&\text{or } n = 2m, d \text{ even}, n_0 \text{ even}, \epsilon = -\\
\end{cases}
\end{align*}}
\end{theorem}

 \begin{remark} We do not need to prove the case $n = 2m+1, p = 2$ since $O^\epsilon(2m+1,2^r) \cong Sp(2m,2^r)$ (\cite{Gr}, Theorem 14.2), so this case is taken care of in the work on the symplectic groups. \end{remark}

\begin{proof}

\noindent By Grove, for $p \neq 2$ (\cite{Gr}, Theorem 9.11 and Corollary 9.12),

$$|P\Omega(2m+1,q)| = \frac{|O(2m+1,q)|}{4} \qquad \text{ and } \qquad |P\Omega^\epsilon(2m,q)| = \frac{|O^\epsilon(2m,q)|}{2(4,q^m - \epsilon 1)}.$$
For $p = 2$ (\cite{Gr}, Theorem 14.48 and Corollary 14.49),
$$|P\Omega^\epsilon(2m,q)| = \frac{|O^\epsilon(2m,q)|}{2}.$$
So in all cases, since $l \neq 2$, we have that
$|P\Omega^\epsilon(n,q)|_l = |O^\epsilon(n,q)|_l.$ Hence since $P\Omega^\epsilon(n,q)$ is a quotient of $O^\epsilon(n,q)$, we can conclude that their Sylow $l$-subgroups are isomorphic.

 Let $d$ be the smallest positive integer such that $l \divides q^d - 1$ and let $s = l^{v_l(q^d-1)}$. 

\subsubsection{\texorpdfstring{The case $n = 2m+1$}{The case n=2m+1}}

\noindent If $\mbf{d}$ \textbf{ is even}: Then $|O(2m+1,q)|_l = |GL_{2m+1}(\F_q)|_l = |GL_{2m}(\F_q)|_l.$ Hence since $O(2m+1,q)$ embeds in $GL_{2m+1}(\F_q)$ and $GL_{2m}(\F_q)$ embeds in $GL_{2m+1}(\F_q)$, we can conclude that the Sylow $l$-subgroups of $O(2m+1,q)$, $GL_{2m+1}(\F_q)$, and $GL_{2m}(\F_q)$ are isomorphic.

\noindent If $\mbf{d}$ \textbf{ is odd}:
Then by Stather (\cite{Sta}), letting $m_0 = \lfloor \frac{m}{d} \rfloor$, we have
\begin{align*}
|O(2m+1,q)|_l &= |GL_m(\F_q)|_l = l^{sm_0} \cdot P_l(S_{m_0})
\end{align*}
Let $\epsilon$ be primitive $l^s$-th root in $\F_{q^d}$, and let $E$ be the image of $\epsilon$ in $GL_d(\F_q)$.  Let
{\scriptsize $$E_1 = \begin{pmatrix}
1 &   &   &        &   &            &   &        &\\ 
  & E &   &        &   &            &   &        &\\
  &   & 1 &        &   &            &   &        &\\
  &   &   & \ddots &   &            &   &        &\\
  &   &   &        & 1 &            &   &        &\\
  &   &   &        &   & (E^{-1})^T &   &        & \\
  &   &   &        &   &            & 1 &        & \\
  &   &   &        &   &            &   & \ddots & \\
  &   &   &        &   &            &   &        & 1
\end{pmatrix}, \ldots,$$
$$E_{m_0} = \begin{pmatrix}
1 &        &   &   &                     &   &        &   &            & \\
  & \ddots &   &   &                     &   &        &   &            & \\
  &        & 1 &   &                     &   &        &   &            & \\
  &        &   & E &                     &   &        &   &            & \\
  &        &   &   & \text{Id}_{m-m_0d}  &   &        &   &            & \\
  &        &   &   &                     & 1 &        &   &            & \\
  &        &   &   &                     &   & \ddots &   &            & \\
  &        &   &   &                     &   &        & 1 &            & \\
  &        &   &   &                     &   &        &   & (E^{-1})^T & \\
  &        &   &   &                     &   &        &   &            & \text{Id}_{m-m_0d}
\end{pmatrix}$$}
Then for all $i$, $E_i \in O(2m+1,p^r)$. Note we can embed $P_l(S_{m_0})$ into $O(2m+1,q)$. Let 
\begin{align*} P = \langle E_1, \ldots, E_{n_0} \rangle \rtimes P_l(S_{m_0}) \cong (\Z/l^s\Z)^{n_0} \rtimes P_l(S_{m_0}) 
\end{align*}
Then $P \in \syl_l(O(2m+1,q))$, and $P$ is isomorphic to a Sylow $l$-subgroup of $GL_m(\F_q)$.

\subsubsection{\texorpdfstring{The case $n = 2m$}{The case n=2m}}

Note that $O^\epsilon(n,q)$ embeds into $O^\epsilon(n+1,q)$ via 
$$X \mapsto \begin{pmatrix} 1 & 0 \\
0 & X \end{pmatrix}.$$
By Grove (\cite{Gr}, Theorem 9.11 and Corollary 9.12),
$$|O^+(2m,q)| = 2q^{m(m-1)}(q^m - 1)\prod_{i=1}^{m-1} (q^{2i}-1).$$
$$|O^-(2m,q)| = 2q^{m(m-1)}(q^m + 1)\prod_{i=1}^{m-1} (q^{2i}-1).$$
and
$$|O(2m+1,1)| = 2q^{m^2}\prod_{i=1}^m (q^{2i}-1).$$
Thus
\begin{align*}
[O(2m+1,q):O^+(2m,q)] &= q^m(q^m + 1)\\
[O^+(2m,q):O(2m-1,q)] &= q^{m-1}(q^m-1)\\
[O(2m+1,q):O^-(2m,q)] &= q^m(q^m - 1)\\
[O^-(2m,q):O(2m-1,q)] &= q^{m-1}(q^m+1)
\end{align*}
Note that since $l \neq 2$, either $q^m+1$ or $q^m - 1$ is prime to $l$.

\noindent If $q^m+1$ is prime to $l$, then 
\begin{align*}
|O^+(2m,q)|_l = |O(2m+1,q)|_l\\
|O^-(2m,q)|_l = |O(2m-1,q)|_l
\end{align*}
Thus when $q^m+1$ is prime to $l$, the Sylow $l$-subgroups of $O^+(2m,q)$ are isomorphic to those of $O(2m+1,q)$, and the Sylow $l$-subgroups of $O^-(2m,q)$ are isomorphic to those of $O(2m-1,q)$.

\noindent If $q^m-1$ is prime to $l$, then
\begin{align*}
|O^+(2m,q)|_l = |O(2m-1,q)|_l\\
|O^-(2m,q)|_l = |O(2m+1,q)|_l
\end{align*}

Thus when $q^m-1$ is prime to $l$, the Sylow $l$-subgroups of $O^+(2m,q)$ are isomorphic to those of $O(2m-1,q)$, and the Sylow $l$-subgroups of $O^-(2m,q)$ are isomorphic to those of $O(2m+1,q)$.

We showed in the subsection on odd orthogonal groups that when $d$ is even, the Sylow $l$-subgroups of $O(2m+1,q)$ are isomorphic to those of $GL_{2m}(\F_q)$, and when $d$ is odd, the Sylow $l$-subgroups of $O(2m+1,q)$ are isomorphic to those of $GL_m(\F_q)$. 

Recall that we defined $n_0 = \lfloor \frac{2m}{d} \rfloor$. By Stather \cite{Sta},
$$|O^+(2m,q)|_l = \begin{cases} |GL_m(\F_q)|_l, &d \text{ odd}\\
|GL_{2m-2}(\F_q)|_l, &d \text{ even}, n_0 \text{ odd}\\
 |GL_{2m}(\F_q)|_l, &d \text{ even}, n_0 \text{ even} \end{cases}$$
and
$$|O^-(2m,q)|_l = \begin{cases} |GL_{m-1}(\F_q)|_l, &d \text{ odd}\\
|GL_{2m}(\F_q)|_l, &d \text{ even}, n_0 \text{ odd}\\
|GL_{2m-2}(\F_q)|_l, &d \text{ even}, n_0 \text{ even} \end{cases}.$$

In order for this to match up with the isomorphisms to the odd orthogonal groups, we must have that when $d$ is odd or $d$ is even with $n_0$ even, then $q^m + 1$ is prime to $l$.  When d is even with $n_0$ odd, then $q^m -1$ is prime to $l$. 

\noindent \textbf{Case 1: }$\mbf{d}$ \textbf{odd}

For $d$ odd, the Sylow $l$-subgroups of $O^+(2m,q)$ are isomorphic to those of $O(2m+1,q)$, which are isomorphic to those of $GL_m(\F_q)$ and the Sylow $l$-subgroups of $O^-(2m,q)$ are isomorphic to those of $O(2m-1,q)$, which are isomorphic to those of $GL_{m-1}(\F_q)$.

\noindent \textbf{Case 2: }$\mbf{d}$ \textbf{even}, $\mbf{n_0}$ \text{odd}

For $d$ even, $n_0$ odd, the Sylow $l$-subgroups of $O^+(2m,q)$ are isomorphic to those of $O(2m-1,q)$, which are isomorphic to those of $GL_{2m-2}(\F_q)$ and the Sylow $l$-subgroups of $O^+(2m,q)$ are isomorphic to those of $O(2m+1,q)$, which are isomorphic to those of $GL_{2m}(\F_q)$.

\noindent \textbf{Case 3: }$\mbf{d}$ \textbf{even}, $\mbf{n_0}$ \text{even}

For $d$ even, $n_0$ even, the Sylow $l$-subgroups of $O^+(2m,q)$ are isomorphic to those of $O(2m+1,q)$, which are isomorphic to those of $GL_{2m}(\F_q)$ and the Sylow $l$-subgroups of $O^-(2m,q)$ are isomorphic to those of $O(2m-1,q)$, which are isomorphic to those of $GL_{2m-2}(\F_q)$.

Putting the above results together, we get Theorem \ref{On}. \qedhere
\end{proof}

\section{The Unitary Groups}

\subsection{Definitions}

 For any prime $p$,  by the Artin-Schreier Theorem (\cite{La}, Theorem 6.4) we can write $\F_{p^{2r}} = \F_{p^r}[\alpha]$ where $\alpha$ is a root of $x^2 - x - \eta$. Then $\F_{p^{2r}} = \F_{p^r} + \alpha\F_{p^r}$. 
 
\begin{definition} For $p \neq 2$, we can change $x^2 - x - \eta$ by a linear change of variables to the form $x^2 + \eta'$. Let $\alpha$ be a root of $x^2 + \eta'$. Then we still have $\F_{p^{2r}} = \F_{p^r} + \alpha\F_{p^r}$. We will write elements in $\F_{p^{2r}}$ as $a = \dot{a} + \ddot{a}\alpha$, where $\dot{a}, \ddot{a} \in \F_{p^r}$.  For $a = \dot{a} + \ddot{a}\alpha$, define $\overline{a} = \dot{a} - \ddot{a}\alpha$. \end{definition}

\begin{remark} For any prime $p$, we have that for all $a \in \F_{p^{2r}}$, $\overline{\overline{a}} = a$. \end{remark}

\begin{definition} For $\mbf{x} = (x_i) \in (\F_{p^{2r}})^n$, let $\overline{\mbf{x}} = (\overline{x_i})$. \end{definition}

\begin{definition} For $M = (M_{ij}) \in GL(n,\F_{p^{2r}})$, let $\overline{M} = (\overline{M_{ij}})$. \end{definition}

\begin{definition} A matrix $\beta$ is called \textbf{Hermitian} if $\beta^T = \overline{\beta}$. \end{definition}

\begin{definition} Given a Hermitian matrix $\beta$, the unitary matrices are defined by $$U_\beta(n,p^{2r}) := \{M \in GL(n,\F_{p^{2r}}) : M^T\beta\overline{M} = \beta\}.$$ \end{definition}

\noindent By Grove (\cite{Gr}, Corollary 10.4), for any $p, r, n, \beta, \beta'$, $U_\beta(n,p^{2r}) \cong U_{\beta'}(n,p^{2r})$, so we may work with whichever choice of $\beta$ we like. So for the remainder of this paper, we will drop the $\beta$ in our notation.

\begin{definition} The special unitary matrices are defined by $$SU(n,p^{2r}) := \{M \in U(n,p^{2r}) : \det(M) = 1\}.$$ Define $$PSU(n,p^{2r}) := SU(n,p^{2r})/(Z(SU(n,p^{2r})).$$ \end{definition}

\begin{remark} Grove shows (\cite{Gr}, Proposition 11.17) that $$Z(SU(n,p^{2r})) = \{a\text{Id}_n : a \in \F_{p^{2r}}, a\overline{a} = 1, \text{ and } a^n = 1\}.$$
\end{remark}

\subsection{The Unitary Groups}

\begin{theorem}\label{Un}
Let $p$ be a prime, $q = p^r$, and $l$ a prime with $l \neq 2,p$.  Let $k$ be a field with $\text{char } k \neq l$. Let $d$ be the smallest positive integer such that $l \divides q^d - 1$. Then
$$\ed_k(U(n,q^2),l) = \begin{cases}\ed_k(GL_n(\F_{q^2}),l), &d = 2 \pmod{4}\\
\ed_k(GL_{\lfloor \frac{n}{2} \rfloor}(\F_{q^2}),l), &d \neq 2 \pmod{4}\end{cases}
$$
\end{theorem}

\begin{proof}

By Stather \cite{Sta}
$$|U(n,q^2)|_l = \begin{cases} |GL_n(\F_{q^2})|_l, &d = 2 \pmod{4}\\
|GL_{\lfloor \frac{n}{2} \rfloor}(\F_{q^2})|_l, &d \neq 2 \pmod{4}\end{cases}$$

\noindent \textbf{Case 1: } $\mbf{d = 2 \pmod{4}}$.

Since $U(n,q^2) \subset GL_n(\F_{q^2})$ and $|U(n,q^2)|_l = |GL_n(\F_{q^2})|_l$ in this case, we can immediately conclude that for $d = 2 \pmod{4}$, the Sylow $l$-subgroups of 
$U(n,q^2)$ and $GL_n(\F_{q^2})$ are isomorphic.

\noindent \textbf{Case 2: } $\mbf{d \neq 2 \pmod 4}$

Let $s = v_l(q^{d} - 1)$. let $\epsilon$ be a primitive $l^s$-root of unity in $\F_{q^{2d}}$. Let $E$ be the image of $\epsilon$ in $GL_d(\F_{q})$.

For $n = 2m$, let 
{\scriptsize $$E_1 = \begin{pmatrix}
  E &   &        &   &            &   &        &\\
    & 1 &        &   &            &   &        &\\
    &   & \ddots &   &            &   &        &\\
    &   &        & 1 &            &   &        &\\
    &   &        &   & (\overline{E^{-1}})^T &   &        & \\
    &   &        &   &            & 1 &        & \\
    &   &        &   &            &   & \ddots & \\
    &   &        &   &            &   &        & 1
\end{pmatrix}, \ldots,$$
$$E_{\lfloor \frac{m}{d} \rfloor} = \begin{pmatrix}
1 &        &   &   &                     &   &        &   &            & \\
  & \ddots &   &   &                     &   &        &   &            & \\
  &        & 1 &   &                     &   &        &   &            & \\
  &        &   & E &                     &   &        &   &            & \\
  &        &   &   & \text{Id}_{m-\lfloor \frac{m}{d} \rfloor d}  &   &        &   &            & \\
  &        &   &   &                     & 1 &        &   &            & \\
  &        &   &   &                     &   & \ddots &   &            & \\
  &        &   &   &                     &   &        & 1 &            & \\
  &        &   &   &                     &   &        &   & (\overline{E^{-1}})^T & \\
  &        &   &   &                     &   &        &   &            & \text{Id}_{m-\lfloor \frac{m}{d} \rfloor d}
\end{pmatrix}$$}

For $n = 2m+1$, let 
$$E_1' = \begin{pmatrix} 1 & \mbf{0}\\
\mbf{0} & E_1 \end{pmatrix}, \ldots, E_{\lfloor \frac{m}{d} \rfloor} = \begin{pmatrix} 1 & \mbf{0}\\
\mbf{0} & E_{\lfloor \frac{m}{d} \rfloor} \end{pmatrix}$$

$P_l(S_{\lfloor \frac{m}{d} \rfloor})$ acts on $\langle E_1, \ldots, E_{\lfloor \frac{m}{d} \rfloor} \rangle \cong \langle E_1', \ldots, E_{\lfloor \frac{m}{d} \rfloor}' \rangle$. 
We can embed $P_l(S_{\lfloor \frac{m}{d} \rfloor})$ into $U(n,q^2)$. Let 
\begin{align*}
P &= \langle E_1, \ldots, E_{\lfloor \frac{m}{d} \rfloor} \rangle \rtimes P_l(S_{\lfloor \frac{m}{d} \rfloor})
\end{align*}
Then $P \in \syl_l(U(n,q^2))$, and $P$ is isomorphic to a Sylow $l$-subgroup of $GL_m(\F_{q^2})$, which is isomorphic to a Sylow $l$-subgroup of $GL_{\lfloor \frac{n}{2} \rfloor}(\F_{q^2})$.
\qedhere
\end{proof}

\subsection{The Special Unitary Groups}

\begin{theorem}\label{SUn} Let $p$ be a prime, $q = p^r$, and $l$ a prime with $l \neq 2,p$.  Let $k$ be a field with $\text{char } k \neq l$. Then
$$\ed_k(SU(n,q^2),l) = \begin{cases} \ed_k(U(n,q^2),l), &l \nmid q+1\\
\ed_k(SL_n(\F_{q^2}),l), &l \divides q + 1\end{cases}$$\end{theorem}

\begin{proof}
By Grove (\cite{Gr}, Theorem 11.28 and Corollary 11.29),
$$|SU(n,q^2)| = \frac{|U(n,q^2)|}{q+1}$$

If $l \nmid q + 1$, then the Sylow $l$-subgroups of $SU(n,q^2)$ are isomorphic to the Sylow $l$-subgroups of $U(n,q^2)$. So we need only prove the case when $l \divides q + 1$. Thus in this section, we will assume $l \divides q + 1$. Then since $l \neq 2$, this implies that $l \nmid q - 1$. Also, since $q^2 - 1 = (q+1)(q-1)$, we must have $l \divides q^2 - 1$. Let $d'$ be the smallest positive integer such that $l \divides q^d - 1$. Then $d' = 2$.  So by the work on the unitary groups, $|U(n,q^2)|_l = |GL_n(\F_{q^2})|_l$. Note that since $l \nmid q - 1$, $v_l(q^2-1) = v_l(q+1)$. 

Note that when finding the Sylow $l$-subgroup of $GL_n(\F_{q^2})$, we would have $d$ the smallest power of $q^2$ such that $l \divides (q^2)^{d} - 1$. So in this case, we would have $d = 1$. Then we would set $s = v_l((q^2)^{d} - 1) = v_l(q^2 -1) = v_l(q+1)$.  We would have $n_0 = \lfloor \frac{n}{d} \rfloor = \lfloor \frac{n}{1} \rfloor = n$.  Thus in the present case,
$$|GL_n(\F_{q^2})|_l = l^{sn} \cdot |S_n|_l.$$
So \begin{align*}
|SU(n,q^2)|_l &= \frac{|U(n,q^2)|_l}{l^{v_l(q+1)}} = \frac{|GL_n(\F_{q^2})|_l}{l^s} = \frac{l^{sn} \cdot |S_n|_l}{l^s} = l^{s(n-1)} \cdot |S_n|_l = |SL_n(\F_{q^2})|_l
\end{align*}

Recall that $SU(n,q^2) = \{M \in U(n,q^2) : \det(M) = 1\}$ and $SL_n(\F_{q^2}) = \{M \in GL_n(\F_{q^2}) : \det(M) = 1\}$. Therefore, since the Sylow $l$-subgroups of $U(n,q^2)$  and $GL_n(\F_{q^2})$ are isomorphic, we can conclude that the Sylow $l$-subgroups of $SU(n,q^2)$ and $SL_n(\F_{q^2})$ are isomorphic.
\qedhere
\end{proof}

\subsection{The Projective Special Unitary Groups}

\begin{theorem}\label{PSUn} Let $p$ be a prime, $q = p^r$, and $l$ a prime with $l \neq 2,p$.  Let $k$ be a field with $\text{char } k \neq l$. Then
$$\ed_k(PSU(n,q^2),l) = \begin{cases} \ed_k(SU(n,q^2),l), &l \nmid n \text{ or } l \nmid q+1\\
\ed_k(PSL_n(\F_{q^2}),l), &l \divides n, \text{ } l \divides q + 1\end{cases}$$
\end{theorem}

\begin{proof}
By Grove (Corollary 11.29),
$$|PSU(n,q^2)| = \frac{|SU(n,q^2)|}{(n,q+1)}.$$

If $l \nmid n$ or $l \nmid q + 1$, then the Sylow $l$-subgroups of $PSU(n,q^2)$ are isomorphic to the Sylow $l$-subgroups of $SU(n,q^2)$. So we need only prove the case when $l \divides n, \text{ }l \divides q + 1$. Thus in this section, we will assume $l \divides n, \text{ } l \divides q + 1$. As before, this implies that $l \nmid q - 1$, $l \divides q^2 - 1$, and $v_l(q^2-1) = v_l(q+1)$. 

By Grove (\cite{Gr}, Proposition 1.1),
$$|PSL_n(\F_{q^2})| = \frac{|SL_n(\F_{q^2})|}{(n,q^2-1)}.$$
Thus 
$$|PSL_n(\F_{q^2})|_l = \frac{|SL_n(\F_{q^2})|_l}{l^{\min(v_l(n),v_l(q^2-1)}} = \frac{|SL_n(\F_{q^2})|_l}{l^{\min(v_l(n),v_l(q+1)}}.$$
By the work for the special unitary groups, we know that $|SL_n(\F_{q^2})|_l = |SU(n,q^2)|_l$. So 
 \begin{align*}
|PSU(n,q^2)|_l &= \frac{|SU(n,q^2)|_l}{l^{min(v_l(n), v_l(q+1))}} = \frac{|SL_n(\F_{q^2})|_l}{l^{min(v_l(n), v_l(q+1))}} = |PSL_n(\F_{q^2})|_l
\end{align*}

Since $PSU(n,q^2)$ and $PSL_n(\F_{q^2}))$ are obtained from $SU(n,q^2)$ and $SL_n(\F_{q^2})$ respectively by modding out by a cyclic group of order $l^{\min(s,t)}$ and the Sylow $l$-subgroups of $SU(n,q^2)$ and $SL_n(\F_{q^2})$ are isomorphic, we can conclude that the Sylow $l$-subgroups of $PSU(n,q^2)$ and $PSL_n(\F_{q^2})$ are isomorphic.
\qedhere 
\end{proof}

\newpage

\section{Appendix}

In this appendix, we provide some details for the computations in this article.

\subsection{Proof of Lemma \ref{Pl(Sn)}}\label{App2}

\begin{lma}[\ref{Pl(Sn)}]
Let $\sigma_{i}^{j}$ be the permutation which permutes the $i$th set of $l$ blocks of size $l^{j-1}$. Then 
$$\langle \{\sigma_i^j\}_{1 \leq j \leq \mu_l(n), 1 \leq i \leq \lfloor \frac{n}{l^{j}} \rfloor} \rangle \in \syl_l(S_n).$$ Let $P_l(S_n)$ denote this particular Sylow $l$-subgroup of $S_n$.
\end{lma}

\begin{proof}\footnote{See \cite{MeyReich}, Corollary 4.2} Let $n' = \lfloor \frac{n}{l} \rfloor$, and let
$$\sigma_1^1 = (1, \cdots, l), \cdots, \sigma_{n'}^1 = ((n' - 1)l + 1, \cdots , n' l).$$

\noindent \textbf{Base Case:} If $n' = 1$, then $n = l + k$ for $k < l$.  Thus the only factor of $n!$ divisible by $l$ is $l$, so we have $|S_n|_l = l$, and $P_l(S_n) = (\Z/l\Z) \in \syl_l(S_n)$ (generated by $\sigma_1^1 = (1, \cdots, l)$).

\noindent \textbf{Induction Step:}

Let $D \cong (\Z/l\Z)^{n'}$.  Then $S_{n'}$ acts on $D$ by permuting the $\sigma_{i}^{1}$. And $D \rtimes S_{n'}$ embeds into $S_n$. Write $n = ln' + *$ for $* < l$; then 
\begin{align*}
v_l(n!) &= v_l((ln' + *)!)\\
 &= v_l((ln')!)\\
&= \sum_{i=1}^{ln'} v_l(i)\\
&= \sum_{i=1}^{n'} v_l(li)\\
&= \sum_{i=1}^{n'}  1 + \sum_{i=1}^{n'} v_l(i) \\
&= n' + v_l(n'!)\\
&= v_l(|D|) + v_l(|S_{n'}|)\\
&= v_l(|D \rtimes S_{n'}|)
\end{align*} Thus $D \rtimes S_{n'}$ embeds into $S_n$ with index prime to $l$.  Therefore, $P_l(S_n) \cong D \rtimes P_l(S_{n'}) \in \syl_l(S_n).$ 

Let $\mu_l(n)$ be the highest power of $l$ such that $\lfloor \frac{n}{l^{\mu_l(n)}} \rfloor > 0$. Let
{\small \begin{align*}
\sigma_1^2 &= (1,l+1,\cdots, l(l-1) + 1)\\
&\cdots\\
&\sigma_{\lfloor \frac{n}{l^2} \rfloor}^2 = (l^2(\lfloor \frac{n}{l^2} \rfloor - 1)+1, l^2(\lfloor \frac{n}{l^2} - 1) + l + 1), \cdots, l^2\lfloor \frac{n}{l^2}\rfloor - l + 1)\\
&\vdots\\
\sigma_1^{\mu_l(n)} &= (1,l^{\mu_l(n)-1}+1, \cdots, l^{\mu_l(n)-1}(l-1) + 1),\\\
&\cdots\\
 &\sigma_{\lfloor \frac{n}{l^{\mu_l(n)}} \rfloor}^{\mu_l(n)} = (l^{\mu_l(n)}(\lfloor \frac{n}{l^{\mu_l(n)}} \rfloor  - 1) + 1, (l^{\mu_l(n)}(\lfloor \frac{n}{l^{\mu_l(n)}} \rfloor - 1) + l^{\mu_l(n)-1} + 1,\\
 &\qquad \qquad \cdots, l^{\mu_l(n)}\lfloor \frac{n}{l^{\mu_l(n)}} \rfloor - l^{\mu_l(n) - 1}+1)
\end{align*}}
Then $P_l(S_n)$ is generated by $\{\sigma_{i}^{j}\}$. And for $j_0$ fixed $\{\sigma_{i}^{j_0}\}$ generates a subgroup of order $(\Z/l\Z)^{\lfloor \frac{n}{l^{j_0}} \rfloor}$. $\sigma_{i}^{j}$ permutes the $i$th set of $l$ blocks of size $l^{j-1}$.
\qedhere
\end{proof}

\subsection{Remark \ref{remark1}}\label{App1} 

Duncan and Reichstein calculated the essential $p$-dimension of the pseudo-reflection groups: For $G$ a pseudo-reflection group with $k[V]^G = k[f_1,\cdots,f_n]$, $d_i = \text{deg}(f_i),$ $\ed_k(G,p) = a(p) = |\{i : d_i \text{ is divisible by } p\}|$ (\cite{DR}, Theorem 1.1). These groups overlap with the groups above in a few small cases (The values of $d_i$ are in \cite{ST}, Table VII):  
\begin{enumerate}[(i)]
\item Group 23 in the Shephard-Todd classification, $W(H_3) \cong \Z/2\Z \times PSL_2(\F_5)$: $d_1, \ldots, d_3$ are $2,6,10$; so 
$$\ed_k(W(H_3),3) = 1 = \ed_k(PSL_2(\F_5),3).$$

\item Group 24 in the Shephard-Todd classification, $W(J_3(4)) \cong \Z/2\Z \times PSL_2(5)$: $d_1, \ldots, d_3$ are $4,6,14$; so
$$\ed_k(W(J_3(4)),3) = 1 = \ed_k(PSL_2(5),3)$$
and 
$$\ed_k(W(J_3(4)),7) = 1 = \ed_k(PSL_2(5),7).$$

\item Group 32 in the Shephard-Todd classification, $W(L_4) \cong \Z/3\Z \times Sp(4,3)$: $d_1, \ldots, d_4$ are $12,18,24,30$; so 
$$\ed_k(W(L_4),5) = 1 = \ed_k(Sp(4,3),5).$$

\item Group 33 in the Shephard-Todd classification, $W(K_5) \cong \Z/2\Z \times PSp(4,3) \cong \Z/2\Z \times PSU(4,2^2)$: $d_1, \ldots, d_5$ are $4, 6, 10, 12, 18$; so 
$$\ed_k(W(K_5),5) = 1 = \ed_k(PSp(4,3),5) = \ed_k(PSU(4,2^2),5)$$
and 
$$\ed_k(W(K_5),3) = 3 = \ed_k(PSU(4,2^2),3).$$

\item Group 35 in the Shephard-Todd classification, $W(E_6) \cong O^-(6,2)$: 
$d_1, \ldots, d_6$ are $2,5,6,8,9,12$; so 
$$\ed_k(W(E_6),5) = 1 = \ed_k(O^-(6,2),5).$$
and 
$$\ed_k(W(E_6),3) = 3 = \ed_k(O^-(6,2),3).$$
%

\item Group 36 in the Shephard-Todd classification, $W(E_7) \cong \Z/2\Z \times Sp(6,2)$:\\
$d_1, \ldots, d_7$ are $2,6,8,10,12,14,18$; so 
$$\ed_k(W(E_7),5) = 1 = \ed_k(Sp(6,2),5),$$
$$\ed_k(W(E_7),3) = 3 = \ed_k(Sp(6,2),3),$$
and
$$\ed_k(W(E_7),7) = 1 = \ed_k(Sp(6,2),7).$$

\item Group 37 in the Shephard-Todd classification, $W(E_8)$ contains $O^+(8,2)$ as in index 2 subgroup: $d_1, \ldots, d_8$ are $2,8,12,14,18,20,24,$ $30$; so
$$\ed_k(W(E_8),3) = 4 = \ed_k(O^+(8,2),3),$$
$$\ed_k(W(E_8),5) = 2 = \ed_k(O^+(8,2),5),$$
and
$$\ed_k(W(E_8),7) = 1 = \ed_k(O^+(8,2),3).$$

\end{enumerate}

\bibliographystyle{plain}
\bibliography{references}

\end{document}